\documentclass[a4paper]{report}

\usepackage{url}
\usepackage{latexsym}
\usepackage{bussproofs}
\usepackage{amsmath, amssymb}
\usepackage{stmaryrd}
\usepackage{mathtools}
\usepackage[all]{xy}

\title{Fibrations of Predicates and Bicategories of Relations}
\author{
  Finn Lawler\\
  Department of Computer Science\\
  Trinity College, Dublin
}
\date{February 2015}

\usepackage{amsthm}

\theoremstyle{remark}

\swapnumbers
\theoremstyle{plain}
\newtheorem{thm}{Theorem}[section]
\newtheorem{prp}[thm]{Proposition}
\newtheorem{lem}[thm]{Lemma}
\newtheorem{cor}[thm]{Corollary}
\theoremstyle{definition}
\newtheorem{dfn}[thm]{Definition}
\newtheorem{rem}[thm]{Remark}

\numberwithin{equation}{section}

\newcommand{\pwrset}{{\mathcal{P}}}
\newcommand{\nullset}{\varnothing}
\newcommand{\set}[2]{\{#1\,|\, #2\}}

\newcommand{\PN}{\pwrset\mathbb{N}}
\newcommand{\Eff}{\cat{Eff}}

\newcommand{\Iff}{if\mbox{}f}

\newcommand{\den}[1]{\llbracket #1 \rrbracket}

\newcommand{\Rel}{\bicat{Rel}}
\DeclareMathOperator{\Map}{Map}

\DeclareMathOperator{\Matr}{Matr}
\DeclareMathOperator{\Pred}{Pred}
\DeclareMathOperator{\colim}{colim}

\DeclareMathOperator{\img}{im}
\newcommand{\im}{\img}
\DeclareMathOperator{\Nat}{Nat}
\DeclareMathOperator{\ob}{ob}

\newcommand{\Span}{\bicat{Span}}

\newcommand{\term}{\mathbf{1}}

\newcommand{\coop}{\mathrm{co\,op}}
\newcommand{\op}{\mathrm{op}}
\newcommand{\co}{\mathrm{co}}

\newcommand{\pbcorner}[1][dr]{\save*!/#1-1.2pc/#1:(-1,1)@^{|-}\restore}

\newdir{ (}{{}*!/-5pt/@{(}}
\newdir^{ (}{{}*!/-5pt/@^{(}}

\makeatletter
\def\slashedarrowfill@#1#2#3#4#5{%
  $\m@th\thickmuskip0mu\medmuskip\thickmuskip\thinmuskip\thickmuskip
   \relax#5#1\mkern-7mu%
   \cleaders\hbox{$#5\mkern-2mu#2\mkern-2mu$}\hfill
   \mathclap{#3}\mathclap{#2}%
   \cleaders\hbox{$#5\mkern-2mu#2\mkern-2mu$}\hfill
   \mkern-7mu#4$%
}
\def\rightslashedarrowfill@{%
  \slashedarrowfill@\relbar\relbar\mapstochar\rightarrow}
\newcommand\xslashedrightarrow[2][]{%
  \ext@arrow 0055{\rightslashedarrowfill@}{#1}{#2}}
\makeatother

\newcommand{\prof}{\xslashedrightarrow{}}

\usepackage{rotating}

\EnableBpAbbreviations{}

\newcommand{\sqar}{\Longrightarrow}
\newcommand{\sq}[1][]{\overset{#1}{\sqar}}

\newcommand{\rel}{\looparrowright}
\newcommand{\map}{\dashrightarrow}
\newcommand{\lmap}{\dashleftarrow}

\newcommand{\tricat}[1]{\mathrm{\mathbb #1}}

\newcommand{\Bicat}{2\mbox{-}\tricat{Cat}}
\newcommand{\Biprof}{2\mbox{-}\tricat{Prof}}

\newcommand{\cat}[1]{\mathbf{#1}}

\newcommand{\xcat}[1]{\mbox{$#1$-$\Cat$}}
\newcommand{\xalg}[1]{\mbox{$#1$-$\cat{Alg}$}}

\newcommand{\iso}{\cong}

\newcommand{\cmp}{\circ}
\newcommand{\inv}[1]{{#1}^{-1}}

\newcommand{\Gray}{\cat{Gray}}
\newcommand{\Set}{\cat{Set}}
\newcommand{\C}{\cat{C}}
\newcommand{\D}{\cat{D}}
\newcommand{\E}{\cat{E}}
\newcommand{\B}{\cat{B}}
\newcommand{\Dsc}{\cat{Dsc}}

\newcommand{\bicat}[1]{\mathit{\mathcal #1}}
\newcommand{\Cat}{\bicat{Cat}}
\newcommand{\bA}{\bicat{A}}
\newcommand{\bB}{\bicat{B}}
\newcommand{\bK}{\bicat{K}}
\newcommand{\bL}{\bicat{L}}
\newcommand{\bJ}{\bicat{J}}

\newcommand{\bM}{\bicat{M}}
\newcommand{\bN}{\bicat{N}}
\newcommand{\bEqt}{\bicat{Eqt}}

\newcommand{\Pw}{\mathbb{P}}

\newcommand{\tc}{\Rightarrow}
\newcommand{\nt}{\tc}
\newcommand{\mdf}{\Rrightarrow}
\newcommand{\xnt}{\mathrel{\ddot{\tc}}}
\newcommand{\eqv}{\simeq}
\newcommand{\mon}{\hookrightarrow}
\newcommand{\sub}{\subseteq}
\newcommand{\tup}[1]{\langle #1 \rangle}

\newcommand{\vcmp}{\mathop{.}}
\newcommand{\hcmp}{\circ}

\newcommand{\elt}{{\textstyle \int}}

\begin{document}
\pagenumbering{roman}
\maketitle{}

\section*{Note to arXiv version}

This is a slightly revised version of my Ph.D. thesis, which was
originally submitted in November 2013.  The viva was held the
following May, with Matthew Hennessy as chair and Martin Hyland and
John Power as external examiners.  The thesis was accepted with no
corrections required, and the final hardbound version was submitted,
with a few minor revisions made, in September 2014.  For this version
I have again made some small corrections and added a few notes, but
the document remains the same in substance.

\newpage

\section*{Declaration}

I hereby declare:
\begin{itemize}
\item that this work has not been submitted as an exercise for a
  degree at this or any other University;
\item that it is, except where indicated, entirely my own work;
\item and that I agree that the Library may lend or copy it upon
  request.
\end{itemize}

\vskip 5em
\noindent
Finn Lawler

\newpage

\section*{Summary}

We reconcile the two different category-theoretic semantics of regular
theories in predicate logic.  A 2-category of \emph{regular
  fibrations} is constructed, as well as a 2-category of \emph{regular
  proarrow equipments}, and it is shown that the two are equivalent.
A regular equipment is a \emph{cartesian equipment} satisfying certain
axioms, and a cartesian equipment is a slight generalization of a
cartesian bicategory.

This is done by defining a tricategory $\Biprof$ whose objects are
bicategories and whose morphisms are category-valued profunctors, and
then defining an \emph{equipment} to be a pseudo-monad in this
tricategory.  The resulting notion of equipment is compared to several
existing ones.  Most importantly, this involves showing that every
pseudo-monad in $\Biprof$ has a \emph{Kleisli object}.  A strict
2-category of equipments, over locally discrete base bicategories, is
identified, and cartesian equipments are defined to be the cartesian
objects in this 2-category.  Thus cartesian equipments themselves form
a 2-category, and this is shown to admit a 2-fully-faithful functor
from the 2-category of regular fibrations.  The cartesian equipments
in the image of this functor are characterized as those satisfying
certain axioms, and hence a 2-category of \emph{regular equipments} is
identified that is equivalent to that of regular fibrations.

It is then shown that a regular fibration admits \emph{comprehension}
for predicates if and only if its corresponding regular equipment
admits \emph{tabulation} for morphisms, and further that the presence
of tabulations for morphisms is equivalent to the existence of
\emph{Eilenberg--Moore objects} for co-monads.  We conclude with a
brief examination of the two different constructions of the effective
topos, via triposes and via assemblies, in the light of the foregoing.

\newpage

\section*{Acknowledgements}

My thanks must go first of all to my supervisor, Hugh Gibbons, for his
unfailing patience and support.  When Hugh took me on, my intention
was to write a thesis on logic and computation, so he must have been
somewhat dismayed to watch me then veer sharply into category theory.
But he did not despair, and in fact this meant that I had to make my
ideas and explanations intelligible to a sympathetic but lay audience,
forcing me to think more deeply and expound more clearly than I might
otherwise have done.  I hope that shows.  Hugh's help in navigating
the College bureaucracy, and indeed all the usual travails of
post-graduate life, has been likewise invaluable.

If Hugh has been my mainstay, then Arthur Hughes has been my binnacle.
The two have together supported and guided me through the whole
process of planning and writing this thesis.  Though his status as my
co-supervisor remains unofficial, Arthur has more than earned that
acknowledgement by his willingness to listen, suggest and advise,
particularly when it comes to category theory.  His good nature and
solicitousness make him a pleasure to work with too.  I should also
thank the other members of the Foundations and Methods Group of the
Computer Science Department for their help and advice, especially
Matthew Hennessy.

Even despite Arthur's mathematical company, though, I would not have
come to understand category theory at anywhere near the level this
thesis required without the $n$-Lab (\url{ncatlab.org}).  Of course, I
can't possibly thank by name everyone involved, but I do want to make
clear just how much I owe to all of them.  My understanding of
category theory, and of mathematics in general, has been immeasurably
deepened and broadened by the material on the Lab, as well as by
discussions on the associated $n$-Forum.  I must thank Mike Shulman
and Todd Trimble in particular for the latter.

Writing a doctoral thesis is never easy, and this one too had a
difficult birth.  But when things looked grim, the College support
services were there to help.  I really cannot overstate how much it
has meant to me to have their professional, compassionate and
effective support available, or how important these services are in
general.  It reflects very well indeed on College that they employ so
many staff whose sole concern is the well-being of students, and long
may this continue.  I don't suppose it would be appropriate to mention
the people I have in mind here by name, but if I say that they are
C.~G., C.~R. and Dr.~N.~F., then they will know who they are, and know
that they have my immense gratitude.  The system works, and this
thesis is the proof.  Likewise, the College administrative staff, in
my dealings with them, have never failed to be friendly, efficient,
flexible and eager to help.  In particular, Helen Thornbury of the
Graduate Studies Office is a sparkling exemplar of administrative
excellence, who has personally dug me out of several holes.

In the end, though, it is beyond certain that none of this would have
happened if it wasn't for the continual and unconditional love and
support of my parents.  Never mind about mainstays and binnacles ---
they have been my hull, my rudder and my sails.  I dedicate this work
to them.

\newpage

\tableofcontents{}

\newpage

\pagenumbering{arabic}

\chapter{Introduction}
\label{cha:introduction}

\section{Background}
\label{sec:idea}

This work is intended primarily as a contribution to the
category-theoretic understanding of predicate logic, with an eye to
clarifying the relationship between the two different constructions of
realizability toposes.  The following section gives more details on
the motivation behind this work, the next explains its development and
major results, and the last gives a detailed outline of the remaining
chapters.

\subsection{Motivation}
\label{sec:motivation}

For our purposes, a \emph{logic} specifies, given a collection of
types, and terms that map from one type to another, and of predicates,
each of which lives over some type, and derivations or proofs that map
from one predicate to another, a set of admissible ways to build new
types, terms, predicates and derivations from existing ones.  A theory
$T$ over a logic is then given by a collection of basic types and
terms and of (equational) \emph{axioms} (equations between terms), and
a collection of basic predicates and derivations and of
(propositional) axioms (equations between derivations).
Traditionally, one did not distinguish between different derivations
of the same entailment, so that a collection of derivations and
propositional axioms is determined by a collection of statements that
one predicate entails another.  But we will take the view that it is
useful to keep different proofs distinct --- one might say that we are
doing type theory, rather than logic as traditionally understood.

Category theory formalizes this situation in one of the following two
ways (see
e.g.~\cite{lawvere69:_adjoin,jacobs99:_categ_logic_and_type_theor} and
\cite{freyd90:_categ_alleg,carboni87:_cartes_bicat_i} respectively):
\begin{enumerate}
\item\label{item:7} The types and terms form a category $\cat{B}_T$,
  with equality on its morphisms generated by the equational axioms of
  $T$.  The predicates over each type $X$ form a category
  $\cat{E}_T(X)$, whose morphisms $P \to Q$ are given by (proofs of)
  entailments $P(x) \vdash Q(x)$, and the propositional axioms furnish
  an equality relation on these.  The terms $t \colon X \to Y$ act on
  these categories by substitution, so as to make $\cat{E}_T(-)$ a
  pseudo-functor $\cat{B}_T^{\op} \to \Cat$, or a fibration over
  $\cat{B}_T$.
\item\label{item:8} A bicategory $\operatorname{Rel}(T)$ is formed,
  whose objects are the types, and in which a morphism $X \rel Y$ is a
  \emph{relation} from $X$ to $Y$, that is, a predicate on $X \times
  Y$.  The composite of $R(x,y)$ and $S(y,z)$ is the relation $\exists
  y. R(x,y) \wedge S(y,z)$, and the 2-cells are morphisms of
  predicates as above.  Each term $t \colon X \to Y$ gives rise to a
  relation $t_\bullet \colon X \rel Y $, given by $t x = y$ and called
  the \emph{graph} of $t$.  The equational axioms of $T$ determine
  propositional equations between graphs.
\end{enumerate}
Notice that the first (\emph{fibrational}) approach requires very
little structure to be present in the theory $T$.  On the other hand,
the second (\emph{relational} or \emph{bicategorical}) approach
requires that (the logic underlying) $T$ have at least finite
conjunctions and the existential quantifier; that is, that $T$ be a
\emph{regular theory}.

By the usual `yoga' (to use Grothendieck's term) of categorical logic,
syntactic models such as these carry structure determined by the logic
underlying the theory $T$; more general models are structures of the
same kind, and an interpretation of the theory in a model is a
homomorphism.  In the above two cases, the most common kinds of
`model' are the subobject fibrations and bicategories of relations of
regular categories.  But these two kinds of structure also arise in
the two distinct recipes for constructing \emph{realizability
  toposes}: one approach \cite{hyland82:_effec_topos} goes via
fibrations, and the other
\cite{carboni88:_categ_approac_to_realiz_and_polym_types} via
bicategories.  The initial motivation for the research described here
was to understand the relationship between these two constructions.

\section{Outline}
\label{sec:outline}

\subsection{Development and results}
\label{sec:results}

We will show that the two ways given above of describing regular
theories and their models are equivalent.  That is, there is a kind of
fibration called a \emph{regular fibration}, and a kind of bicategory,
or rather proarrow equipment \cite{wood82:proarrows_i}, that we call a
\emph{regular equipment}, and the bicategories of which these are the
objects are equivalent.  In particular, the syntactic examples above
correspond to each other, and we describe also how the two
constructions of the effective topos fit into this framework.

A regular fibration is a bifibration with fibred finite products,
satisfying the Frobenius condition and the Beck--Chevalley conditions
for certain (\emph{product-absolute}) pullback squares.  Structures
like these have been studied before, although except for in
\cite{pavlovic96:_maps_ii} this has usually been restricted to those
fibrations whose fibres are preorders.  In the syntactic case
described above, these are the term models that record only the
existence of a proof of one proposition from another.  Our results
apply in full generality.

Similarly, the locally preordered versions of the bicategories
described in (\ref{item:8}) above are well known as \emph{allegories}
\cite{freyd90:_categ_alleg}.  The allegories that arise in the
`regular' context carry certain extra structure, making them
\emph{unitary} and \emph{pre-tabular}.  We show that such allegories
are the same thing as \emph{bicategories of relations}
\cite{carboni87:_cartes_bicat_i}.  These are locally ordered
\emph{cartesian bicategories} \cite{carboni08:_cartes_ii} satisfying
some extra axioms.

In order to construct an equivalence between regular fibrations and
cartesian bicategories, it is necessary to equip the latter with
distinguished subcategories of morphisms with right adjoints, making
them into proarrow equipments.  Intuitively, this lets a cartesian
bicategory remember the difference, which fibrations account for,
between \emph{functions} or \emph{terms} on the one hand, and
\emph{functional relations} on the other.  So we are looking for a
notion of \emph{cartesian equipment}.

There are several definitions of equipments in the literature, namely
Wood's original one \cite{wood82:proarrows_i}, Shulman's \emph{framed
  bicategories} \cite{shulman08:_framed_bicat_and_monoid_fibrat}, and
the (strictly more general) equipments of Carboni
et.~al.~\cite{carboni98:change_base_geom_ii}.  We give an abstract
definition, involving the tricategory whose objects are bicategories
and whose morphisms are category-valued profunctors, that subsumes
those of Wood and of Shulman, whose relation to that of Carboni
et.~al.~is clear, and that is quite similar to Verity's notion of
\emph{double bicategory} \cite{verity92:_enric}.  We also show that
our equipments form a category that is equivalent to the ordinary
category underlying Shulman's strict 2-category of framed
bicategories, and so we may take 2-cells between equipment-morphisms
to be transformations between the associated framed functors, yielding
a 2-category of equipments.

A cartesian equipment is then defined to be a cartesian object in this
last 2-category, and a \emph{regular equipment} to be a cartesian one
satisfying some well-known axioms; we show that cartesian bicategories
are a special case of cartesian equipments, and that the 2-category of
regular equipments is equivalent to that of regular fibrations, as
expected.  We can then show that (suitable notions of)
\emph{comprehension} in a fibration and \emph{tabulation} in an
equipment correspond to each other, and that completion with respect
to these is equivalent, in the preordered case, to one of the steps in
the construction of the effective topos.

\subsection{Detailed outline}
\label{sec:detailed-outline}

Chapter~\ref{cha:categ-fibr-alleg} begins with basic background
definitions, before going on to describe the syntax of regular logic
and its semantics in regular fibrations.  Comprehension in regular
fibrations is also discussed.  Section~\ref{sec:class-fibr-regul}
shows that a regular theory gives rise to a `syntactic' or classifying
regular fibration, a result that we will not make essential further
use of but that it is worth giving in the context of
section~\ref{sec:regul-fibr-regul} as a whole.  The next section
defines allegories and the structures on them that we want, and
describes idempotents and the construction of the universal allegory
in which a class of them splits.  The last section of the chapter
defines bicategories of relations and proves that they are equivalent
to unitary pre-tabular allegories.

Chapter~\ref{cha:2-categories} introduces more new ideas and results
than the preceding one.  It starts with definitions of adjunctions and
mates and of monads and modules in a bicategory.  This material is of
course very well known, but we present the theory of monads and
modules in what seems to be a somewhat original way.
Section~\ref{sec:2-categories} concludes with definitions of monoidal
bicategories and pseudo-monads, which will be used in
chapter~\ref{cha:equipments}.  As mentioned above, we want to define
equipments to be pseudo-monads in the tricategory of bicategories and
category-valued profunctors; in order to define this tricategory we
mimic the definition of the usual bicategory of profunctors as
consisting of presheaf categories and cocontinuous functors.  So we
spend section~\ref{sec:limits-colimits}, the remainder of
chapter~\ref{cha:2-categories}, defining and exploring the properties
of bicategorical colimits.  In particular, our description of
2-dimensional (co)ends appears to be new, as do the results of
section~\ref{sec:computing-colimits} on computing bicategorical
colimits in $\Cat$ (but see the footnote to prop.~\ref{prp:25}).

Chapter~\ref{cha:equipments} is the core of this work.  In it we
define the tricategory $\Biprof$ as promised, and show that it admits
the construction of Kleisli objects for pseudo-monads.  This is what
enables us to go on and show, in section~\ref{sec:defin-equipm}, that
to give a pseudo-monad in $\Biprof$, satisfying certain properties, is
precisely to give a proarrow equipment in the sense of Wood
\cite{wood82:proarrows_i}.  The remainder of that section compares our
notion of equipment to Shulman's notion
\cite{shulman08:_framed_bicat_and_monoid_fibrat} of framed bicategory,
showing that together with their morphisms (equipment-morphisms having
been defined) they make up equivalent categories.  Even though our
abstract approach to equipments via pseudo-monads works well for 0-
and 1-cells, it does not quite go through when it comes to 2-cells.
Section~\ref{sec:equipments} discusses how we might rectify this, and
it is certainly work that ought to be done, but for our purposes here
we can get away with simply defining equipment 2-cells to be
transformations between corresponding functors between framed
bicategories.

Section~\ref{sec:equipm-fibr} is where our earlier work begins to bear
fruit.  Section~\ref{sec:cartesian-equipments} defines what it is for
an equipment to be \emph{cartesian}, and gives equivalent descriptions
of this structure in both equipments and framed bicategories.  In
section~\ref{sec:comp-with-regul} it is shown that Shulman's
construction
\cite[theorem~14.2]{shulman08:_framed_bicat_and_monoid_fibrat} of a
monoidal equipment from a regular fibration extends to a functor from
the bicategory of regular fibrations to that of cartesian equipments,
and further that this functor is fully faithful.  The construction of
a would-be right inverse to its action on objects shows that a regular
fibration will only result if two additional axioms are assumed to
hold in a given cartesian bicategory.  One of these is well-known, and
the second is a Beck--Chevalley-type condition that automatically
holds in the locally ordered case when a simpler \emph{Frobenius}
axiom holds, as well as in the cases of bicategories of spans and of
relations, which may explain why it has not previously been considered
in the bicategorical context.  With this done, we have an equivalence
of bicategories between regular fibrations and these \emph{regular
  equipments}.  The last part of section~\ref{sec:equipm-fibr}
compares comprehension in regular fibrations to tabulation in regular
equipments, showing that they are equivalent modulo the equivalence of
bicategories just noted.  The existence of tabulation is also shown to
be equivalent to the existence of \emph{Eilenberg--Moore objects} for
co-monads.  Chapter~\ref{cha:equipments} ends with an application to
the original motivation for our work: a discussion of the effective
topos and the relationship between its two constructions, through the
lens what we have already done.

Finally, chapter~\ref{cha:concl-future-work} reviews the results of
the preceding three chapters, noting some links with existing work.
We conclude with some prospects for future work, and some ideas on how
to go about doing it: further elaboration of the abstract approach to
equipments in section~\ref{sec:biprofunctors}, and an attempt to
generalize the equipment side of the correspondence we have
established in order to go beyond the regular context.  There is also
reason to hope that the latter may help to connect our work with some
other abstract approaches to realizability.


\chapter{Categories, fibrations and allegories}
\label{cha:categ-fibr-alleg}

This chapter serves as background on the structures that will be used
in those to come.  After giving some very basic definitions, we define
what is meant by \emph{regular logic}, and then discuss the fibrations
in which regular theories find their models, namely \emph{regular
  fibrations}.  We show that any regular theory gives rise to a
syntactic model.  Then the definition of \emph{allegory} is recalled
and the splitting of idempotents described, material that will be used
later to connect our work with one of the constructions of the
effective topos.  Because the structures we will go on to use are a
slightly generalized version of cartesian bicategories, we show that
certain locally ordered cartesian bicategories, namely
\emph{bicategories of relations}, are the same as certain allegories,
namely the \emph{unitary pre-tabular} ones.

It is assumed that the reader is familiar with elementary category
theory, as expounded in e.g.~\cite{mac98:_categ_workin_mathem}, as
well as the theory of enriched categories, for which see
e.g.~\cite{kelly82}, and with `formal category theory', i.e.~those
parts of ordinary category theory, such as the theory of adjunctions,
monads and Kan extensions, that can be replicated in 2-categories
other than $\Cat$.

Everything we talk about will be assumed to be `weak' or `pseudo' by
default --- if something is strict or lax we will say so.  A
`2-category' is therefore a bicategory, a `functor' is a
pseudofunctor, and so on.  On the other hand, we will make broad use
of coherence and strictification theorems in order to simplify
definitions and calculations.  For example, monoidal categories and
bicategories will be (mostly) silently assumed to have been
strictified.

Ordinary (possibly monoidal) categories are written in bold face:
$\cat{Cat}$, $\cat{Gray}$, 2-categories in `calligraphic': $\bK$,
$\Cat$, and 3-categories with `blackboard bold': $\Bicat$, $\Biprof$.
Transformations and other 2-cells are written with a double arrow:
$\alpha \colon F \tc G$, extranaturals
(section~\ref{sec:2-extranaturality}) with a dotted arrow $\xnt$.
Modifications and other 3-cells are written with a triple arrow: $m
\colon \alpha \Rrightarrow \beta$.

Identities are called $1$ and terminal objects are called
$\mathbf{1}$.

\section{Regular fibrations and regular logic}
\label{sec:regul-fibr-regul}

\subsection{Basic definitions}
\label{sec:basic-definitions}

We give some elementary definitions in order to fix terminology and
notation.

\begin{dfn}
  The \emph{image} of a morphism $f \colon A \to B$ is a factorisation
  \begin{equation*}
    f = A \overset{e}{\to} \operatorname{im}(f) \overset{m}{\mon} B
  \end{equation*}
  in which $m$ is a monomorphism, and such that in any other such
  factorisation $f = m' e'$, $m' \sub m$ as subobjects of $B$.
\end{dfn}

\begin{dfn}
  \label{dfn:14}
  A \emph{regular category} is a category with finite limits in which
  every morphism has an image, and in which images are
  pullback-stable; that is, if $f \colon A \to B$ and $g \colon C \to
  B$, then $g^* \operatorname{im}(f) \iso \operatorname{im}(g^* f)$.
\end{dfn}

We assume familiarity with the notions of fibrations and of indexed
categories, and of the equivalence between the two.  In fact, we will
rarely distinguish between them, and will mostly use the term
`fibration' to denote either concept.  A \emph{bifibration} is of
course a functor that is both a fibration and an opfibration.  We will
write $f^*$ etc.~for the pullback functors of fibrations and either
$f_!$ or $\exists_f$ for the pushforwards of opfibrations.

The 2-category $\bicat{Fib}$ of fibrations can then be thought of as
the `2-category of elements' (def.~\ref{dfn:27}) of either of two
equivalent functors
\begin{equation*}
  \bicat{Fib}(-) \sim [-^\op, \Cat] \colon \Cat^\coop \to \Bicat
\end{equation*}

\begin{dfn}
  \label{dfn:28}
  The 2-category $\bicat{Fib}$ is defined as follows:
  \begin{itemize}
  \item an object is a pair of a category $\B$ and a fibration $\E$
    over $\B$;
  \item a morphism $(\B_1,\E_1) \to (\B_2,\E_2)$ is a functor $F
    \colon \B_1 \to \B_2$ and a morphism of fibrations $\phi \colon
    \E_1 \to F^*\E_2$, i.e.~either a natural transformation $\phi_X
    \colon \E_1(X) \to \E_2(F X)$ or a cartesian-morphism-preserving
    functor $\E_1 \to \E_2$ between total categories that fits into a
    commuting square
    \begin{equation*}
      \xymatrix{
        \E_1 \ar[r]^\phi \ar[d] & \E_2 \ar[d] \\
        \B_1 \ar[r]_F & \B_2
      }
    \end{equation*}
  \item a 2-cell $(F,\phi) \to (G,\gamma)$ is a transformation $\alpha
    \colon F \tc G$ such that
    \begin{equation*}
      \xymatrix{
        \E_1 \ar[dr]_\gamma \ar[r]^\phi & F^*\E_2
        \ar[d]^{\alpha^*\E_2} \\
        & G^*\E_2
      }
    \end{equation*}
    commutes.
  \end{itemize}
  $\bicat{Fib}$ then has a locally full sub-2-category $\bicat{BiFib}$
  consisting of bifibrations, opcartesian-morphism-preserving
  fibration morphisms and all fibration 2-cells.
\end{dfn}

\begin{dfn}
  A \emph{monoidal (bi)fibration} is given by a pair of monoidal
  categories together with a functor between them that is both
  (strong) monoidal and a (bi)fibration.
\end{dfn}

\begin{prp}[{\cite[theorem~12.7]{shulman08:_framed_bicat_and_monoid_fibrat}}]
  If $\B$ is a cartesian monoidal category, then the category of
  monoidal fibrations over $\B$ is equivalent (via the usual
  Grothendieck construction and its inverse) to the category of
  (pseudo)functors from $\B^\op$ to the 2-category of monoidal
  categories.
\end{prp}

\begin{dfn}[{\cite[2.8]{street81:_consp_of_variab_categ}}]
  \label{dfn:30}
  Let $\C$ and $\B$ be categories.  A \emph{two-sided fibration} from
  $\B$ to $\C$ is given by a span $(p,q) \colon \E \to \C \times \B$
  such that
  \begin{itemize}
  \item $p$ is a fibration whose chosen cartesian lifts are
    $q$-vertical (i.e.~they are inverted by $q$);
  \item $q$ is an opfibration whose chosen opcartesian lifts are
    $p$-vertical;
  \item for any composable cartesian-opcartesian pair $i^* x \to x \to
    j_! x$ in $\E$, the canonical morphism $j_!i^* x \to i^* j_! x$ is
    invertible.
  \end{itemize}
  We will say that a two-sided fibration the opposite of whose
  underlying span is also such is a \emph{two-sided bifibration}.
\end{dfn}

\begin{dfn}
  \label{dfn:5}
  An adjoint pair $F \dashv G$ of \emph{colax} monoidal functors
  between symmetric monoidal categories satisfies \emph{Frobenius
    reciprocity} \cite{lawvere70:_equal} if the canonical morphism
  \begin{equation*}
    \xymatrix{
      F(A \otimes G B) \ar[r] &
      F A \otimes F G B \ar[r] &
      F A \otimes B
    }
  \end{equation*}
  is invertible.  (Such an adjoint pair is also called a \emph{Hopf
    adjunction} \cite{bruguieres11:_hopf}).
\end{dfn}

\subsection{Regular logic}
\label{sec:regular-logic}

Regular logic is the fragment of first-order predicate logic that uses
only the connectives $\top$ for truth, $\land$ for conjunction and
$\exists$ for existential quantification.  We will mostly follow
\cite{seely83:_hyper_natur_deduc_and_beck_condit}.

\begin{dfn}
  A (regular) \emph{signature} $S$ is given by a collection
  $X,Y,\ldots$ of sorts, together with a collection of typed predicate
  and function symbols.  A \emph{type} is a finite sequence
  $X_1,X_2,\ldots$ of sorts, and types will also be denoted
  $X,Y,\ldots$.  If $P$ is a predicate of type $X$ we may write $P
  \colon X$, and similarly $f\colon X \to Y$ indicates the type of
  $f$.  Every signature contains at least the equality predicate
  ${=_X} \colon X,X$.
\end{dfn}

We assume given an inexhaustible supply of free variables
$x,x',y,y'\ldots$ and bound variables
$\xi,\xi',\upsilon,\upsilon'\ldots$ of each sort, with the notation
extended to types so that a variable of type $X,Y$ is the same as a
pair $x,y$ of variables of sorts $X$ and $Y$.

\begin{dfn}
  A \emph{context} is a finite list $x \colon X, y \colon Y,\ldots$ of
  sorted variables, or equivalently a single variable $z \colon
  X,Y,\ldots$.  A \emph{term} is either a variable, a tuple of terms
  or a function symbol $f$ applied to a term, all with the obvious
  well-typedness constraints.  Every term lives in a context, which is
  assumed to contain every variable in the term, perhaps together with
  `dummy' variables that don't.  We write $t[x]$ to indicate that $x$
  is the context of $t$, and $t[s]$ to denote the substitution of the
  term $s$ for the variable(s) $x$ in $t$.
\end{dfn}

\begin{dfn}
  A (regular) \emph{formula} is either the constant $\top$, a
  predicate symbol $P(t)$ applied to a term, the conjunction $\phi
  \land \psi$ of two formulas, a quantified formula $\exists \xi.\phi$
  or the substitution $\phi[t]$ of the term $t$ into the formula
  $\phi$, defined in the usual way.  Every formula lives in a context,
  which we assume contains (perhaps strictly) all of its free
  variables, and we write $\phi[x]$ for this.
\end{dfn}

\begin{dfn}
  The inference rules of regular logic are as follows: conjunction is
  governed by
  \begin{prooftree}
    \AXC{$\phi$} \AXC{$\psi$}
    \BIC{$\phi \land \psi$}
    \DisplayProof
    \qquad\qquad
    \AXC{$\phi \land \psi$}
    \UIC{$\phi$}
    \DisplayProof
    \qquad
    \AXC{$\phi \land \psi$}
    \UIC{$\psi$}
  \end{prooftree}
  truth by
  \begin{prooftree}
    \AXC{$\phi$}
    \UIC{$\top$}
  \end{prooftree}
  existentials by
  \begin{prooftree}
    \AXC{$\phi[t]$}
    \UIC{$\exists \xi.\phi[\xi]$}
    \DisplayProof{}
    \qquad\qquad
    \AXC{$\exists \xi.\phi[\xi]$}
    \AXC{}
    \UIC{$\phi[x]$}\noLine
    \UIC{$\vdots$}\noLine
    \UIC{$\psi$}
    \BIC{$\psi$}
  \end{prooftree}
  where on the right $x$ is not free in $\psi$, and equality by
  \begin{prooftree}
    \AXC{}
    \UIC{$t=t$}
    \DisplayProof{}
    \qquad\qquad
    \AXC{$t=s$}
    \AXC{$\phi[t]$}
    \BIC{$\phi[s]$}
  \end{prooftree}
  The notion of context is easily extended to derivations.  Observe that
  the rules for $\exists$ are the only rules that do not preserve the
  contexts of formulas.
\end{dfn}

Derivations using these rules may be composed:
\begin{prooftree}
  \AXC{$\phi$}\noLine
  \UIC{$\vdots$}\noLine
  \UIC{$\psi$}\DisplayProof{}
  ,
  \AXC{$\psi$}\noLine
  \UIC{$\vdots$}\noLine
  \UIC{$\chi$}\DisplayProof{}
  \ $\mapsto$\ 
  \AXC{$\phi$}\noLine
  \UIC{$\vdots$}\noLine
  \UIC{$\psi$}\noLine
  \UIC{$\vdots$}\noLine
  \UIC{$\chi$}
\end{prooftree}
as long as both derivations have the same context, and this
composition is clearly associative, with units the identity
derivations $\phi$.  We may write $p \colon \phi \sq[x] \psi$ to
indicate that $p$ is a derivation of $\psi$ from the assumption $\phi$
with context $x$, arriving at the rules
\begin{prooftree}
  \AXC{}\UIC{$1_\phi \colon \phi \sq[x] \phi$}
  \DisplayProof \qquad \qquad
  \AXC{$p \colon \phi \sq[x] \psi$} \AXC{$q \colon \psi \sq[x] \chi$}
  \BIC{$q \cmp p \colon \phi \sq[x] \chi$}
\end{prooftree}
and thus at a category of derivations in any given context $x$.

The substitution $p[t]$ of a term $t \colon Y \to X$ into a derivation
$p[x]$ with $x$ free is defined in the obvious way, and an induction
over the structure of derivations shows that the `substitute $t$'
mapping $t^*$ is a functor from the category of derivations in the
context $x$ to derivations in the context $y$ that commutes with the
finite-product structure given by the following.

If $p_i \colon \phi \sq[x] \psi_i$ for $i=1,2$, then we may use the
$\land$-introduction rule to form a derivation $\tup{p_1,p_2} \colon
\phi \sq[x] \psi_1 \land \psi_2$, and conversely given a derivation
$p$ of the latter type the elimination rules give $\pi_i \cmp p \colon
\phi \sq[x] \psi_i$.  Imposing the ($\beta$- and $\eta$-)equalities
\begin{equation*}
  \pi_i\tup{p_1,p_2} = p_i \qquad \qquad \tup{\pi_1 p, \pi_2 p} = p
\end{equation*}
then gives a `bijective' rule
\begin{prooftree}
  \AXC{$p_1 \colon \phi \sq[x] \psi_1$}
  \AXC{$p_2 \colon \phi \sq[x] \psi_2$}\doubleLine
  \BIC{$\tup{p_1, p_2} \colon \phi \sq[x] \psi_1 \land \psi_2$}
\end{prooftree}
where to move from bottom to top we compose with $\pi_i$, and this
gives binary products in each category of derivations.  As for
$\top$, we will say that any derivation $p \colon \phi \sq[x] \top$ is
equal to the canonical $!_\phi \colon \phi \sq[x] \top$, making $\top$
the terminal object in each category of derivations.

Similarly, there is a $\beta$ rule for equality:
\begin{prooftree}
  \AXC{}\UIC{$t=t$}
  \AXC{}\noLine\UIC{$\vdots$}\noLine
  \UIC{$\phi[t]$}
  \BIC{$\phi[t]$}
  \DisplayProof{}
  \quad = \quad
  \AXC{}\noLine\UIC{$\vdots$}\noLine
  \UIC{$\phi[t]$}
\end{prooftree}
and an $\eta$ rule:
\begin{prooftree}
  \AXC{$\llap{$p\,$} \vdots$}\noLine
  \UIC{$t=t'$}\noLine
  \UIC{$\llap{$q[t,t']\,$} \vdots$}\noLine
  \UIC{$\phi[t,t']$}
  \DisplayProof{}
  \quad = \quad
  \AXC{$\llap{$p\,$} \vdots$}\noLine
  \UIC{$t=t'$}
  \AXC{}\UIC{$t=t$}\noLine
  \UIC{$\llap{$q[t,t]\,$} \vdots$}\noLine
  \UIC{$\phi[t,t]$}
  \BIC{$\phi[t,t']$}
\end{prooftree}
and these set up a bijection
\begin{equation}
  \label{eq:8}
  \centering
  \leavevmode
  \AX$\phi, x=x' \fCenter\ \sq[x,x']\ \psi[x,x']$
  \doubleLine
  \UI$\phi \fCenter\ \sq[x]\ \psi[x,x]$
  \DisplayProof
\end{equation}
between derivations of the indicated types
\cite{jacobs99:_categ_logic_and_type_theor}.  There is also a
`coherence' rule
\begin{prooftree}
  \AXC{$\vdots$} \noLine\UIC{$t=t$} \UIC{$\top$} \UIC{$t=t$}
  \DisplayProof{}
  \qquad = \qquad
  \AXC{$\vdots$} \noLine \UIC{$t = t$}
\end{prooftree}
which makes sure that $\top_X \Leftrightarrow x = x$, so that $x=x$ is the
terminal object in the category of derivations in the context $x$.

\begin{dfn}
  A (regular) theory $T$ over a signature $S$ is given by a collection
  of axioms (derivation constants, perhaps including purely equational
  axioms $t=t'$) together with a collection of equations between
  derivations built from those axioms and the above rules.
\end{dfn}

The terms of a signature, together with the equational axioms $t=t'$
of a theory over that signature, give rise to a category $\cat{B}_T$
with finite products --- the `multisorted Lawvere theory' associated
to the theory.  In this category an object is a type $X_1, X_2,
\ldots, X_n$, and a morphism from $X_1, X_2, \ldots, X_n$ to $Y_1,
Y_2, \ldots, Y_m$ is given by an $m$-tuple $\tup{t_1, t_2, \ldots,
  t_m}$ of terms, where each $t_i \colon X_1, X_2, \ldots, X_n \to
Y_i$.  Thus a theory $T$ gives rise to a pseudofunctor $\cat{E}_T(-)
\colon \B^\op_T \to \Cat$, which takes a type $X$ to the
finite-product category $\cat{E}_T(X)$ of formulas and derivations
whose context is of type $X$, and takes a term $t \colon X \to Y$ to
the substitution functor $t^* \colon \cat{E}_T(Y) \to \cat{E}_T(X)$.

\subsection{Regular fibrations}
\label{sec:regular-fibrations}

In this section we define the structures that serve as fibrational
models of regular theories.  We also recall and discuss Lawvere's
notion of comprehension in a fibration.

\begin{dfn}
  \label{dfn:11}
  Let $\cat{B}$ be a category with finite products.  The following
  squares are pullbacks in $\cat{B}$
  (\cite{seely83:_hyper_natur_deduc_and_beck_condit},
  cf.~\cite[p.~9]{lawvere70:_equal}) for any morphisms $t,t'$.
  \begin{gather*}
    \xymatrix{
      X \pbcorner \ar[r]^{\tup{X,t}} \ar[d]_t \ar@{}[dr]|{\mathrm{(A)}}
      & X \times Y \ar[d]^{t\times Y} \\
      Y \ar[r]_{d} & Y \times Y
    }
    \qquad
    \xymatrix{
      X \pbcorner \ar[r]^X \ar[d]_X \ar@{}[dr]|{\mathrm{(B)}}
      & X \ar[d]^{d} \\
      X \ar[r]_{d} & X \times X
    } \\
    \intertext{and}
    \xymatrix{
      X' \times X \pbcorner \ar[r]^{X' \times t} \ar[d]_{t' \times X}
      \ar@{}[dr]|{\mathrm{(C)}}
      & X' \times Y \ar[d]^{t' \times Y} \\
      Y' \times X \ar[r]_{Y' \times t} & Y' \times Y
    }
  \end{gather*}
  Also, if $tu = sv$ is a pullback, then so is its product with any
  object:
  \begin{gather*}
    \xymatrix{
      P \times Z \pbcorner \ar[r]^{u \times Z} \ar[d]_{v \times Z}
      \ar@{}[dr]|{\mathrm{(D)}}
      & X \times Z \ar[d]^{t \times Z} \\
      X' \times Z \ar[r]_{s \times Z} & Y \times Z
    }
  \end{gather*}
  and similarly for products on the right.

  The squares (A), (B) and (C), and those built from them using (D)
  and pasting side-by-side, are called \emph{product-absolute}
  pullbacks \cite{walters08:_froben}, because they are preserved by
  any functor that preserves products.
\end{dfn}

\begin{rem}
  \label{rem:7}
  The coassociativity square for the diagonal $d$ is
  product-absolute:
  \begin{equation*}
    \vcenter{
      \xymatrix{
        X \ar[r]^{d} \ar[d]_{d} & X^2 \ar[d]^{1 \times
          d} \\
        X^2 \ar[r]_{d \times 1} & X^3
      }
    }
    \qquad = \qquad
    \vcenter{
      \xymatrix{
        X \ar[r]^{(1,d)} \ar[d]_d & X^3 \ar[d]|{d
          \times X^2} \ar[r]^{X \times p_2} & X^2 \ar[d]^{d
          \times X} \\
        X^2 \ar[r]_d & X^4 \ar[r]_{X^2 \times p_2} & X^3
      }
    }
  \end{equation*}
  See the example after definition~5 at \cite{trimble13:_notes}.
\end{rem}

\begin{dfn}
  \label{dfn:7}
  A \emph{regular fibration} is a fibration $\E \colon \B^\op \to
  \Cat$, such that
  \begin{enumerate}
  \item $\B$, and $\E X$ for each object $X$ of $\B$, have finite
    products (the product in $\B$ is denoted $(\times, \term )$
    and that in each $\E X$ as $(\cap, \top)$);
  \item $f^* = \E f$, for each morphism $f$ of $\B$, has a left
    adjoint, denoted $\exists_f$ or $f_!$, and (hence) preserves
    finite products;
  \item these adjoints satisfy Frobenius reciprocity
    (def.~\ref{dfn:5}) and the Beck--Chevalley (def.~\ref{dfn:4})
    conditions with respect to product-absolute pullbacks
    (def.~\ref{dfn:11}) in $\B$.
  \end{enumerate}
  A morphism of regular fibrations is a product-preserving morphism of
  bifibrations, and a transformation of such is simply a
  transformation of fibration-morphisms.  These make up the 2-category
  $\bicat{RegFib}$.
\end{dfn}

Our regular fibrations are (nearly) those of
\cite{pavlovic96:_maps_ii}.  A similar definition is given in
\cite{jacobs99:_categ_logic_and_type_theor}, the only difference being
that the latter sort of regular fibration is required to have all
fibres preordered.  These we call \emph{ordered} regular fibrations.
They form a full sub-2-category $\bicat{OrdRegFib} \mon
\bicat{RegFib}$.

\begin{rem}
  \label{rem:8}
  In logical terms, the point of the Frobenius condition is that
  together with the Beck--Chevalley conditions it ensures that
  $\exists_t \phi$ is equivalent to $\exists \xi . t[\xi] = y \wedge
  \phi[\xi]$.  See \cite[Theorem,~p.~8]{lawvere70:_equal}.
\end{rem}

\begin{dfn}
  The \emph{internal language} of a regular fibration $\cat{E} \to
  \cat{B}$ is the regular theory defined as follows:
  \begin{itemize}
  \item The sorts and terms are those of the Lawvere theory $\cat{B}$,
    so that a sort is a finite list of objects of $\cat{B}$, with
    products $X \times Y$ identified with lists $X,Y$, and a term is
    either a variable (product projection) or the application
    (composition) of a function symbol (morphism of $\cat{B}$) to a
    tuple of terms.
  \item The predicates and derivations of sort $X$ are given by the
    objects and morphisms of $\cat{E}(X)$.  That is, a predicate $P$
    of sort $X_1,\ldots,X_n$ is an object $\den{P}$ of $\cat{E}$ in
    the fibre over $X_1 \times \cdots \times X_n$, conjunction $\land$
    and quantification $\exists$ are given by the regular structure of
    $p$, and a derivation is a $p$-vertical morphism of $\cat{E}$.
  \end{itemize}
\end{dfn}

\begin{dfn}
  \label{dfn:13}
  The \emph{soundness theorem} \cite[theorem 2.1.6]{oosten08:_realiz}
  says that if $\cat{E} \to \cat{B}$ is a regular fibration, then to
  each proof of a sequent
  \begin{equation*}
    P_1,\ldots,P_n \sq[x] Q
  \end{equation*}
  where $\vec x$ contains the free variables of the $P_i,Q$, there
  corresponds a vertical morphism $\den{P_1}\times \cdots \times
  \den{P_n} \to \den{Q}$ in $\cat{E}$ over the type of $x$.

  We therefore say that a fibration \emph{satisfies} a sequent if such
  a vertical morphism exists.
\end{dfn}

\begin{prp}[{\cite[Theorem, \textsection
  8]{seely83:_hyper_natur_deduc_and_beck_condit}}]
  \label{prp:8}
  If a hyperdoctrine satisfies the Beck--Chevalley condition
  (def.~\ref{dfn:4}) for the product-absolute pullbacks of
  def.~\ref{dfn:11}, then it satisfies the condition for an arbitrary
  pullback $tu=sv$ if and only if it satisfies
  \begin{gather*}
    t[m] = s[m'] \sq[] \exists \xi.(u[\xi] = m \land v[\xi] = m') \\
    \intertext{and}
    u[p] = u[p'], v[p] = v[p'] \sq[] p = p'
  \end{gather*}
  that is, if the hyperdoctrine `knows\!' that the diagram is a pullback.  
\end{prp}

Seely's proof of prop.~\ref{prp:8} goes through unchanged for a
regular fibration.

The connection with regular categories (def.~\ref{dfn:14}) is as
follows.

\begin{prp}
  A category $\C$ is regular if and only if its subobject fibration
  $\operatorname{Sub}(\C) \to \C$ that sends $S \mon X$ to $X$ is a
  (necessarily ordered) regular fibration.
\end{prp}
\begin{proof}
  If $\C$ is a regular category, then the adjunctions $\exists_f
  \dashv f^*$ come from pullbacks and images in $\C$ \cite[lemma
  A1.3.1]{johnstone02:_sketc_of_eleph} as does the Frobenius property
  [\textit{op.~cit.}, lemma A1.3.3].  The terminal object of
  $\operatorname{Sub}(X) = \operatorname{Sub}(\C)_X$ is the identity
  $1_X$ on $X$, and binary products in the fibres
  $\operatorname{Sub}(X)$ are given by pullback.  These products are
  preserved by reindexing functors $f^*$ because the $f^*$ are right
  adjoints, and the projection to $\C$ clearly preserves them too.
  The Beck--Chevalley condition follows from pullback-stability of
  images in $\C$.

  Conversely, suppose $\operatorname{Sub}(\C) \to \C$ is a regular
  fibration.  We need to show that $\C$ has equalizers (to get finite
  limits) and pullback-stable images.  But the equalizer of $f, g
  \colon X \rightrightarrows Y$ is $(f,g)^*d$.  For images, let
  $\operatorname{im} f = \exists_f \top_X$ as in \cite[lemma
  A1.3.1]{johnstone02:_sketc_of_eleph}.  Pullback-stability follows
  from the Beck--Chevalley condition, together with the fact that
  $\operatorname{Sub}(\C) \to \C$ `knows', in the sense of
  prop.~\ref{prp:8}, that any pullback is indeed a pullback.
\end{proof}

We will write $\operatorname{Arr}(\C) \to \C$ for the codomain
projection out of the category $[\cat{2}, \C]$ of morphisms of a
category $\C$.  It is well known that this is a regular fibration if
and only if $\C$ has finite limits; the non-trivial parts of the proof
are essentially as above.

Images $\im t = \exists_t \top_X = t_! \top_X$ as above make sense in
any regular fibration, and can be made functorial: for a morphism
\begin{equation*}
  \xymatrix{
    Y \ar[rr]^g \ar[dr]_t & & Z \ar[dl]^{t'} \\
    & X
  }
\end{equation*}
in $\B/X$, the morphism $\im g \colon \im t \to \im t'$ is the
composite
\begin{equation*}
  \xymatrix{t_! \top_X \ar[r]^-{\sim} & t'_! g_!
    \top_X \ar[r] & t'_! \top_Z}
\end{equation*}
where the second morphism is $t'_!$ applied to the unique $g_! \top_X
\to \top_Z$; if $g$ is the identity then the composite is the
identity, by the coherence laws for the pseudofunctor $t \mapsto t_!$
together with uniqueness of maps into a terminal object.  For a
composable pair $g, g'$ of morphisms over $X$ we get
\begin{equation*}
  \xymatrix{
    t_! \top_X \ar[d]_{\sim} & t'_! \top_Z \ar[d]_{\sim} & t''_!
    \top_W \\
    t'_! g_! \top_X \ar[ur] \ar[d]_{\sim} & t''_!g'_! \top_Z \ar[ur]
    \\
    t''_!g'_!g_! \top_X \ar[ur] \ar@/_2.5pc/[uurr]
  }
\end{equation*}
where the rectangular cell commutes by naturality and the other by
functoriality of $t''_!$ and uniqueness of maps into terminals again.
So image is a functor $\im \colon \B/X \to \E X$, for each $X \in \B$.

\begin{dfn}[{\cite{lawvere70:_equal}}]
  \label{dfn:22}
  A regular fibration $\E$ over $\B$ \emph{has comprehension} or
  \emph{is comprehensive} if for each $X \in \B$ the functor $\im
  \colon \B/X \to \E X$ has a right adjoint $P \mapsto \{P\} \colon \E
  X \to \B/X$, called \emph{extension}.  $\E$ has \emph{full
    comprehension} if each such extension functor is fully faithful,
  making each $\E X$ a reflective subcategory of $\B/X$.
\end{dfn}

This means that for each $P \in \E X$ there is a morphism $i_P
\colon \{P\} \to X$ such that for each $t \colon Y \to X$ there is a
bijection between factorizations
\begin{gather*}
\xymatrix{
    Y \ar[rr] \ar[dr]_t  & & \{P\} \ar[dl]^{i_P} \\
    & X
  }
\end{gather*}
in $\B$ and morphisms
\begin{gather*}
  \xymatrix{t_! \top_Y \ar[r] & P}
\end{gather*}
in $\E X$.  Notice that these are the same as morphisms $\top_Y \to
t^*P$ in $\E Y$, and hence correspond to maps $\top_Y \to P$ over $t$
in the total category $\E$ (cf.~the definition of `subset types',
i.e.~comprehension, in
\cite[def.~4.6.1]{jacobs99:_categ_logic_and_type_theor}).

For an object $X$ in the base of a regular fibration, the equality
predicate over $X$ is given by the image $d_! \top_X$ of the diagonal
morphism $d \colon X \to X \times X$.  The Beck--Chevalley condition
for squares of type (B) in def.~\ref{dfn:11} requires that the unit $1
\tc d^* d_!$ of the adjunction $d_! \dashv d^*$ be invertible.  Using
Frobenius reciprocity we can show
\begin{align*}
  d_! \top \cap d_! \top & \iso d_!(\top \cap d^*d_! \top) \\
  & \iso d_!(\top \cap \top) \\
  & \iso d_! \top
\end{align*}
But this isomorphism means that the following square is a pullback:
\begin{equation*}
  \xymatrix{
    d_! \top \ar[r]^1 \ar[d]_1 \pbcorner & d_! \top \ar[d] \\
    d_! \top \ar[r] & \mathbf{1}
  }
\end{equation*}
and that is equally to say that $d_! \top$ is subterminal in $\E(X
\times X)$, so that there can be at most one proof of any equality
(cf.~\cite[prop.~3.4]{lack10:_bicat}).

\begin{dfn}
  \label{dfn:33}
  Equality in a regular fibration $\E$ over $\B$ is \emph{extensional}
  (cf.~the `very strong equality' of
  \cite[3.4.2]{jacobs99:_categ_logic_and_type_theor}) if two parallel
  morphisms $f,g \colon Y \rightrightarrows X$ in $\B$ are equal
  whenever the (then necessarily unique) morphism
  \begin{equation*}
    \top_Y \longrightarrow \den{f y = g y} = (f,g)^*d_! \top_X
  \end{equation*}
  in $\E Y$ exists.
\end{dfn}

For the following result compare
\cite[Theorem,~p.~13]{lawvere70:_equal} and
\cite[exercise~4.6.6]{jacobs99:_categ_logic_and_type_theor}.

\begin{prp}
  \label{prp:7}
  A comprehensive regular fibration over $\B$ has extensional equality
  if and only if, for any parallel pair $f,g \colon Y
  \rightrightarrows X$, the map $i \colon \{ \den{f y = g y} \} \to Y$
  exhibits its domain as the equalizer of $f$ and $g$.
\end{prp}
\begin{proof}
  For any $t \colon Z \to Y$ in $\B$, there is a bijection between
  morphisms $\top_Z \to \den{ftz = gtz}$ in $\E Z$ and morphisms $t
  \to i$ in $\B/Y$, there being therefore at most one of the latter.
  If equality is extensional, then $t$ factors through $i$ if and only
  if $ft = gt$, making $i$ the equalizer of $f$ and $g$.  Conversely,
  taking $t = 1$, there is a morphism $1 \to i$ if and only if $f =
  g$, but this then corresponds to a morphism $\top = \im 1 \to
  \den{fy=gy}$.
\end{proof}

The adjunction $\im \dashv \{-\}$ gives, for each $t \colon Y \to X$,
a unit
\begin{equation*}
  \xymatrix{
    Y \ar[rr]^{e_t} \ar[dr]_t & & \{ \im t \} \ar[dl]^{i_t} \\
    & Z
  }
\end{equation*}
We will say that \emph{$t$ is an injection} if this $e_t$ is
invertible.  Notice that injections in an ordered fibration must be
monomorphisms.  Conversely, if comprehension is full then $\hom(Q,P)
\iso \hom(\im i_Q,P) \iso \hom(i_Q,i_P)$, so that if $i_P$ is a
monomorphism then there is at most one morphism into $P$ from any
object in the same fibre, and so if every injection is a monomorphism
then $\E$ is ordered.

The following is proved in
\cite[prop.~4.9.3]{jacobs99:_categ_logic_and_type_theor} in a somewhat
more general context than ours, but only for ordered fibrations.

\begin{prp}
  \label{prp:21}
  A regular fibration has extensional equality if and only if each
  diagonal $d \colon X \to X \times X$ is an injection (supposing $\{
  \im d \}$ to exist).
\end{prp}
\begin{proof}
  For any parallel pair $f,g$, there are bijections
  \begin{prooftree}
    \AXC{$\xymatrix{\top \ar[r] & (f,g)^*d_! \top}$}
    \UIC{$\xymatrix{\im\, (f,g) \ar[r] & \im d}$}
    \UIC{$\xymatrix{Y \ar[rr] \ar[dr]_{(f,g)} & & \{ \im d \}
        \ar[dl]^{i_d} \\ & X \times X}$}
  \end{prooftree}
  If $d$ is an injection, then factorizations of the last form are in
  bijection with factorizations of $(f,g)$ through $d$, but since $d$
  is monic, there can be at most one such, which exists precisely when
  $f = g$.  Conversely, to say that equality is extensional is to say
  that morphisms $\top \to \den{fy = gy}$ are in bijection with
  factorizations of $(f,g)$ through $d$, but by the correspondence
  above the former are also in bijection with maps $(f,g) \to i_d$,
  naturally in $(f,g)$.  Taking $f = g = 1$ shows that there is
  exactly one morphism $d = (1,1) \to i_d$, which must be $e_d$, and
  because the induced map $\hom(-,d) \iso \hom(-,i_d)$ is an
  isomorphism $e_d$ must be invertible.
\end{proof}

The following proposition does not seem to have been published before
in this particular form, but it is a generalization of a very
well-known fact.  Although the fibration that sends a category $\C$ to
$[\C^\op, \Set]$ is not regular, as noted already by Lawvere
\cite{lawvere70:_equal}, it does have full comprehension, with the
extension of a presheaf given by its category of elements.  In that
case the proposition reduces to the fact
\cite[exercise~III.8(a)]{mac92:_sheav_in_geomet_and_logic} that for
any presheaf $P$ on $\C$ there is an equivalence $[\C^\op, \Set]/P
\eqv [(\int P)^\op, \Set]$.

\begin{prp}
  \label{prp:24}
  Let $P$ be a predicate over $X$ in a regular fibration $\E$ with
  full comprehension.  There is then an equivalence
  \begin{equation}
    \label{eq:19}
    \E \{ P \}\  \eqv\  \E X / P
  \end{equation}
\end{prp}
\begin{proof}
  First we note two small facts: firstly, if a functor $F \colon \C
  \to \D$ is fully faithful, then so is its action $F_X \colon \C/X
  \to \D/FX$ on any slice of $\C$, and secondly, if $F \colon \C \to
  \D$ has a right adjoint $G$, then for any $Y$ in $\D$ the functor
  $G_Y \colon \D/Y \to \C/GY$ has a left adjoint given by $F_{GY}$
  followed by composition with $\epsilon_Y \colon GFY \to Y$.  These
  two facts together imply that reflective subcategories give rise to
  reflective subcategories of slice categories.

  If $\E$ has full comprehension, then $\E X$ is reflective in
  $\B/X$, and hence so is $\E X / P$ in $(\B/X)/i_P$.  A folklore
  result says that the latter is equivalent to $\B/\{P\}$:
  \begin{equation*}
    \vcenter{
      \xymatrix{
        Y \ar[rr]^t \ar[dr] & & \{P\} \ar[dl]^{i_P} \\
        & X
      }
    }
    \qquad \Leftrightarrow \qquad
    \vcenter{
      \xymatrix{
        Y \ar[d]^t \\
        \{P\}
      }
    }
  \end{equation*}
  So both $\E X / P$ and $\E\{P\}$ are reflective in $\B/\{P\}$.  To
  show that they are equivalent, it suffices to show that for each $t
  \colon Y \to \{P\}$ in $\B$, the unit $\eta_t$ is invertible in
  $\B/\{P\}$ if and only if the corresponding $\eta'_t$ is so in
  $(\B/X)/i_P$:
  \begin{equation*}
    \vcenter{
      \xymatrix{
        Y \ar[rr]^{\eta_t} \ar[dr]_t & & \{ \im t \} \ar[dl] \\
        & \{P\}
      }
    }
    \qquad  \qquad
    \vcenter{
      \xymatrix{
        Y \ar[rr]^{\eta_t} \ar[dr]_t \ar@/_1pc/[ddr]
        & & \{ \im t \} \ar[dl] \ar@/^1pc/[ddl]
        \\
        &\{P\} \ar[d] \\
        & X
      }
    }
  \end{equation*}
  But the two are manifestly given by the same morphism of $\B$, and
  invertibility in a slice category is equivalent to invertibility in
  the underlying category.
\end{proof}

\subsection{The classifying fibration of a regular theory}
\label{sec:class-fibr-regul}

As something of an aside, we will construct in this section the
`syntactic model' of a regular theory.  Most of this material is at
least sketched in \cite{seely83:_hyper_natur_deduc_and_beck_condit}
for the hyperdoctrine corresponding to a first-order theory, but an
explicit presentation of what remains for the regular case is useful
and illuminating.

We want to show firstly that a regular theory $T$ gives rise to a
bifibration $\cat{E}_T \to \cat{B}_T$.

\begin{prp}
  Let $T$ be a regular theory.  For each term $t \colon X \to Y$, the
  functor $t^* \colon \cat{E}_T(Y) \to \cat{E}_T(X)$ has a left
  adjoint $\exists_t$.
\end{prp}
\begin{proof}
  Define $\exists_t$ on formulas as
  \begin{equation*}
    \exists_t \phi = \exists \xi.(t[\xi] = y \land \phi[\xi])
  \end{equation*}
  It suffices to show that for any $\phi[x]$ of type $X$ there is a
  universal $\eta^t_\phi \colon \phi \sq[x] t^*\exists_t \phi$; that
  is, for any equivalence class of proofs $p \colon \phi \sq[x]
  t^*\psi$, there is a unique $\hat p \colon \exists_t \phi \sq[y]
  \psi$ such that $t^* \hat p \cmp \eta^t_\phi$ is equal to $p$.  The
  derivation $\eta^t_\phi$ is obtained by forming the derivation
  \begin{equation}
    \label{eq:9}
    \centering\leavevmode
    \AXC{$x=x'$} \AXC{}\UIC{$t[x] = t[x]$}
    \BIC{$t[x'] = t[x]$}
    \AXC{$x=x'$} \AXC{$\phi[x]$}
    \BIC{$\phi[x']$}
    \BIC{$t[x']=t[x] \land \phi[x']$}
    \UIC{$\exists \xi.(t[\xi] = t[x] \land \phi[\xi])$}
    \DisplayProof
  \end{equation}
  of type $\phi[x], x = x' \sq[x,x'] t^*\exists_t \phi$ and using the
  bijection \eqref{eq:8} above to get rid of the hypothesis $x=x'$.
  Given $p \colon \phi \sq[x] t^*\psi$, let $\hat p$ be
  \begin{prooftree}
    \AXC{$\exists \xi.(t[\xi] = y \land \phi[\xi])$}

    \AXC{}\UIC{$t[x]=y \land \phi[x]$}
    \UIC{$t[x] = y$}

    \AXC{}\UIC{$t[x]=y \land \phi[x]$}
    \UIC{$\phi[x]$}\noLine\UIC{$\vdots$}\noLine
    \UIC{$\psi[t[x]]$}

    \BIC{$\psi[y]$}

    \BIC{$\psi[y]$}
  \end{prooftree}
  The $\beta$ and $\eta$ equalities given above show that the composite
  $t^*\hat p \cmp \eta^t_\phi$ is equal to $p$, and uniqueness of $\hat
  p$ follows from the normal form theorem for natural deduction
  \cite{prawitz06}.  So we have another bijection
  \begin{prooftree}
    \AX$\exists_t \phi \fCenter\ \sq[y]\ \psi$
    \doubleLine
    \UI$\phi \fCenter\ \sq[x]\ t^*\psi$
  \end{prooftree}
\end{proof}

In particular, we have the usual rewriting rules, as given in
\cite{seely83:_hyper_natur_deduc_and_beck_condit}:
\begin{prooftree}
  \AXC{$\llap{$p\,$}\vdots$}\noLine
  \UIC{$\phi[t]$} \UIC{$\exists \xi . \phi[\xi]$}
  \AXC{}\UIC{$\phi[x]$}\noLine\UIC{$\vdots\rlap{$\,q[x]$}$}\noLine\UIC{$\psi$}
  \BIC{$\psi$}
  \DisplayProof{}
  \qquad = \qquad
  \AXC{$\llap{$p\,$}\vdots$}\noLine
  \UIC{$\phi[t]$}\noLine
  \UIC{$\llap{$q[t]\,$}\vdots$}\noLine
  \UIC{$\psi$}
\end{prooftree}
and
\begin{prooftree}
  \AXC{$\llap{$p\,$}\vdots$}\noLine
  \UIC{$\exists \xi . \phi[\xi]$} \noLine
  \UIC{$\llap{$q\,$}\vdots$}\noLine
  \UIC{$\psi$}
  \DisplayProof{}
  \qquad = \qquad
  \AXC{$\llap{$p\,$}\vdots$}\noLine
  \UIC{$\exists \xi . \phi[\xi]$}
  \AXC{}\UIC{$\phi[x]$}
  \UIC{$\exists \xi. \phi[\xi]$}
  \noLine
  \UIC{$\llap{$q\,$}\vdots$}\noLine
  \UIC{$\psi$}
  \BIC{$\psi$}
\end{prooftree}

For $\cat{E}_T \to \cat{B}_T$ to be a regular fibration, it must satisfy the
Frobenius and Beck--Chevalley conditions.  The former means that for
any term $t$ the canonical map $\exists_t (\phi \land t^*\psi) \sq[y]
(\exists_t \phi) \land \psi$ is an isomorphism.  This canonical map is
given \cite[definition D1.3.1(i)]{johnstone02:_sketc_of_eleph} by
\begin{prooftree}
  \AXC{$\phi \land t^* \psi \sq[x] t^* \psi$}
  \UIC{$\exists_t (\phi \land t^*\psi) \sq[y] \psi$}

  \AXC{$\phi \land t^* \psi \sq[x] \phi$}
  \AXC{$\exists_t \phi \sq[y] \exists_t \phi$}
  \UIC{$\phi \sq[x] t^* \exists_t \phi$}
  \BIC{$\phi \land t^* \psi \sq[x] t^* \exists_t \phi$}
  \UIC{$\exists_t(\phi \land t^* \psi) \sq[x] \exists_t \phi$}

  \BIC{$\exists_t(\phi \land t^* \psi) \sq[x] (\exists_t \phi) \land
    \psi$}
\end{prooftree}
So we must insist that in $\cat{E}_T$ the above proof, call it $p$, have a
formal inverse $\inv p \colon (\exists_t \phi) \land \psi \sq[x]
\exists_t(\phi \land t^* \psi)$, adding to the equations above $\inv p
p = 1$ and $p \inv p = 1$.

The Beck--Chevalley conditions for the product-absolute pullbacks (A),
(C) and (D) in def.~\ref{dfn:11} are shown as in \cite[\textsection
4]{seely83:_hyper_natur_deduc_and_beck_condit}.

\begin{prp}
  The Beck--Chevalley condition for \ref{dfn:11}(B) holds; that is,
  $\eta^d$, as defined by \eqref{eq:9}, is invertible.
\end{prp}
\begin{proof}
  An inverse is given by
  \begin{prooftree}
    \AXC{$\exists \xi. d[\xi] = d[x] \land \phi[\xi]$}

    \AXC{$(x', x') = (x, x) \land \phi[x']$}
    \UIC{$(x', x') = (x, x)$}

    \AXC{$(x', x') = (x, x) \land \phi[x']$}
    \UIC{$\phi[x']$}
    \BIC{$\phi[x]$}

    \BIC{$\phi[x]$}
  \end{prooftree}
  That this derivation is a left inverse for $\eta^d_\phi$
  follows from the $\beta$-reductions given above, and conversely that
  it is a right inverse follows from the $\eta$-reductions for
  $\land$, $\exists$ and $=$.
\end{proof}

We can now perform the usual rites of categorical logic: a model of a
regular theory $T$ in a regular fibration $\cat{E} \to \cat{B}$ is a
morphism of regular fibrations from $\cat{E}_T \to \cat{B}_T$ to
$\cat{E} \to \cat{B}$, and it is easy to see that this is equivalent
to the traditional notion.  Completeness is automatic, because if a
sequent is true in every model then it is true in the syntactic model
and thence provable.

\section{Allegories and bicategories of relations}
\label{sec:alleg-their-compl}

In this section we recall the structures used to give the
bicategorical or relational semantics of regular theories, in the
locally ordered case.  First \emph{allegories} and then
\emph{bicategories of relations} are defined, and the latter are shown
to be the same as certain allegories.  Later on it will become clear
that they are also a special case of the \emph{regular equipments}
that we will define.

\subsection{Allegories and their completions}
\label{sec:allegories}

Here we define allegories and the idempotent splitting construction.
Nothing in this section is original.

\begin{dfn}[{\cite{freyd90:_categ_alleg,johnstone02:_sketc_of_eleph}}]
  An \emph{allegory} $\bA$ is a strict 2-category whose each
  hom-category $\bA(X,Y)$ is a poset with binary meets $\cap$,
  and that comes equipped with a strict involution $(-)^\circ \colon
  \bA^\op \to \bA$ that is the identity on objects and
  satisfies the \emph{modular law} for all suitably-typed morphisms
  $r$, $s$, $t$:
  \begin{equation}
    \label{eq:10}
    s r \cap t \leq (s \cap t r^\circ)r
  \end{equation}
  An allegory functor $F \colon \bA \to \bB$ is a 2-functor that
  preserves $\cap$ and $(-)^\circ$.  A transformation $\alpha \colon F
  \nt G$ is an oplax transformation (i.e.~$\alpha_Y \hcmp F s \leq G s
  \hcmp \alpha_X$ for any $s \in \bA(X,Y)$) whose components
  $\alpha_X$ have right adjoints (they are \emph{maps}).  There is
  then a 2-category $\bicat{All}$ of allegories, functors and
  transformations.
\end{dfn}
Note that hom-posets are not required to have top elements, and that
composition is not required to preserve local meets (although it must
preserve the local ordering).  Note also that the modular law as above
is equivalent to the dual form
\begin{equation}
  \label{eq:11}
  sr \cap t \leq s (r \cap s^\circ t)
\end{equation}

Morphisms are written as e.g.~$r \colon X \rel Y$.  A morphism is
called a \emph{map} if it has a right adjoint.  Maps are written as $f
\colon X \map Y$; the right adjoint of $f$ is $f^\bullet$.

We recall some basic facts about allegories.

\begin{lem}[{\cite[lemma A3.2.3]{johnstone02:_sketc_of_eleph}}]
  \label{lem:3}
  If $f \colon X \map Y$ is a map, then its right adjoint is
  $f^\circ$.  Further, the ordering on maps is discrete: if $f \leq g$
  then $f = g$.  Hence the evident sub-2-category $\Map(\bA)$ is a
  category.
\end{lem}

\begin{rem}
  It follows that the componentwise ordering on transformations
  between allegory functors is also discrete, so that the 2-category
  $\bicat{All}$ is just a 2-category.
\end{rem}

\begin{lem}
  \label{lem:6}
  If $r^\circ r \leq 1$ (e.g.~if $r = f^\bullet$ is a right adjoint),
  then the modular law \eqref{eq:10} is an identity, and if $s s^\circ
  \leq 1$ (e.g.~if $s = f$ is a map) then the dual modular law
  \eqref{eq:11} is an identity.
\end{lem}
\begin{proof}
  $(s \cap t r) r^\circ \leq (s r \cap t) r^\circ r \leq s r \cap t$,
  and dually.
\end{proof}

\begin{lem}[{\cite[corollary~A3.1.6]{johnstone02:_sketc_of_eleph}}]
  The \emph{distributivity laws} hold:
  \begin{align*}
    (r \cap s) t & \leq r t \cap s t \\
    t (r \cap s) & \leq t r \cap t s 
  \end{align*}
  If $t t^\circ \leq 1$ then the first is an identity, and dually if
  $t^\circ t \leq 1$ then the second is an identity.
\end{lem}

\begin{dfn}
  A \emph{tabulation} of $r \colon X \rel Y$ is a span of maps $f
  \colon Z \map X$, $g \colon Z \map Y$ such that $r = g f^\circ$ and
  $f^\circ f \cap g^\circ g = 1$.  An allegory is \emph{tabular} if
  every morphism has a tabulation; it is \emph{pre-tabular} if every
  morphism is contained in one that has a tabulation.  Clearly, an
  allegory whose every hom-poset has a top element is pre-tabular if
  and only if each such element has a tabulation.
\end{dfn}

\begin{dfn}
  \label{dfn:12}
  A \emph{unit} in an allegory $\bA$ is an object $U$ such that $1_U$
  is the top element of $\bA(U,U)$ and for any $X$ there exists a
  morphism $p \colon X \rel U$ satisfying $p^\circ p \geq 1$.  An
  allegory is called \emph{unitary} if it has a unit.
\end{dfn}

\begin{lem}[{\cite[lemmas~A3.2.8,~A3.2.9]{johnstone02:_sketc_of_eleph}}]
  \label{lem:4}
  \leavevmode
  \begin{enumerate}
  \item If $\bA$ has a unit then its hom-posets have top
    elements.
  \item If $U$ is a unit in $\bA$, then it is the terminal
    object of $\Map(\bA)$.
  \end{enumerate}
\end{lem}

\begin{lem}[{\cite[lemma~A3.2.4]{johnstone02:_sketc_of_eleph}}]
  Suppose that $r \colon X \rel Y$ has a tabulation $(f,g)$, and that
  $i \colon W \map X$, $j \colon W \map Y$ are maps.  Then $j i^\circ
  \leq r$ if and only if there exists a map $h$ such that $i = f h$
  and $j = g h$.  Such a $h$ is necessarily unique.
\end{lem}

\begin{cor}
  \label{cor:7}
  If the top element of $\bA(X,Y)$ has a tabulation $(f,g)$,
  then that span is a product cone.  Hence, if $\bA$ is a
  unitary pre-tabular allegory, then $\Map(\bA)$ has finite
  products.
\end{cor}

\begin{prp}[{\cite[theorem~A3.2.10]{johnstone02:_sketc_of_eleph}}]
  \label{prp:11}
  An allegory $\bA$ is unitary and tabular if and only if $\Map(\bA)$
  is a regular category.  In that case $\bA \eqv \Rel(\Map(\bA))$.  If
  $\C$ is a regular category, then $\Rel(\C)$ is a unitary tabular
  allegory, and $\C \eqv \Map(\Rel(\C))$.
\end{prp}

Next we recall the theory of idempotents in allegories and the
construction of the universal allegory in which a given class of
idempotents splits.

\begin{dfn}
  \label{dfn:35}
  An endomorphism $r \colon X \rel X$ in an allegory is called
  \begin{itemize}
  \item \emph{reflexive} if $1 \leq r$;
  \item \emph{transitive} if $r r \leq r$;
  \item \emph{symmetric} if $r^\circ = r$;
  \item \emph{coreflexive} if $r \leq 1$;
  \item \emph{idempotent} if $r r = r$.
  \end{itemize}
  A morphism that is reflexive, transitive and symmetric is called an
  \emph{equivalence}.
\end{dfn}

\begin{lem}[{\cite[lemma~A3.3.2]{johnstone02:_sketc_of_eleph}}]
  \label{lem:8}
  A symmetric transitive morphism is idempotent.  A coreflexive
  morphism is symmetric and idempotent.
\end{lem}

\begin{dfn}
  \label{dfn:16}
  An idempotent $r \colon X \rel X$ \emph{splits} if there is an
  object $X_s$ and a pair of morphisms $s \colon X_s \rel X$, $s'
  \colon X \rel X_s$ such that $s s' = r$ and $s' s = 1$.
\end{dfn}

\begin{lem}[{\cite[lemma~A3.3.3]{johnstone02:_sketc_of_eleph}}]
  \label{lem:7}
  If a symmetric idempotent $r \colon X \rel X$ splits as $r = s s'$,
  then $s' = s^\circ$.  If $r$ is reflexive then $s'$ is a map; if it
  is coreflexive then $s$ is a map.
\end{lem}

\begin{dfn}
  An allegory is \emph{effective} if all of its equivalences split.
\end{dfn}

\begin{prp}[{\cite[prop.~A3.3.6]{johnstone02:_sketc_of_eleph}}]
  \label{prp:12}
  A unitary tabular allegory $\bA$ is effective if and only if
  $\Map(\bA)$ is (Barr) exact.  A regular category $\C$ is exact if
  and only if $\Rel(\C)$ is effective.
\end{prp}

\begin{dfn}
  \label{dfn:17}
  Let $\bA$ be an allegory and $S$ a class of symmetric
  idempotents in $\bA$ that includes the identities.  Then the
  \emph{splitting} of $S$ is the allegory $\bA[\check S]$ with
  objects the elements $s \colon X \rel X$ of $S$ and morphisms $s
  \rel s'$ given by morphisms $m \colon X \rel X'$ such that $m s = m
  = s' m$.
\end{dfn}

\begin{prp}[{\cite[theorem~A3.2.10]{johnstone02:_sketc_of_eleph}}]
  \label{prp:9}
  $\bA[\check S]$ is an allegory, and there is a functor
  $\bA \to \bA[\check S]$, which preserves the unit, if
  $\bA$ has one.
\end{prp}

\begin{rem}
  \label{rem:2}
  The functor $\bA \to \bA[\check S]$ sends an object $X$
  to $1_X$ and a morphism $X \rel Y$ to the same morphism considered
  as a morphism $1_X \rel 1_Y$ of idempotents.  It is thus fully
  faithful.
\end{rem}

\begin{prp}[{\cite[prop.~A3.3.6]{johnstone02:_sketc_of_eleph}}]
  \label{prp:13}
  An allegory is (unitary and) tabular if and only if it is (unitary
  and) pre-tabular and all of its coreflexives split.  If $\bA$
  is (unitary and) pre-tabular and $\mathrm{crf}$ is the class of
  coreflexives in $\bA$, then the functor $\bA \to
  \bA[\check{\mathrm{crf}}]$ is universal from $\bA$ to
  (unitary and) tabular allegories.
\end{prp}

\begin{prp}[{\cite[prop.~A3.3.9]{johnstone02:_sketc_of_eleph}}]
  \label{prp:14}
  If $\bA$ is any allegory, and $\mathrm{eqv}$ is its class of
  equivalences, then $\bA[\check{\mathrm{eqv}}]$ is effective,
  and tabular if $\bA$ is, and the functor $\bA \to
  \bA[\check{\mathrm{eqv}}]$ is universal from $\bA$ to
  effective allegories.
\end{prp}

\begin{dfn}
  If $\C$ is a finitely complete category, let $\Span(\C)$ denote the
  2-category of spans in $\C$.  The allegory $\Span'(\C)$ is the local
  poset reflection of $\Span(\C)$.  It is unitary and pre-tabular
  \cite[example~3.3.8]{johnstone02:_sketc_of_eleph}.
\end{dfn}

\begin{cor}[of props \ref{prp:11}, \ref{prp:12}, \ref{prp:13} and \ref{prp:14}]
  \label{cor:2}
  \leavevmode
  \begin{enumerate}
  \item If $\C$ is a finitely complete category, then its regular
    completion is given by
    \begin{equation*}
      \C_{\mathrm{reg/lex}} \eqv
      \Map(\Span'(\C)[\check{\mathrm{crf}}])
    \end{equation*}
  \item If $\C$ is a regular category, then its exact completion is
    given by
    \begin{equation*}
      \C_{\mathrm{ex/reg}} \eqv
      \Map(\Rel(\C)[\check{\mathrm{eqv}}])
    \end{equation*}
  \end{enumerate}
\end{cor}

\subsection{Bicategories of relations}
\label{sec:alleg-bicat-relat}

\begin{dfn}[{\cite{carboni87:_cartes_bicat_i}}]
  \label{dfn:15}
  A (locally ordered) \emph{cartesian bicategory} is a locally
  partially ordered 2-category $\bicat{C}$ satisfying the following:
  \begin{enumerate}
  \item $\bicat{C}$ is symmetric monoidal: there is a pseudofunctor
    $\otimes \colon \bicat{C} \times \bicat{C} \to \bicat{C}$ together
    with natural isomorphisms $\alpha$, $\lambda$, $\rho$ and $\sigma$
    satisfying the usual coherence conditions;
  \item every object of $\bicat{C}$ is a commutative comonoid, that
    is, comes equipped with maps
    \begin{equation*}
      d_X \colon X \map X \otimes X
      \qquad e_X \colon X \map I
    \end{equation*}
    whose right adjoints we write $d^\bullet_X,
    e^\bullet_X$, where $I$ is the tensor unit, satisfying the
    obvious associativity, symmetry and unitality axioms, and this is
    the only such comonoid structure on $X$;
  \item \label{item:6} every morphism $r \colon X \rel Y$ is a lax
    comonoid morphism:
    \begin{equation*}
      d_Y \hcmp r \leq (r \otimes r) \hcmp d_X
      \qquad e_Y \hcmp r \leq e_X
    \end{equation*}
  \end{enumerate}
  A \emph{cartesian functor} between cartesian bicategories is a
  (strong) monoidal 2-functor, and a cartesian transformation is an
  oplax transformation whose components are maps, as for allegories.
\end{dfn}

\begin{prp}[{\cite[theorem~1.6]{carboni87:_cartes_bicat_i}}]
  \label{prp:15}
  A 2-category $\bicat{C}$ is a cartesian bicategory if and only if
  the following hold:
  \begin{enumerate}
  \item $\Map(\bicat{C})$ has finite 2-products (given by $\otimes$
    and $I$).
  \item The hom-posets of $\bicat{C}$ have finite products $\cap,
    \top$, and $1_I$ is the terminal object of $\bicat{C}(I,I)$.
  \item The tensor product defined as
    \begin{equation*}
      r \otimes s = (p_1^\bullet r p_1) \cap (p_2^\bullet s
      p_2) 
    \end{equation*}
    where the $p_i$ are the product projections, is functorial.
  \end{enumerate}
\end{prp}

\begin{rem}
  \label{rem:10}
  This definition clearly makes sense even if $\bicat{C}$ is not
  locally ordered, and indeed is the one given in
  \cite[defs~3.1,~4.1]{carboni08:_cartes_ii}, but we will stick to the
  locally ordered ones until section \ref{sec:cartesian-equipments}.
  The local finite products referred to are given by: $r \cap s =
  d^\bullet (r \otimes s) d$ and $\top = e^\bullet
  e$.
\end{rem}

\begin{dfn}[{\cite[def.~2.1]{carboni87:_cartes_bicat_i}}]
  An object $X$ in a cartesian bicategory is called \emph{Frobenius}
  (Carboni--Walters say \emph{discrete}) if it satisfies
  \begin{equation}
    \label{eq:12}
    d \hcmp d^\bullet = (d^\bullet \otimes 1) \hcmp (1
    \otimes d)
  \end{equation}
  or, in other words, if the (in fact, either, see
  \cite[lemma~3.2]{walters08:_froben}) Beck--Chevalley condition holds
  for $d_X$'s associativity square $(1 \otimes d)d =
  (d \otimes 1)d$.

  A \emph{bicategory of relations} is a cartesian bicategory in which
  every object is Frobenius.
\end{dfn}

\begin{rem}
  \label{rem:3}
  By \cite[remark~2.2]{carboni87:_cartes_bicat_i} the unit $I$ is
  always Frobenius, and $X \otimes Y$ is Frobenius if $X$ and $Y$ are.
  So a full sub-2-category of a bicategory of relations that contains
  $I$ and is closed under $\otimes$ is again a bicategory of
  relations.
\end{rem}

\begin{prp}[{\cite[theorem~2.4]{carboni87:_cartes_bicat_i}}]
  \label{prp:17}
  A bicategory of relations $\bB$ is compact closed, that is,
  there is an identity-on-objects involution $(-)^\circ \colon
  \bB^\op \to \bB$ and a natural isomorphism
  \begin{equation*}
    \bB(X \otimes Y, Z) \iso \bB(X, Z \otimes Y)
  \end{equation*}
  In addition, two dual forms of the modular law hold:
  \begin{align}
    \label{eq:15}
    (r \otimes 1) d & \leq (1 \otimes r^\circ) d r
    \\
    \label{eq:16}
    d^\bullet (r \otimes 1) & \leq r d^\bullet (1 \otimes
    r^\circ)
  \end{align}
  with equality in the first if $r^\circ r \leq 1$ (e.g.~if $r =
  f^\bullet$ is a right adjoint) and in the second if $r r^\circ \leq
  1$ (e.g.~if $r = f$ is a map) (cf.~lemma~\ref{lem:6}).
\end{prp}
\begin{proof}[Sketch of proof]
  The bijection is given by composition with $1 \otimes \eta_Y$ in one
  direction and $1 \otimes \zeta_Y$ in the other, where
  \begin{align*}
    \eta_Y & = \xymatrix{ I \ar[r]^{e_Y^\bullet} & Y
      \ar[r]^-{d_Y} & Y \otimes Y } \\
    \zeta_Y & = \xymatrix{Y \otimes Y \ar[r]^-{d_Y^\bullet} & Y
      \ar[r]^{e_Y} & I}
  \end{align*}
  One then shows that these are the unit and counit for a duality $Y
  \dashv Y$.  The bijection above is natural in $X$ and $Z$ and
  `extranatural' in $Y$, meaning that the correspondence
  \begin{prooftree}
    \AXC{$\xymatrix{X \otimes Y' \ar[r]^{1 \otimes s} & X \otimes Y
      \ar[r]^-{r} &  Z}$}
    \UIC{$\xymatrix{X \ar[r]_-{r'} & Z \otimes Y \ar[r]_{1 \otimes
        s^\circ} & Z \otimes Y'}$}
  \end{prooftree}
  holds, where $r \mapsto r'$ is transposition and $(-)^\circ$ is
  given by composition with $1 \otimes \eta$ on one side and $\zeta
  \otimes 1$ on the other.
\end{proof}

\begin{lem}[{\cite[corollary~2.6]{carboni87:_cartes_bicat_i}},
  cf.~lemma~\ref{lem:3}]
  \label{lem:9}
  In a bicategory of relations, if $f$ is a map then $f^\bullet =
  f^\circ$, and if $f$ and $g$ are maps and $f \leq g$ then $f = g$.
\end{lem}

\begin{lem}
  \label{lem:5}
  If $\C$ is a regular category, then $\Rel(\C)$ is a bicategory of
  relations.
\end{lem}
\begin{proof}
  Conditions 1 and 2 of prop.~\ref{prp:15} clearly hold.  For the
  third, we may reason in the internal language of $\C$.  Clearly
  \begin{equation*}
    1 \otimes 1 = \den{x = x \wedge x' = x'} = 1
  \end{equation*}
  Suppose $X \overset{r}{\rel} Y \overset{s}{\rel} Z$ and $X'
  \overset{r'}{\rel} Y' \overset{s'}{\rel} Z'$.  Then $sr \otimes
  s'r'$ is the meaning of
  \begin{align*}
    & (\exists \upsilon. r(x,\upsilon) \wedge s(\upsilon,z)) \wedge
    (\exists \upsilon'. r'(x',\upsilon') 
    \wedge s'(\upsilon',z')) \\
    & \Leftrightarrow \exists \upsilon. (r(x,\upsilon) \wedge s(\upsilon,z) \wedge
    \upsilon^*\exists \upsilon'. r'(x',\upsilon') 
    \wedge s'(\upsilon',z')) \\
    & \Leftrightarrow \exists \upsilon. (r(x,\upsilon) \wedge s(\upsilon,z) \wedge
    \exists \upsilon'. \upsilon^* (r'(x',\upsilon') 
    \wedge s'(\upsilon',z'))) \\
    & \Leftrightarrow \exists \upsilon, \upsilon'. r(x,\upsilon) \wedge
    s(\upsilon,z) \wedge r'(x',\upsilon') 
    \wedge s'(\upsilon',z')
  \end{align*}
  by two uses of Frobenius reciprocity and one of Beck--Chevalley (for
  a product-absolute pullback), and this last is the meaning of $(s
  \otimes s')(r \otimes r')$.  Finally, the Frobenius law is
  \begin{align*}
    & \exists \xi'. (x_1, x_2) = (\xi',\xi') = (x_3, x_4) \\
    & \Leftrightarrow
    \exists \xi'. (x_1, \xi') = (x_3, x_3) \wedge (x_2,x_2) = (x_4,
    \xi')
  \end{align*}
  which follows simply from transitivity and symmetry of $=$.
\end{proof}

\begin{prp}
  \label{prp:16}
  A bicategory of relations is the same thing as a unitary pre-tabular
  allegory.
\end{prp}
\begin{proof}
  Suppose $\bB$ is a bicategory of relations.  It is thus a
  locally partially ordered 2-category equipped with an
  identity-on-objects involution.  It satisfies the the modular law by
  \cite[remark~2.9(ii)]{carboni87:_cartes_bicat_i} and so is an
  allegory.  The tensor unit $I$, the terminal object of
  $\Map(\bB)$, is a unit (def.~\ref{dfn:12}): there is a unique
  map $X \to I$ for any $X$, and $1_I$ is the top element of
  $\bB(I,I)$ by prop.~\ref{prp:15}.  By corollary~\ref{cor:7}
  the product projections tabulate the top elements, so $\bB$ is
  pre-tabular.

  Conversely\footnote{This part of the proof was suggested by Mike
    Shulman.  A direct proof is possible, but essentially amounts to
    translating the proof of lemma~\ref{lem:5} through the equivalence
    $\bA[\check{\mathrm{crf}}] \eqv \Rel(\C)$.}, let $\bA$ be a
  unitary pre-tabular allegory.  By remark~\ref{rem:2}, $\bA$ embeds
  faithfully into $\bA[\check{\mathrm{crf}}]$, which is unitary and
  tabular by prop.~\ref{prp:13}, hence equivalent by
  prop.~\ref{prp:11} to $\Rel(\C)$ for $\C$ the regular category
  $\Map(\bA)$, hence a bicategory of relations by lemma~\ref{lem:5}.
  So by remark~\ref{rem:3} it suffices to show that $\bA$ is closed
  under $\times$ in $\bA[\check{\mathrm{crf}}]$.  Any allegory functor
  must preserve tabulations (because it preserves $\cap$ and
  $(-)^\circ$), while the inclusion $\bA \to
  \bA[\check{\mathrm{crf}}]$ preserves the unit and the property of
  being a map, and thus preserves top morphisms.  So the tabulation
  \begin{equation*}
    X \lmap X \times Y \map Y
  \end{equation*}
  of $\top_{XY}$ in $\bA$ is a tabulation of $\top_{1_X 1_Y}$ in
  $\bA[\check{\mathrm{crf}}]$, and therefore $1_{X \times Y}
  \iso 1_X \times 1_Y$.
\end{proof}

\begin{thm}
  The 2-category $\bicat{BiRel}$ of bicategories of relations and
  cartesian functors and transformations is equivalent to the locally
  full sub-2-category $\bicat{UPtAll}$ of $\bicat{All}$ on the unitary
  pre-tabular allegories and unit-preserving functors.
\end{thm}
\begin{proof}
  It suffices to show that a 2-functor is a cartesian functor if and
  only if it is a unit-preserving allegory functor.  But a strong
  monoidal functor must preserve products in categories of maps, hence
  $d$ and $t$ and their right adjoints, hence $\cap$ and
  $(-)^\circ$, and also the unit object.  Conversely, a functor that
  preserves $\cap$ and $(-)^\circ$ must preserve (tabulations and
  thus) products, and so preserve the tensor product.
\end{proof}


\chapter{2- and 3-categories}
\label{cha:2-categories}

This chapter prepares the ground for the next by recalling some
existing definitions and facts regarding higher categories, and
developing some new ones that will be needed later.

The next section reviews some notions of formal category theory in a
2-category, and defines monoidal 2-categories and functors between
them, along with a few other 3-dimensional notions.  The subsequent
section is where our original work begins: we want to define a
3-category of `2-profunctors' while avoiding the long and tedious
calculations that would be needed to prove that it is a 3-category.
Instead we will mimic the definition of $\bicat{Prof}$ as the
2-category of presheaf categories and cocontinuous functors.  This is
clearly a 2-category, because it is a sub-2-category of $\Cat$.  So in
section~\ref{sec:limits-colimits} we review the relevant facts about
2-dimensional limits and colimits, and define 2-dimensional ends and
coends and show how they may be computed in $\Cat$.  In the next
chapter we will define $\Biprof$ as the locally full sub-3-category of
$\Bicat$ on the `presheaf 2-categories' and the colimit-preserving
functors.

\section{Adjunctions and monads}
\label{sec:2-categories}

\subsection{Adjunctions}
\label{sec:adjunctions-monads}

We take as known the notion of adjoint morphisms in a 2-category.  In
that setting there is a useful generalization of adjoint
transposition.

\begin{dfn}[{\cite{kelly74:_review}}]
  \label{dfn:6}
  Given adjunctions $f \dashv u$ and $f' \dashv u'$ in a 2-category
  $\bK$, the \emph{mate} of a 2-cell
  \begin{equation*}
    \xymatrix{
      x \ar[r] \ar[d]_{f} & x' \ar[d]^{f'} \\
      y \ar[r] & y' \ar@{}[ul]|{\displaystyle\Downarrow}
    }
  \end{equation*}
  is the 2-cell
  \begin{equation*}
    \vcenter{
      \xymatrix{
        x \ar[r] & x' \\
        y \ar[r] \ar[u]^u & y' \ar[u]_{u'} \ar@{}[ul]|{\displaystyle\Downarrow}
      }
    }
    \qquad = \qquad
    \vcenter{
      \xymatrix{
      y \ar[r]^u \ar[dr]_1^{}="a" & x \ar@{}[];"a"|(.7){\displaystyle\Downarrow} \ar[r]
      \ar[d]^{f} & x' \ar[d]_{f'} \ar[dr]^1_{}="b"
      \\
      & y \ar[r] & y' \ar@{}[];"b"|(.7){\displaystyle\Downarrow}
      \ar@{}[ul]|{\displaystyle\Downarrow} \ar[r]_{u'} & x'
      }
    }
  \end{equation*}
  given by pasting with the counit of $f \dashv u$ and the unit of $f'
  \dashv u'$.  Dually, the mate of a square with opposite sides $u$
  and $u'$ is given by pasting with the unit of the first adjunction
  and the counit of the other.  This correspondence is bijective, by
  the triangle equalities.

  The mate of an invertible cell is not in general invertible.  On the
  other hand, given a square each of whose sides has a right adjoint,
  its mate, defined as above, has a further mate with respect to the
  other pair of opposite sides.  It then follows from the triangle
  equalities that this `double mate' is invertible if the original
  2-cell was.
\end{dfn}

\begin{dfn}
  \label{dfn:4}
  Given a bifibration $\E$ over a category $\B$, a commuting square
  in $\B$ is sent by $\E$ to a square in $\Cat$ that is filled by an
  isomorphism and each of whose sides has a left adjoint.
  \begin{equation*}
    \vcenter{
      \xymatrix{
        Y \ar[r] \ar[d] & W \ar[d] \\
        X \ar[r] & Z
      }
    }
    \qquad {\Large \mapsto} \qquad
    \vcenter{
      \xymatrix{
        \E Y \ar@{}[dr]|{\iso} & \E W \ar[l] \\
        \E X \ar[u] & \E Z \ar[l] \ar[u]
      }
    }
  \end{equation*}
  Because this 2-cell is invertible, there are two possible mates that
  could be taken: we say that \emph{the Beck--Chevalley condition
    holds} for the square in $\B$ if \emph{both} of the mates of its
  image are invertible.
\end{dfn}

\subsection{Monads and modules}
\label{sec:monads-modules}

We review the notions of monads and modules in a 2-category.  For
background, etc.~see \cite{kelly74:_review,street72}.  This material
is classical, but the presentation of modules in terms of a canonical
distributive law seems to be new.  It gives a pleasing
characterization of the Kleisli and Eilenberg--Moore completions of a
2-category $\bK$ as being locally full in the 2-category
$\mathrm{Mod}(\bK)$ on the left- or right-free modules.

\begin{dfn}
  \label{dfn:9}
  A \emph{monad} in a 2-category $\bK$ is given by a morphism $t
  \colon x \to x$ together with 2-cells $\mu \colon t t \tc t$ and
  $\eta \colon 1 \tc t$ that make $t$ a monoid in the monoidal
  category $\bK(x,x)$.  A \emph{comonad} in $\bK$ is a monad in
  $\bK^\co$.
\end{dfn}

Given a monad $t \colon x \to x$ in $\bK$ and any other object $y \in
\bK$, there is a monad $t^* = \bK(t,y) \colon \bK(x,y) \to \bK(x,y)$
in $\Cat$ given by pre-composition with $t$, and of course there is
also a post-composition monad $t_* = \bK(z,t)$ on $\bK(z,x)$ for any
object $z$.

\begin{dfn}
  \label{dfn:10}
  If $t \colon x \to x$ is a monad in $\bK$, then a \emph{left
    $t$-module} is an algebra in the usual sense for (one of) the
  monad(s) $t_*$: it is given by an object $z$, a morphism $a \colon z
  \to x$ and a 2-cell $\alpha \colon t a \tc a$ satisfying the
  appropriate identities.  A \emph{right $t$-module} is a
  $t^*$-algebra.  The category $\mathrm{LMod}(t,z)$ is the category of
  algebras for the monad $\bK(z,t)$; the category $\mathrm{RMod}(t,y)$
  is the category of $\bK(t,y)$-algebras.

  Left and right \emph{comodules} for a comonad are defined
  analogously.

  Given monads $t$ and $s$ on $x$ and $y$ respectively, the associator
  of $\bK$ gives rise to an invertible distributive law $t^* s_* \iso
  s_* t^*$, so that both of these composites are themselves monads on
  $\bK(x,y)$ with equivalent categories of algebras.  The category
  $\mathrm{Mod}(t,s)$ of \emph{bimodules from $t$ to $s$} is the
  category of algebras for this composite monad, which we will call
  $\bK(t,s)$.  Objects of this category will be written thus: $m
  \colon t \prof s$.
\end{dfn}

Standard facts about distributive laws \cite{beck69:_distr_laws} then
show that there is a commuting square of monadic functors:
\begin{equation*}
  \xymatrix{
    &
    *\txt{$\mathrm{Mod}(t,s)$ \\ $\eqv \xalg{\bK(t,s)}$}
    \ar[dl]_{U_{\hat{s_*}}} \ar[dr]^{U_{\hat{t^*}}} \\
    *\txt{$\mathrm{RMod}(t,y)$\\ $\eqv \xalg{\bK(t,y)}$}
    \ar[dr]_{U_{t^*}} & &
    *\txt{$\mathrm{LMod}(s,x)$\\ $\eqv \xalg{\bK(x,s)}$}
    \ar[dl]^{U_{s_*}} \\
    & \bK(x,y)
  }
\end{equation*}
and that each of $\bK(t,y)$ and $\bK(x,s)$ canonically induces a monad
on the other's category of algebras called $\hat{t^*}$ and $\hat{s_*}$
respectively.

\begin{dfn}
  \label{dfn:21}
  Given a monad $t \colon x \to x$ in a 2-category $\bK$, the
  \emph{Eilenberg--Moore object} $x^t$ of $t$ is, if it exists, the
  universal left $t$-module, i.e.~a representation of the functor $y
  \mapsto \mathrm{LMod}(t,y)$.  In more concrete terms, the EM object
  comes equipped with the structure of a left $t$-module $u^t \colon
  x^t \to x$, composition with which sets up an equivalence
  $\bK(y,x^t) \eqv \mathrm{LMod}(t,y)$.

  The \emph{Kleisli object} $x_t$ of $t$ is the universal right
  $t$-module: it comes with a right $t$-module structure $f_t \colon x
  \to x_t$ that mediates an equivalence $\bK(x_t,z) \eqv
  \mathrm{RMod}(t,z)$.  Equivalently, it is the EM object of $t$
  considered as a monad in $\bK^\op$.

  The \emph{co-Kleisli} and \emph{co-Eilenberg--Moore} objects of a
  comonad in $\bK$ are its Kleisli and EM objects in $\bK^\co$.  (The
  co- prefix will be omitted where it is unnecessary.)
\end{dfn}

Eilenberg--Moore objects are weighted limits \cite{street76:_limit},
and so Kleisli objects are colimits.  The theory of completions under
colimits is well understood, and leads in this case to the following.

\begin{dfn}
  \label{dfn:24}
  The \emph{Kleisli completion} $\mathrm{Kl}(\bK)$
  \cite{lack02:ftm_ii} of a 2-category $\bK$ is the full
  sub-2-category of $[\bK^\op, \Cat]$ on those functors that are
  Kleisli objects of monads on representable functors.  It is
  convenient to take the objects of $\mathrm{Kl}(\bK)$ to be the
  monads in $\bK$ themselves.

  The \emph{Eilenberg--Moore completion} $\mathrm{EM}(\bK)$ of $\bK$
  is $\mathrm{Kl}(\bK^\op)^\op$.

  The \emph{co-Kleisli} and \emph{co-Eilenberg--Moore} completions of
  $\bK$ are then $\mathrm{Kl}^\co(\bK) = \mathrm{Kl}(\bK^\co)^\co$ and
  $\mathrm{EM}^\co(\bK) = \mathrm{EM}(\bK^\co)^\co$.
\end{dfn}

We may follow \cite{lack02:ftm_ii} and give a more hands-on
description of the Kleisli completion: if $t \colon x \to x$ and $s
\colon y \to y$ are monads in $\bK$ as above, then morphisms $t \to s$
in $\mathrm{Kl}(\bK)$ are transformations $\bK(1,x)_{\bar t} \tc
\bK(1,y)_{\bar s}$ in $[\bK^\op, \Cat]$ between the Kleisli objects of
$\bar t = \bK(1,t)$ and $\bar s = \bK(1,s)$.  The universal property
of the domain makes the category of these equivalent to
$\mathrm{RMod}(\bar t, \bK(1,y)_{\bar s})$.  The Yoneda lemma then
shows that this category is in turn equivalent to
$\mathrm{RMod}(\hat{t^*}, \bK(x,y)_{s_*})$ --- that is,
$\mathrm{Kl}(\bK)(t,s)$ is the category of algebras for the monad
induced by $t$ on the Kleisli category of the monad $\bK(x,s)$.  This
can be expanded in two ways, corresponding to the two different
constructions of the Kleisli category: the first takes
$\bK(x,y)_{s_*}$ to be the full subcategory of $\bK(x,y)^{s_*} =
\mathrm{LMod}(s,x)$ on the free $s$-modules, and the monad induced by
$t$ to be simply precomposition with $t$.  Note that this makes
$\mathrm{Kl}(\bK)(t,s)$ equivalent to the full subcategory of
$\mathrm{Mod}(t,s)$ on the `left-free bimodules' from $t$ to $s$
(cf.~\cite[p.~166]{wood85:_proar_ii}).  The second description of
$\mathrm{Kl}(\bK)(t,s)$ uses the direct presentation of the Kleisli
category \cite[VI.5]{mac98:_categ_workin_mathem}: $\bK(x,y)_{s_*}$ has
objects the morphisms $x \to y$ in $\bK$, and as morphisms $a \tc b
\colon x \to y$ the 2-cells $a \tc sb$ in $\bK$, with identities and
composition given by $\eta^s$ and the usual Kleisli composition.  Then
the monad induced by $t$ is given by pre-(Kleisli-)composition with
$\eta^s t$, but the monad axioms make this the same as precomposing
the underlying $\bK$-morphism with $t$.  Working everything out as in
\cite{lack02:ftm_ii}, we see that a Kleisli morphism from $t$ to $s$
is given by a morphism $a \colon x \to y$ in $\bK$ together with a
2-cell $\alpha \colon at \tc sa$ that satisfies:
\begin{equation*}
  \xymatrix{
    a t t \ar[d]_{a \mu} \ar[r]^{\alpha t} & s a t \ar[r]^{s \alpha} &
    s s a \ar[d]^{\mu a} \\
    a t \ar[rr]_{\alpha} & & s a
  }
  \qquad
  \xymatrix{
    a \ar[r]^{a \eta} \ar[dr]_{\eta a} & a t \ar[d]^{\alpha} \\
    & s a
  }
\end{equation*}
and a 2-cell $a \tc b$ is given by a 2-cell $\phi \colon a \tc s b$ of
$\bK$ satisfying:
\begin{equation*}
  \xymatrix{
    a t \ar[r]^{\alpha} \ar[d]_{\phi t} & s a \ar[r]^{s \phi} \ar[r] &
    s s b \ar[d]^{\mu b} \\
    s b t \ar[r]_{s \beta} & s s b \ar[r]_{\mu b} & s b
  }
\end{equation*}
Notice that a Kleisli morphism $t \to s$ is precisely a `monad
op-functor' from $t$ to $s$ in the sense of \cite{street72}.  Such a
morphism determines and is determined by
\cite[section~2.1]{lack02:ftm_ii} an essentially-commuting square
\begin{equation*}
  \xymatrix{
    \bK(1,x) \ar[r] \ar[d] \ar@{}[dr]|{\iso} & \bK(1,x)_{\bK(1,t)}
    \ar[d] \\
    \bK(1,y) \ar[r] & \bK(1,y)_{\bK(1,s)}
  }
\end{equation*}
Similarly, a `monad opfunctor transformation' is precisely a `free'
Kleisli 2-cell, i.e.~one of the form $\eta \hcmp \phi' \colon a \tc b
\tc sb$, hence a commuting cylinder of the following form:
\begin{equation*}
  \xymatrix{
    \bK(1,x)
    \ar@/_2ex/[d]^{}="s1" \ar@/^2ex/[d]_{}="t1"
    \ar@{}"s1";"t1"|{\Leftarrow}
    \ar[r]
    & \bK(1,x)_{\bK(1,t)}
    \ar@/_2ex/[d]^{}="s2" \ar@/^2ex/[d]_{}="t2"
    \ar@{}"s2";"t2"|{\Leftarrow} \\
    \bK(1,y) \ar[r] & \bK(1,y)_{\bK(1,s)}
  }
\end{equation*}

\begin{dfn}
  \label{dfn:25}
  Given modules $m \colon t \prof s$ and $n \colon s \prof r$, where
  $t,s,r$ are monads on $x,y,z$ respectively, the \emph{composite} $n
  \hcmp m$ is given by the following coequalizer in $\bK(x,z)$
  \cite[p.~165]{wood85:_proar_ii}, \cite[4.1]{carboni87}:
  \begin{equation*}
    \xymatrix{
      n s m \ar@<+.5ex>[r] \ar@<-.5ex>[r] & n m \ar[r] & n \hcmp m
    }
  \end{equation*}
  where the parallel morphisms are the actions of $s$ on $n$ and $m$
  and their codomain is the composite $n m$ in $\bK$.  This is a
  reflexive coequalizer, with section $n \eta m$.  If it is preserved by
  $\bK(t,r)$ then it is reflected by the monadic functor
  $\mathrm{Mod}(\bK)(t,r) \to \bK(x,z)$
  \cite[prop.~4.3.2]{borceux94:_handb_categ_algeb_vol_2} and so $n
  \hcmp m$ is a module from $t$ to $r$.

  If the hom-categories of $\bK$ admit all such coequalizers, and if
  these are preserved by composition on either side, then this formula
  defines the composition operation of a 2-category
  $\mathrm{Mod}(\bK)$.  The identity on a monad $t$ is $t$ equipped
  with its left and right self-actions, and the 2-category axioms
  follow from the universal property of the coequalizer above.  (Note
  that $\mathrm{Mod}(\bK)$ will almost never be strict even if $\bK$
  is, because composition depends on a choice of colimit.)
\end{dfn}

Observe that if $m$ in the above is a left-free module (i.e.~a
morphism in $\mathrm{Kl}(\bK)$) then we may write $m = s m'$, and the
pair to be coequalized is then
\begin{equation*}
  \xymatrix{
    n s s m' \ar@<+.5ex>[r]^{\nu s m'} \ar@<-.5ex>[r]_{n \mu m'}
    & n s m' 
  }
\end{equation*}
where $\nu$ is the action of $s$ on $n$ and $\mu$ is the
multiplication of $s$.  But because $\nu$ is an algebra for
$\bK(s,z)$, it is the coequalizing map
\cite[lemma~4.3.3]{borceux94:_handb_categ_algeb_vol_2} in
\begin{equation*}
  \xymatrix{
    n s s \ar@<+.5ex>[r]^{\nu s} \ar@<-.5ex>[r]_{n \mu}
    & n s \ar[r]^\nu & n
  }
\end{equation*}
Moreover, this is a split, hence absolute, coequalizer, so that
whiskering by $m'$ yields the composite $n \hcmp s m' \eqv n m'$, with
$\nu m'$ as the coequalizing map.  This fact ensures that
$\mathrm{Kl}(\bK)$ always exists, even if $\mathrm{Mod}(\bK)$ does
not.  It is also not hard to see that the former will be strict if
$\bK$ is, because composition of left-free modules can be taken to be
just composition in $\bK$.

Thus there are inclusions
\begin{equation*}
  \mathrm{Mnd}^\op(\bK) \to \mathrm{Kl}(\bK) \to
  \mathrm{Mod}(\bK) \leftarrow \mathrm{EM}(\bK) \leftarrow
  \mathrm{Mnd}(\bK)
\end{equation*}
which are all the identity on objects.  Here $\mathrm{Mnd}(\bK)$ is
the 2-category described by \cite{street72}, $\mathrm{Mnd}^\op(\bK)$
is $\mathrm{Mnd}(\bK^\op)^\op$, and $\mathrm{EM}(\bK)$ is the
Eilenberg--Moore completion of \cite{lack02:ftm_ii},
i.e.~$\mathrm{Kl}(\bK^\op)^\op$.

In chapter~4 we will consider structures called (proarrow)
\emph{equipments}, which have several definitions that we will try to
relate.  For the purposes of the following definition, we may take an
equipment to be given by a pair of 2-categories with the same objects
and a locally fully faithful identity-on-objects functor between them
\cite{wood82:proarrows_i}.

\begin{dfn}[{\cite{garner13:_enric}}]
  \label{dfn:32}
  Given an equipment $\bK \to \bM$ such that $\mathrm{Mod}(\bM)$
  exists as a 2-category, its \emph{Kleisli completion} is given by
  the functor
  \begin{equation*}
    \mathrm{Kl}(\bK \to \bM) = \mathrm{Kl}_{\bK}(\bM) \to
    \mathrm{Mod}(\bM)
  \end{equation*}
  whose domain is the locally full sub-2-category of
  $\mathrm{Kl}(\bM)$ on the morphisms whose underlying 1-cell in $\bM$
  is in the image of the functor $\bK \to \bM$.
\end{dfn}

\subsection{Pseudo-monads and monoidal bicategories}
\label{sec:pseudo-monads}

Recall \cite{gurski07} that $\Gray$ is the symmetric monoidal closed
category of strict 2-categories and strict 2-functors, with the
`pseudo' Gray tensor product as $\otimes$.  The hom is the right
adjoint $\mathrm{Ps}(\bicat{C},-)$ to $-\otimes \bicat{C}$, where
$\mathrm{Ps}(\bicat{C},\bicat{D})$ is the 2-category of strict
2-functors, \emph{pseudonatural} transformations and modifications.
Recall next that $\xcat{\Gray}$ is the usual category of categories
enriched in $\Gray$, and that every tricategory is equivalent to a
$\Gray$-category.  These are almost strict 3-categories, except that
the interchange law holds only up to coherent isomorphism.  So we may
pretend that our 3-categories are almost-strict in this sense.

\begin{dfn}
  A \emph{monoidal 2-category} $\bB$ is given by a 3-category $\Sigma
  \bB$ with a single object, which by the coherence theorem is
  essentially the same thing as a (strict) monoid object in $\Gray$.
\end{dfn}

\begin{dfn}[{\cite{day97:_monoid_hopf}}]
  \label{dfn:26}
  A \emph{monoidal functor} $(\bB,\otimes,i) \to (\bB',\boxtimes,i')$
  is given by an ordinary functor $F \colon \bB \to \bB'$ together
  with transformations $\mu, \eta$, with components
  \begin{equation*}
    \mu_{xy} \colon Fx \boxtimes Fy \to F(x\otimes y)
    \qquad \eta_{\ast} \colon i' \to Fi
  \end{equation*}
  and modifications $a,l,r$, with components (omitting tensor symbols)
  \begin{equation*}
    \xymatrix{
      F(x)F(y)F(z) \ar[r]^{\mu F(z)} \ar[d]_{F(x) \mu}
      \ar@{}[dr]|{a} & F(x y)F(z) \ar[d]^{\mu} \\
      F(x)F(y z) \ar[r]_{\mu} & F(x y z)
    }
    \qquad
    \xymatrix{
      F x \ar[r]^-{F(x) \eta} \ar[dr]_1^{}="s" & F(x) F(i)
      \ar[d]|{\mu} & F x \ar[l]_-{\eta F(x)} \ar[dl]^1_{}="s2" \\
      & F x \ar@{}"s";[u]|(.4){r} \ar@{}"s2";[u]|(.4){l}
    }
  \end{equation*}
  that satisfy
  \begin{multline}
    \label{eq:21}
    \xymatrix{
      F(x)F(y)F(z)F(w) \ar[r]^{\mu F(z) F(w)} \ar[d]_{F(x) F(y) \mu}
      \ar[dr]|{F(x) \mu F(w)}
      & F(xy)F(z)F(w) \ar[dr]^{\mu F(w)} \ar@{}[d]|{a F(w)}
      \\
      F(x)F(y)F(zw) \ar[dr]_{F(x) \mu} \ar@{}[r]|{F(x)a}
      & F(x)F(yz)F(w) \ar[r]^{\mu F(w)} \ar[d]|{F(x) \mu}
      & F(xyz)F(w) \ar[d]^\mu
      \\
      & F(x)F(yzw) \ar[r]_{\mu} \ar@{}[ur]|{a}
      & F(xyzw)
    }
    \\ = \\
    \xymatrix{
      F(x)F(y)F(z)F(w) \ar[r]^{\mu F(z) F(w)} \ar[d]_{F(x) F(y) \mu}
      \ar@{}[dr]|{\mu \mu}
      & F(xy)F(z)F(w) \ar[d]^{F(xy) \mu} \ar[dr]^{\mu F(w)}
      \\
      F(x)F(y)F(zw) \ar[r]^{\mu F(zw)} \ar[dr]_{F(x) \mu}
      & F(xy)F(zw) \ar[dr]|{\mu} \ar@{}[d]|{a} \ar@{}[r]|{a}
      & F(xyz)F(w) \ar[d]^\mu
      \\
      & F(x)F(zyw) \ar[r]_\mu
      & F(xyzw)
    }
  \end{multline}
  where $\mu \mu$ is the relevant interchange isomorphism, and
  \begin{multline}
    \label{eq:1}
    \xymatrix{
      F(x)F(y) \ar[dr]^{F(x) \eta F(y)} \\
      & F(x)F(i)F(y) \ar[r]^{\mu F(y)} \ar[d]_{F(x) \mu} \ar@{}[dr]|{a}
      & F(x)F(y) \ar[d]^{\mu} \\
      & F(x)F(y) \ar[r]_{\mu} & F(xy)
    }
    \\ = \\
    \xymatrix{
      F(x)F(y) \ar[r]^{F(x) \eta F(y)} \ar[d]_{F(x) \eta F(y)}
      \ar[dr]|{1}="x" \ar@{}"x";[r]|(.4){r F(y)}
      \ar@{}"x";[d]|(.4){F(x) l}
      & F(x)F(i)F(y) \ar[d]^{\mu F(y)} \\
      F(x)F(i)F(y) \ar[r]_{F(x) \mu} & F(x)F(y) \ar[dr]^\mu \\
      & & F(xy)
    }
  \end{multline}
\end{dfn}

\begin{dfn}
  \label{dfn:29}
  A (pseudo-)\emph{monoid} in a monoidal 2-category $\bB$ is a
  monoidal functor $\mathbf{1} \to \bB$, and the 2-category
  $\mathrm{PsMon}(\bB)$ is defined as in
  \cite{mccrudden00:_balan,day97:_monoid_hopf} to be the 2-category of
  monoidal functors, transformations and modifications from
  $\mathbf{1}$ to $\bB$.  If $F \colon \bB \to \bB'$ is a monoidal
  2-functor then $\mathrm{PsMon}(F) \colon \mathrm{PsMon}(\bB) \to
  \mathrm{PsMon}(\bB')$ is a 2-functor
  \cite[section~2]{mccrudden00:_balan}.

  A (pseudo-)\emph{monad} \cite{marmolejo99:_distr_laws_for_pseud} in
  a 3-category $\tricat{T}$ is given by an object $x \in \tricat{T}$
  and a pseudomonoid in $\tricat{T}(x,x)$.
\end{dfn}

\begin{dfn}[{\cite{marmolejo99:_distr_laws_for_pseud}}]
  \label{dfn:34}
  Given a pseudomonad $T$ on a 2-category $\bK$, its 2-category
  \mbox{$T$-$\bicat{Alg}$} of \emph{algebras} is given by the
  following:
  \begin{itemize}
  \item An object is an object $x$ of $\bK$ and a morphism $a \colon T
    x \to x$ together with invertible 2-cells
    \begin{equation*}
      \xymatrix{
        T^2 x \ar[r]^{\mu x} \ar[d]_{T a}
        \ar@{}[dr]|{\alpha} & T x \ar[d]^{a} \\
        T x \ar[r]_{a} & x
      }
      \qquad
      \xymatrix{
        x \ar[r]^-{\eta x} \ar[dr]_1^{}="s" & T x
        \ar[d]^{a} \\
        & F x \ar@{}"s";[u]|(.4){\upsilon}
      }
    \end{equation*}
    that satisfy equations (6) and (7) of
    \cite{marmolejo99:_distr_laws_for_pseud} (these are much the same
    as equations~\eqref{eq:21} and \eqref{eq:1} above, altered in the
    only way that makes sense).
  \item A morphism $(x,a) \to (y,b)$ is given by a morphism $f \colon
    x \to y$ and an invertible 2-cell
    \begin{equation*}
      \xymatrix{
        T x \ar[r]^{T f} \ar[d]_a \ar@{}[dr]|\phi & T y \ar[d]^b \\
        x \ar[r]_f & y
      }
    \end{equation*}
    that satisfies equations (9) and (10) of
    \cite{marmolejo99:_distr_laws_for_pseud}.
  \item A 2-cell between algebra morphisms is given by a 2-cell
    between the undelying morphisms in $\bK$ that makes the evident
    `cylinder' commute.
  \end{itemize}
\end{dfn}

Now if $t \colon x \to x$ is a monad in a 3-category $\tricat{T}$, we
can define the 2-categories of left and right $t$-modules as
\begin{equation*}
  \mathrm{LMod}(t,y) = \mbox{$\tricat{T}(y,t)$-$\bicat{Alg}$}
  \qquad
  \mathrm{RMod}(t,z) = \mbox{$\tricat{T}(t,z)$-$\bicat{Alg}$}
\end{equation*}
These assignments are functorial in the objects $y$ and $z$
[\textit{op.~cit.}], and we may define Eilenberg--Moore and Kleisli
objects as representations of these functors just as before.
Similarly, the Kleisli completion of a 3-category $\tricat{T}$ can be
defined as the full sub-3-category of $[\tricat{T}^\op, \Bicat]$ on
the Kleisli objects of representable monads.  From this we could,
following the reasoning of the previous section, define monad
morphisms, EM/Kleisli 2- and 3-cells, and so on.  We won't do that
fully here, but we will touch on the matter again in
section~\ref{sec:defin-equipm}.

\begin{rem}[{cf.~\cite[section~2]{kelly97:_proper_like_struc}}]
  \label{rem:9}
  If $T$ is a pseudo-monad on and $A$ an object of $\bK$, and if
  evaluation at $A$ has a right adjoint $\langle -, A \rangle \colon
  \bK \to [\bK,\bK]$, then in the equivalence
  \begin{prooftree}
    \AXC{$T A \longrightarrow A$}
    \UIC{$T \longrightarrow \langle A,A \rangle$}
  \end{prooftree}
  a morphism above is a $T$-algebra if and only if its transpose below
  is a morphism of pseudo-monoids.  The proof is simply a matter of
  unwinding the definitions and using the $T$-algebra and adjunction
  axioms.
\end{rem}

\section{Limits and colimits}
\label{sec:limits-colimits}

In this section we treat the material on colimits in $\Cat$ that will
be needed in order to define the 3-category of 2-profunctors in the
next chapter.  Everything from section~\ref{sec:2-extranaturality}
onwards is original, except where noted.

\subsection{Representables and colimits}
\label{sec:representables}

Some notation:

\begin{dfn}
  If $H \colon \bL^\op \times \bK \to \Cat$ is a (pro)functor, and $F
  \colon \bK' \to \bK$ and $G \colon \bL' \to \bL$, then we write
  \begin{equation*}
    H(G,F) \quad = \quad \xymatrix@C+1pc{ \bL'^\op \times \bK'
      \ar[r]^{G^\op \times F} & \bL^\op \times \bK \ar[r]^-{H} & \Cat}
  \end{equation*}
  A profunctor of the form $H(1,F)$ or $H(G,1)$ is called
  \emph{representable} or \emph{corepresentable}.
\end{dfn}

\begin{dfn}
  \label{dfn:19}
  If $H \colon \bL^\op \times \bK \to \Cat$ is a (pro)functor, then we
  write an object $h \in H(\ell,k)$ as $h \colon \ell \prof k$, and
  call it a \emph{heteromorphism} from $\ell$ to $k$.  We also write
  the action of morphisms of $\bK$ and $\bL$ on $h$ as e.g.~$k'
  \xrightarrow{f} k \xslashedrightarrow{h} \ell \xrightarrow{g}
  \ell'$.
\end{dfn}

The usual generalities hold, up to the expected level of weakness, for
representable 2-functors.  In particular, there is a 2-categorical
Yoneda lemma.

\begin{prp}[{\cite[1.9, 1.11]{street80:_fibrat_in_bicat}}]
  \label{prp:2}
  For a 2-category $\bK$ and a functor $F \colon \bK \to \Cat$,
  there is an equivalence
  \begin{equation*}
    [\bK, \Cat](\bK(k,-), F) \simeq Fk
  \end{equation*}
  If $G \colon \bK \to \Cat$ is representable as $G \simeq \bK(k,-)$
  then there is an object $x \in Gk$ such that
  \begin{enumerate}
  \item \label{item:1} For any $y \in Gj$, there is a morphism $f
    \colon k \to j$ in $\bK$ and an isomorphism $f_*x = Gf(x) \iso y$.
  \item \label{item:2} For any $g,h \colon k \to j$ in $\bK$ and $a
    \colon g_*k \to h_*k$ in $Gj$, there is a \emph{unique} 2-cell $q
    \colon g \tc h$ in $\bK$ such that $q_*k = a$.
  \end{enumerate}
\end{prp}

\begin{cor}
  \label{cor:5}
  The Yoneda embedding $\bK \to [\bK^\op, \Cat]$ is 2-fully-faithful,
  that is, locally fully faithful and essentially surjective on
  morphisms.
\end{cor}

\begin{rem}
  \label{rem:5}
  One consequence of this is the following: suppose $H \colon \bL^\op
  \times \bK \to \Cat$ is a functor, and that for each $\ell \in \bL$
  there is an object $\bar H \ell \in \bK$ and a representation
  $H(\ell,-) \eqv \bK(\bar H \ell, -)$, i.e~that the corresponding
  functor $\bL^\op \to [\bK, \Cat]$ takes its values in
  representables.  Then there is an essentially unique way to make
  $\bar H$ into a functor $\bL \to \bK$, and $H$ is then equivalent to
  $\bK(\bar H, 1)$: if $g \colon \ell \to m$ is a morphism of $\bL$,
  then by Yoneda there are universal objects $\bar \ell \in H(\ell,
  \bar H \ell)$ and $\bar m \in H(m, \bar H m)$, so that $g^*\bar m$
  induces, by property \ref{item:1} of the proposition, a morphism
  $\bar H g \colon \bar H \ell \to \bar H m$ such that $(\bar H
  g)_*\bar\ell \iso g^*\bar m$.  The comparison maps of $\bar H$ arise
  from the 2-cells given by property \ref{item:2} in the proposition
  above, and their uniqueness implies their coherence.
\end{rem}

Of course, the dual property also holds, and we may sum up the two as
follows:

\begin{cor}
  \label{cor:4}
  A pointwise (co)representable profunctor is (co)representable.
\end{cor}

A 2-category, strict or not, may have colimits of varied strictness.
We will be mainly concerned with the weakest sort.

\begin{dfn}[{\cite[1.12]{street80:_fibrat_in_bicat}}]
  \label{dfn:8}
  Let $F \colon \bJ \to \bK$ and $W \colon \bJ^\op \to \Cat$ be
  functors.  If the functor $[\bJ^\op,\Cat](W,\bK(F,1))$ is
  representable as
  \begin{equation*}
    \bK(W\star F, 1) \simeq [\bJ^\op,\Cat](W,\bK(F,1))
  \end{equation*}
  then we call the representing object $W \star F$ the
  \emph{2-colimit} (or just the \emph{colimit}) of $F$ \emph{weighted
    by} $W$.  (This is known in much of the literature as a
  \emph{bicolimit}.)  The \emph{conical} colimit $\colim F$ of $F$ is
  $\Delta\mathbf{1} \star F$, where $\Delta\mathbf{1}$ is the constant
  functor at the terminal category.
\end{dfn}

\begin{rem}
  From this we may immediately derive two dual forms of the Yoneda
  lemma:
  \begin{equation}
    \label{eq:17}
    \{ \bJ(j,1), F \} \simeq F j \qquad \bJ(1,j) \star F \simeq F j
  \end{equation}
  (Cf.~\cite[(3.10)]{kelly82}.)
\end{rem}

Stricter kinds of colimit are useful in constructing the above sort.
If $\bK$ is a strict 2-category, then the \emph{pseudo-colimit}
\cite[1.14]{street80:_fibrat_in_bicat} of the functors in the
definition is representable via an \emph{isomorphism}
\begin{equation*}
  \bK(W \star_{\mathrm{ps}} F,k) \iso [\bJ^\op,\Cat](W,\bK(F,k))
\end{equation*}
If both $\bK$ and $\bJ$ are strict 2-categories, and the functors $F$
and $W$ are strict too, then we can consider the \emph{strict
  pseudo-colimit} \cite[6.10]{lack10}, which satisfies the property of
the pseudocolimit with the 2-category $[\bJ^\op,\Cat]$ replaced by the
2-category $\mathrm{Ps}(\bJ^\op,\Cat)$ of strict functors,
pseudonatural transformations and modifications.  The \emph{strict
  colimit} is the same, except that now the functor 2-category
involved is that of strict functors, strict transformations and
modifications.  This last is the $\cat{Cat}$-colimit, in the usual
enriched sense.

Pseudo-colimits, strict ones in particular, are thus \textit{a
  fortiori} 2-colimits, and moreover strict pseudo-colimits are strict
colimits whose weights are suitably `cofibrant' \cite[6.10]{lack10}.
Further, if $\bK$ is a strict 2-category, then for any 2-category
$\bJ$ there is a strict $\bJ'$ such that $[\bJ,\bK] \sim
\mathrm{Ps}(\bJ', \bK)$ \cite[6.12]{lack10}, and so for a diagram $F
\colon \bJ \to \bK$ and a weight $W \colon \bJ^\op \to \Cat$, there
are strict functors $F' \colon \bJ' \to \bK$ and $W' \colon \bJ'^\op
\to \Cat$ such that
\begin{equation*}
  [\bJ^\op, \Cat](W, \bK(F,1)) \eqv \mathrm{Ps}(\bJ'^\op, \Cat)(W',
  \bK(F',1))
\end{equation*}
That is, the strict pseudo-colimit of the strictified functors is
equivalent to the 2-colimit of the originals.  So a strict 2-category
that has all strict (i.e.~$\cat{Cat}$-weighted) colimits also has all
strict pseudo-colimits and hence all 2-colimits.  In particular,
$\Cat$ is strictly 2-cocomplete and so is 2-cocomplete.

\subsection{2-extranaturality}
\label{sec:2-extranaturality}

We will need to talk about 2-categorical ends and coends.

\begin{dfn}
  Let $T \colon \bK^\op \times \bK \to \bL$ be a functor and $\ell$ be
  an object of $\bL$.  A family $\beta_k \colon \ell \to T(k,k)$ is
  \emph{extranatural} (in $k$) if for each $f \colon k \to j$ in $\bK$
  there is an invertible 2-cell $\beta_f$
  \begin{equation*}
    \xymatrix{
      \ell \ar[r]^{\beta_k} \ar[d]_{\beta_j} \ar@{}[dr]|{\iso
        \mathrlap{\,\beta_f}} &
      T(k,k) \ar[d]^{T(k,f)} \\
      T(j,j) \ar[r]_{T(f,j)} & T(k,j)
    }
  \end{equation*}
  satisfying the following (fairly obvious) axioms:
  \begin{enumerate}
  \item $\beta_f$ is natural in $f$: for $p \colon f \tc g$
    \begin{equation*}
      \vcenter{
        \xymatrix{
          \ell \ar[r]^{\beta_k} \ar[d]_{\beta_j}
          \ar@{}[dr]|{\mathllap{\beta_g\,} \iso} &
          T(k,k) \ar@/_1.25pc/[d]_(.35){T(k,g)}^{}="t"
          \ar@/^1.25pc/[d]^{T(k,f)}_{}="s"
          \ar@{=>}"s";"t"_{T(k,p)}\\
          T(j,j) \ar[r]_{T(g,j)} & T(k,j)
        }
      } \qquad = \qquad
      \vcenter{
        \xymatrix{
          \ell \ar[r]^{\beta_k} \ar[d]_{\beta_j} \ar@{}[dr]|{\iso
            \mathrlap{\,\beta_f}} &
          T(k,k) \ar[d]^{T(k,f)} \\
          T(j,j) \ar[r]^{T(f,j)}_{}="s" \ar@/_2pc/[r]_{T(g,j)}^{}="t"
          \ar@{=>}"s";"t"|{T(p,j)}
          & T(k,j)
        }
      }
    \end{equation*}
  \item $\beta_{1_k} = 1_{\beta_k}$, modulo the unitors of $T$ and
    $\bL$:
    \begin{equation*}
      \vcenter{
        \xymatrix{
          \ell \ar[r]^{\beta_k} \ar[d]_{\beta_k} \ar@{}[dr]|{\iso
            \mathrlap{\,\beta_{1_k}}} &
          T(k,k) \ar[d]_(.25){T(k,1)}^{}="x"
          \ar@/^1.25pc/[d]^{1}_{}="y"
          \ar@{}"x";"y"|{\iso} \\
          T(k,k) \ar[r]^{T(1,k)}_{}="a"
          \ar@/_1.25pc/[r]_{1}^{}="b"
          \ar@{}"a";"b"|{\iso}
          & T(k,k)
        }
      }
      \qquad = \qquad
      \vcenter{
        \xymatrix{
          \ell \ar[r]^{\beta_k} \ar[d]_{\beta_k} \ar@{}[dr]|{=
            \mathrlap{\,1_{1 \beta_k}}} & 
          T(k,k) \ar[d]^1 \\
          T(k,k) \ar[r]_1 & T(k,k)
        }
      }
    \end{equation*}
  \item $\beta$ respects composition: for $gf \colon k \to j \to i$
    \begin{equation*}
      \vcenter{
        \xymatrix{
          \ell \ar[r]^{\beta_k} \ar[d]_{\beta_i} \ar@{}[dr]|{\iso
            \mathrlap{\,\beta_{gf}}} &
          T(k,k) \ar[d]^{T(k,gf)} \\
          T(i,i) \ar[r]_{T(gf,j)} & T(k,i)
        }
      }
      \quad = \quad
      \vcenter{
        \xymatrix{
          \ell \ar[dd]_{\beta_i} \ar[rr]^{\beta_k} \ar[dr]|{\beta_j} &
          & T(k,k) \ar[d]^{T(k,f)} \ar@{}[dl]|{\mathllap{\beta_f\,}
            \iso} \\
          & T(j,j) \ar[r]^{T(f,j)} \ar[d]_(.25){T(j,g)}
          \ar[dr]|{T(f,g)} \ar@{}[dl]|{\iso \mathrlap{\,\beta_g}} &
          T(k,j) \ar[d]^{T(k,g)} \\
          T(i,i) \ar[r]_{T(g,i)} & T(j,i) \ar[r]_{T(f,i)} & T(k,i)
        }
      }
    \end{equation*}
    where the triangles at the lower right contain the obvious
    compositors of $T$ and ($T$ applied to) the unitors of $\bK^\op
    \times \bK$, and the compositors of $T$ on the boundaries are left
    implicit.
  \end{enumerate}
  There is an obvious notion of modification between two extranatural
  transformations, i.e.~an object-indexed family of 2-cells that
  commute in the evident way with the naturality 2-cells, so we get a
  category $\mathrm{Exnat}(\ell,T)$.
\end{dfn}

\begin{lem}
  If $T,S \colon \bK^\op \times \bK \to \bL$ are functors, $\alpha
  \colon T \tc S$ is a transformation and the family $\beta_k \colon
  \ell \to T(k,k)$ is extranatural in $k$, then the family
  $\alpha_{k,k} \beta_k$ is again extranatural.
\end{lem}
\begin{proof}
  For $f \colon k \to j$ as before, the structure 2-cell $(\alpha
  \beta)_f$ is
  \begin{equation*}
    \xymatrix{
      \ell \ar@{}[dr]|{\iso\mathrlap{\,\beta_f}} \ar[r]^{\beta_k}
      \ar[d]_{\beta_j} &
      T(k,k) \ar[r]^{\alpha_{k,k}} \ar[d]^(.7){T(k,f)} & S(k,k)
      \ar[dd]^{S(k,f)} \\ 
      T(j,j) \ar[r]^{T(f,j)} \ar[d]_{\alpha_{j,j}} & T(k,j)
      \ar@{}[dl]|{\iso\mathrlap{\,\alpha_{f,j}}} 
      \ar@{}[ur]|{\iso\mathrlap{\,\alpha_{k,f}}}
      \ar[dr]^{\alpha_{k,j}} \\
      S(j,j) \ar[rr]_{S(f,j)} & & S(k,j)
    }
  \end{equation*}
  The naturality and unit axioms follow more or less obviously from
  this and the corresponding properties of $\beta$ and $\alpha$.

  Using the definition above, and the composition axioms for $\beta$
  and $\alpha$, we may expand $(\alpha\beta)_{gf}$ to
  \begin{equation*}
    \xymatrix{
      \ell \ar[rr] \ar[dr] \ar[dd]
      &
      & T(k,k) \ar[r] \ar[d]
      \ar@{}[dr]|{\alpha_{k,f}}
      & S(k,k) \ar[d] \\
      & T(j,j) \ar[r] \ar[dr] \ar[d] \ar@{}[ur]|{\beta_f}
      & T(k,j) \ar[r] \ar[d]
      & S(k,j) \ar[dd] \\
      T(i,i) \ar[r] \ar[d] \ar@{}[ur]|{\beta_g}
      & T(j,i) \ar[r] \ar[d]
      & T(k,i) \ar[dr] \ar@{}[ur]|{\alpha_{k,g}} \\
      S(i,i) \ar[r] \ar@{}[ur]|{\alpha_{g,i}}
      & S(j,i) \ar[rr] \ar@{}[ur]|{\alpha_{f,i}}
      &
      & S(k,i)
    }
  \end{equation*}
  which is to equal
  \begin{equation*}
    \xymatrix{
      \ell \ar[rr] \ar[dr] \ar[dd]
      & 
      & T(k,k) \ar[r] \ar[d]
      & S(k,k) \ar[dd] \\
      & T(j,j) \ar[r] \ar[dr] \ar[d] \ar@{}[ur]|{\beta_f}
      & T(k,j) \ar[dr] \ar@{}[ur]|{\alpha_{k,f}} \\
      T(i,i) \ar[r] \ar[d] \ar@{}[ur]|{\beta_g}
      & T(j,i) \ar[dr] \ar@{}[r]|{\alpha_{j,g}}
      & S(j,j) \ar[r] \ar[dr] \ar[d] \ar@{}[u]|{\alpha_{f,j}}
      & S(k,j) \ar[d] \\
      S(i,i) \ar[rr] \ar@{}[ur]|{\alpha_{g,i}}
      &
      & S(j,i) \ar[r]
      & S(k,i)
    }
  \end{equation*}
  So it suffices to show that the composite 2-cells, $\gamma$ and
  $\delta$, say, from
  \begin{equation*}
    T(j,j) \to T(k,j) \to S(k,j) \to S(k,i)
  \end{equation*}
  to
  \begin{equation*}
    T(j,j) \to T(j,i) \to S(j,i) \to S(k,i)
  \end{equation*}
  in the lower right-hand corners are equal.  The diagonals of these
  are the boundaries of $\alpha_{f,g}$, and by gluing $\gamma$ and
  $\delta$ together along this and their boundaries we get two `cones'
  of 2-cells, one relating $\alpha_{f,g}$ to $\alpha_{j,g}$ and
  $\alpha_{f,i}$ and the other relating it to $\alpha_{f,j}$ and
  $\alpha_{k,g}$, whose commutativity implies that $\gamma$ and
  $\delta$ are equal.  Because $(f,i)(j,g) \iso (f,g) \iso (k,g)(f,j)$
  in $\bK^\op \times\bK$, the naturality and composition axioms for
  $\alpha_{f,g}$ show that these cones do indeed commute, so that
  $\alpha\beta$ satisfies the composition axiom and hence is
  extranatural.
\end{proof}

If $\bK$ is a 2-category, then the family $1_k \colon \mathbf{1} \to
\bK(k,k)$ is extranatural in $k$, and in fact this is the universal
extranatural transformation out of $\mathbf{1}$, in the following
sense.

\begin{prp}[Extranatural Yoneda]
  Let $H \colon \bK^\op \times \bK \to \Cat$ be a functor.  Then there
  is an equivalence of categories
  \begin{equation*}
    \mathrm{Exnat}(\mathbf{1},H) \eqv \Nat(\hom_{\bK},H)
  \end{equation*}
  given from right to left by composition with the extranatural $1
  \colon \mathbf{1} \xnt \bK$.
\end{prp}
\begin{proof}
  Given an extranatural $\beta \colon \mathbf{1} \xnt H$ we get a
  natural $\hat \beta$ with components $\hat \beta_{k,j} \colon f
  \mapsto H(k,f)(\beta_k)$ (we could equally choose the isomorphic
  $H(f,j)(\beta_j)$).  For morphisms $g,h$, the mediating 2-cell
  $\hat\beta_{g,h}$ comes from the unitors of $H$, the compositors of
  $\bK$, and $\beta_g$.  These are all suitably natural and thus so is
  $\hat\beta$.

  If $\alpha \colon \bK \tc H$ is natural, then the isomorphisms
  $(\alpha_{k,f})_{1_k}$ provide the components of an invertible
  modification $\hat{\alpha 1} \mdf \alpha$, which in fact is natural
  in $\alpha$.  The equivalence is completed by the fact that $\beta_k
  \iso H(k,1_k)(\beta_k)$.
\end{proof}

In enriched category theory, (co)ends are $\hom$-weighted (co)limits
\cite[section~3.10]{kelly82}.  Because in our setting $\Nat(F,G)$ is
equivalent to the $F$-weighted limit $\{\!F,G\}$ of $G$
\cite[prop.~1.15]{street80:_fibrat_in_bicat}, we have shown that
$\Cat$ admits `2-ends'
\begin{equation*}
 \int_k H(k,k) \eqv \mathrm{Exnat}(\mathbf{1},H).
\end{equation*}
If $T \colon \bK^\op \times \bK \to \bL$, we find that
\begin{equation*}
  \mathrm{Exnat}(\ell,T) \eqv \mathrm{Exnat}(\mathbf{1},\bL(\ell,T))
  \eqv \{\bL,\bL(\ell,T)\}
\end{equation*}
where we write $\bL$ for $\hom_{\bL}$.  If any of these is
representable as a functor of $\ell$ we may call the representing
object the end $\int_kT(k,k)$ of $T$.

The following result is immediate.
\begin{prp}
  \label{prp:1}
  If $F,G \colon \bK \to \bL$ are functors, then
  \begin{equation*}
    \Nat(F,G) \iso \mathrm{Exnat}(\mathbf{1},\bL(F-,G-))
  \end{equation*}
  and hence
  \begin{equation*}
    \Nat(F,G)\; \eqv \int_k \bL(Fk,Gk).
  \end{equation*}
\end{prp}
For strictly enriched categories this is a definition rather than a
theorem \cite[section~2.2]{kelly82}, but here it shows that our
definition of 2-ends is the right one.

Coends are dual: if $S \colon \bK^\op \times \bK \to \bL$ we have that
\begin{equation*}
  \mathrm{Exnat}(S, \ell) \eqv \mathrm{Exnat}(\mathbf{1},\bL(S, \ell))
  \eqv \{\bL, \bL(S,\ell)\}
\end{equation*}
and a representation of any of these may be called the coend $\int^k
S(k,k)$ of $S$.

\subsection{The free cocompletion of a 2-category}
\label{sec:free-cocompl-bicat}

We want to show now that if $\bK$ is a 2-category then its \emph{free
  2-cocompletion} is given by $\Pw\bK = [\bK^\op, \Cat]$.  We know that
$\Cat$ is cocomplete, and suspect that colimits in $\Pw\bK$ will be
calculated pointwise: let $F \colon \bJ \to \Pw\bK$ and $W \colon
\bJ^\op \to \Cat$ and set $(W \star F)k = W \star F(-,k)$.  Then, using
prop.~\ref{prp:1},
\begin{align*}
  \Pw\bK(W \star F, G)
  & \eqv \int_k \Cat(W \star F(-,k), Gk) \\
  & \eqv \int_k [\bJ^\op, \Cat](W, \Cat(F(-,k), Gk)) \\
  & \eqv [\bJ^\op, \Cat](W, \int_k \Cat(F(-,k), Gk)) \\
  & \eqv [\bJ^\op, \Cat](W, \Pw\bK(F-,G))
\end{align*}
As a corollary to this and prop.~\ref{prp:2}, one easily verifies the
\emph{co-Yoneda lemma}: if $Y \colon \bK \to \Pw\bK$ is the Yoneda
embedding, and $W \colon \bK^\op \to \Cat$ is a weight, then
\begin{equation*}
  W \star Y \eqv W
\end{equation*}

\begin{prp}
  \label{prp:3}
  The (strict) 2-category $\Pw\bK = [\bK^\op, \Cat]$ is the free
  cocompletion of $\bK$, in that for a cocomplete 2-category $\bL$
  there is a 2-equivalence
  \begin{equation*}
    [\bK, \bL] \sim \mathrm{Cocont}(\Pw\bK, \bL)
  \end{equation*}
  given from right to left by composition with the Yoneda embedding
  $\bK \to \Pw\bK$.
\end{prp}
\begin{proof}
  The inverse to $- \cmp Y$ sends $F \colon \bK \to \bL$ to
  $\overline{F} \colon W \mapsto W \star F$.  Applying
  corollary~\ref{cor:4}, and the cocompleteness of $\bL$, to the
  functor
  \begin{equation*}
    (W,F) \mapsto [\bK^\op, \Cat](W, \bK(F,1))
  \end{equation*}
  shows that $(W,F) \mapsto W\star F$ extends to a functor $[\bK^\op,
  \Cat] \times [\bK,\bL] \to \bL$, and hence fixing $F$ does yield a
  functor $(- \star F) \colon \Pw\bK \to \bL$, which is cocontinuous
  essentially because representables are continuous (cf.~\cite[section
  3.3]{kelly82}).

  We want to show that this inverse is a 2-equivalence.  The co-Yoneda
  lemma above implies that if $H \colon \Pw\bK \to \bL$ is
  cocontinuous then $H(W) \eqv H(W \star Y) \eqv W \star H Y$, showing
  that $F \mapsto \overline{F} = (- \star F)$ is essentially
  surjective.  That it is 2-fully-faithful again follows easily from
  the Yoneda lemma.
\end{proof}

\subsection{Computing colimits}
\label{sec:computing-colimits}

An explicit description of conical colimits in $\Cat$ is not too
difficult to find, thanks in part to a classical result due to
Grothendieck and Verdier.  First recall the following:

\begin{dfn}[{\cite[1.10]{street80:_fibrat_in_bicat}}]
  \label{dfn:27}
  If $\bJ$ is a 2-category and $D \colon \bJ \to \Cat$ is a functor
  then the \emph{2-category of elements} $\elt D$ of $D$ is given as
  follows:
  \begin{itemize}
  \item an object is an object $j \in \bJ$ together with an object $x
    \in Dj$;
  \item a morphism from $x \in Dj$ to $y \in Di$ is given by a
    morphism $m \colon j \to i$ in $\bJ$ together with a morphism
    $m_*x \to y$ in $Di$ (where $m_*x = (Dm)(x)$);
  \item a 2-cell from $(m,m_*x \to y)$ to $(n, n_*x \to y)$ is given
    by a 2-cell $\alpha \colon m \tc n$ such that $\alpha_* x \colon
    m_* x \to n_* x$ fits into a commuting triangle over $y$.
  \end{itemize}
  The projection $\elt D \to \bJ$ is a strict functor, and a morphism
  $(m, f \colon m_*x \to y)$ in $\elt D$ is called \emph{opcartesian}
  when $f$ is invertible.  Of course, if $\bJ$ is an ordinary category
  then so is $\elt D$.
\end{dfn}

\begin{prp}[{\cite[expos\'e VI, def.~6.3]{grothendieck72:sga4_2}}]
  If $\cat{C}$ is an ordinary category and $D \colon \cat{C} \to \Cat$
  is a pseudofunctor, then the 2-colimit of $D$ is obtained by taking
  the category of elements $\elt D$ of $D$ and formally inverting the
  opcartesian morphisms.
\end{prp}

In more detail, one verifies that $\elt D$ is the \emph{lax conical
  colimit}\footnote{Note that lax transformations $D \to
  \Delta\B$ correspond to \emph{oplax} transformations
  $\Delta\mathbf{1} \to \Cat(D,\B)$, so that a lax conical
  colimit in this sense is actually an oplax weighted colimit.} of
$D$, i.e.~that there is an equivalence
\begin{equation*}
  \operatorname{Lax}(D,\Delta\B) \eqv [\elt D, \B]
\end{equation*}
that is natural in the ordinary category $\B$, where the left-hand
side is the category of lax transformations and modifications from $D$
to the constant functor at $\B$.  Furthermore, the pseudonatural
transformations on the left correspond to the functors on the right
that invert the opcartesian morphisms of $\elt D$, so that
\begin{equation*}
  \operatorname{Ps}(D,\Delta\B) \eqv [\elt D, \B]_{S^{-1}}
  \eqv [\elt D[S^{-1}], \B]
\end{equation*}
where $S$ is the class of opcartesian morphisms of $\elt D$ and the
notation $[-,-]_{S^{-1}}$ denotes the full subcategory of the functor
category on those functors that invert the elements of $S$.  The
objects of the category of fractions $\elt D[S^{-1}]$ are those of
$\elt D$, and its morphisms are zig-zags of morphisms in $\elt D$
in which the backwards-pointing components are in $S$.

To extend this result to the case of diagrams indexed by 2-categories,
we first recall that there is a monoidal adjunction $\pi \dashv d$, in
which $\pi = \pi_0 \colon \cat{Cat} \to \cat{Set}$ is the `connected
components' functor and $d \colon \cat{Set} \to \cat{Cat}$ the
`discrete category' functor.  Then a suitably `weak' version of the
usual change-of-enrichment arguments, or simply direct calculation,
verify the following.

\begin{prp}
  There is an adjunction
  \begin{equation*}
    [\pi_* \bK, \B] \eqv [\bK, d_* \B]
  \end{equation*}
  in which the functors $\pi_* \dashv d_*$ apply $\pi$ or $d$
  hom-wise.  Moreover, this adjunction descends to the case of
  functors that invert a class $S$ of morphisms of $\bK$:
  \begin{equation*}
    [\pi_* \bK, \B]_{S^{-1}} \eqv [\bK, d_* \B]_{S^{-1}}
  \end{equation*}
\end{prp}
\begin{proof}
  A functor $\bK \to d_*\B$ must take any 2-cell of $\bK$ to an
  identity, and therefore identify any pair of connected morphisms,
  which defines an essentially unique functor out of $\pi_* \bK$.  The
  former inverts a specified morphism if and only if the latter
  inverts its equivalence class, because their image in $\B$ is the
  same and invertibility in $d_*\B$ is precisely invertibility in
  $\B$.

  A transformation between two such functors takes objects of $\bK$ to
  morphisms of $\B$, and morphisms to strictly commuting squares.
  Naturality of these with respect to 2-cells means that any pair of
  connected morphisms must be assigned the same square, and this
  specifies a unique transformation out of $\pi_*\bK$.
\end{proof}

\begin{prp}[\footnote{\label{fn:1}Note to arXiv version: Dorette Pronk
    informs me that she and Laura Scull had arrived independently at
    this result some time ago, but were unable to find any published
    description of it, and have not yet published it themselves.}]
  \label{prp:25}
  If $D \colon \bJ \to \Cat$, then
  \begin{equation*}
    \operatorname{Lax}(D, \Delta \B) \eqv [\elt D, d_*\B]
    \eqv [\pi_*\elt D, \B]
  \end{equation*}
  and
  \begin{equation*}
    \operatorname{Ps}(D, \Delta \B) \eqv [\elt D,
    d_*\B]_{S^{-1}}
    \eqv [\pi_*\elt D, \B]_{S^{-1}} \eqv [(\pi_* \elt D)[S^{-1}],
    \B]
  \end{equation*}
  so that the lax colimit of $D$ is $\pi_* \elt D$, and the colimit of
  $D$ is got by inverting the images of the opcartesian morphisms of
  $\elt D$ in this category.
\end{prp}
\begin{proof}
  (In light of the previous proposition, only the first in each chain
  of equivalences requires proof.)  Let $A \colon \int D \to d_*\B$ be
  a functor.  The inclusions $\delta_i \colon D i \to \int D$ of the
  fibres of $D$ give a family of (ordinary) functors $(A \delta_i
  \colon D i \to \B)_i$.  For each $m \colon i \to j$ and $x \in Di$,
  the canonical opcartesian $(m,1) \colon x \to m_* x$ gives a family
  $\delta_m x \colon \delta_i x \to \delta_j m_* x$, and $A \delta_m
  x$ is then natural in $x$, because given $k \colon x \to x'$ both
  sides of the naturality square are equal to $A(m,m_*k)$.  It is easy
  to check then that $m \mapsto A \delta_m$ is functorial --- the two
  morphisms $x \to (nm)_*x$ in $\int D$ that are required to be equal
  for $\delta$ to be a lax transformation are instead just isomorphic,
  but $A$ turns this 2-cell into an identity.  Similarly, a 2-cell
  $\phi \colon m \tc n$ gives rise to a 2-cell between the two
  morphisms $x \to n_*x$ that naturality would require to be equal,
  but applying $A$ ensures that their images are equal in $d_*\B$.  So
  $A\delta$ is a lax transformation, and because the components of
  $\delta$ are opcartesian, $A\delta$ will be pseudo-natural if $A$
  inverts them.  Moreover, if $\mu \colon A \tc B$ is a
  transformation, then so is each $\mu\delta_i$, and these form a
  modification $A\delta \mdf B\delta$ by virtue of the interchange law
  for $\Cat$, which applies because $d_*\B$ is locally discrete, and
  this assignment is functorial by interchange again.

  Conversely, suppose given a lax transformation $\alpha \colon D \tc
  \Delta \B$ and consider a morphism
  \begin{equation*}
    (m \colon i \to j , f \colon m_* x \to y) \colon (i,x) \to (j,y)
  \end{equation*}
  in $\int D$.  We get a configuration like this:
  \begin{equation*}
    \xymatrix{
      & D i \ar[dr]^{\alpha_i} \ar[dd]|{Dm}^{}="rs"_{}="lt" \\
      \mathbf{1} \ar@{}[];"lt"|{f \displaystyle\Downarrow} \ar[ur]^x \ar[dr]_y & &
      \B \ar@{}[];"rs"|{{\displaystyle\Downarrow} \alpha_m} \\
      & D j \ar[ur]_{\alpha_j}
    }
  \end{equation*}
  which defines a morphism $\alpha_i x \to \alpha_j y$ in $\B$.  This
  assignment is functorial because of the coherence of the $\alpha_m$
  with respect to identities and composition in $\bJ$, and it takes
  2-cells $\phi \colon f \tc g$ to identities in $\B$ because of the
  naturality condition on $\alpha$.  If $\alpha$ is pseudo, moreover,
  then this functor $\hat \alpha$ clearly inverts any morphism $(m,f)$
  in $\int D$ for which $f$ is invertible.  If $p \colon \alpha \mdf
  \beta$ is a modification, then the morphisms $p_i x \colon \alpha_i
  x \to \beta_i x$ assemble, by the modification axiom and naturality
  of each $p_i$, into a natural transformation $\hat \alpha \tc \hat
  \beta$.

  Applying this recipe to the transformation $A \delta$ arising from a
  functor $A \colon \int D \to d_* \B$ gives a functor whose value at
  a morphism $(m,f)$ as above is
  \begin{equation*}
    \xymatrix{
      A x \ar[r]^{A(m,1)} & A m_* x \ar[r]^{A (1,f)} & A y
    }
  \end{equation*}
  which is the $A$-image of the opcartesian--vertical factorization of
  $(m,f)$.  Now the latter is only isomorphic in $\int D$ to $(m,f)$
  itself, but as before applying $A$ makes this an equality.  So the
  functor arising from $A \delta$ is equal to $A$.  In the other
  direction, the functor $\hat \alpha \colon \int D \to d_* \B$
  corresponding to a lax transformation $\alpha$ is equal to
  $\alpha_i$ on each fibre $D i$, and applying it to $\delta_m$
  produces exactly $\alpha_m$.  Finally, a simple calculation shows
  that this correspondence is also bijective on morphisms.
\end{proof}

\subsection{Coends again}
\label{sec:coends-again}

As in ordinary category theory, there are useful relationships between
(co)ends and (co)limits.  If the $W$-weighted colimit $W \star F$ of
$F$ (def.~\ref{dfn:8}) exists in $\bK$ we find that
\begin{align*}
  \bK(W \star F, k)
  & \eqv \{W, \bK(F-, k)\} \\
  & \eqv \int_j \Cat(Wj, \bK(Fj, k)) \\
  & \eqv \int_j \bK(Wj \otimes Fj, k)
\end{align*}
if the tensors $Wj \otimes Fj$ exist, so that
\begin{equation}
  \label{eq:2}
  W \star F \eqv \int^j Wj \otimes Fj
\end{equation}
as usual \cite[section 3.10]{kelly82}.  Dually, of course, we may, by
only a slight generalization of prop.~\ref{prp:1}, write
\begin{equation*}
  \{W, G\} \eqv \int_j Wj \pitchfork Gj
\end{equation*}
as long as the necessary cotensors exist.

\begin{dfn}
  \label{dfn:18}
  The category $\Dsc$ is the subcategory of the (unaugmented) simplex
  category $\Delta$ generated by the following diagram:
  \begin{equation*}
    \xymatrix{
      \mathbf{3}
      \ar@<1ex>[r];[] \ar[r];[] \ar@<-1ex>[r];[] &
      \mathbf{2}
      \ar@<1ex>[r];[] \ar@<-1ex>[r];[] & \ar[l];[]
      \mathbf{1}
    }
  \end{equation*}
  A \emph{codescent diagram} in $\bK$ is a functor $\Dsc^\op \to \bK$.
  The conical colimit of such a diagram is called its \emph{codescent
    object}.
\end{dfn}

Coends may be expressed as codescent objects (cf.~the usual
presentation of 1-dimensional coends as coequalizers, as in
\cite[(2.2)]{kelly82}): for any 2-category $\bK$, the co-Yoneda lemma
lets us write $\bK(k,\ell) \eqv \bK(k,-) \star \bK(-,\ell)$, and the
(dual of the) construction of weighted 2-limits in
\cite{street87:_correc_fibrat} yields $\bK(k,\ell)$ as the following
codescent object, where we adopt a tensor-style notation as in
$\bK^k_j \bK^j_\ell = \coprod_j \bK(k,j) \times \bK(j,\ell)$, etc.,
with the obvious summation convention for repeated indices:
\begin{equation}
  \label{eq:3}
  \xymatrix{
  \bK^k_j \bK^j_i \bK^i_h \bK^h_\ell
  \ar@<1ex>[r] \ar[r] \ar@<-1ex>[r] &
  \bK^k_j \bK^j_i \bK^i_\ell
  \ar@<1ex>[r] \ar@<-1ex>[r] & \ar[l]
  \bK^k_j \bK^j_\ell
  \ar[r] &
  \bK^k_\ell
  }
\end{equation}
The morphisms in the diagram are the actions of $\bK$ on itself given
by composition and the insertion of identities.  The whole diagram is
functorial in $(k,\ell)$ and so presents $\hom_\bK$ as a codescent
object in the functor category.  (The diagram is also the `canonical
presentation' \cite{le02:_becks} of the algebra $\hom_{\bK}$ for the
2-monad on $[\ob \bK, [\bK, \Cat]]$ whose algebras are functors
$\bK^\op \times \bK \to \Cat$.)

Now if $T \colon \bK^\op \times \bK \to \bL$ is a functor, then
applying the functor $(- \star T)$ (which is cocontinuous, as observed
in the proof of prop.~\ref{prp:3}) to the codescent diagram
\eqref{eq:3} for $\hom_{\bK^\op}$ yields the following codescent
object in $\bL$:
\begin{equation}
  \label{eq:4}
  \xymatrix{
  T^k_j \bK^j_i \bK^i_k
  \ar@<1ex>[r] \ar[r] \ar@<-1ex>[r] &
  T^k_j \bK^j_k
  \ar@<1ex>[r] \ar@<-1ex>[r] & \ar[l]
  T^k_k
  \ar[r] &
  \int^k T(k,k)
  }
\end{equation}
In more detail, $\bigl(-\star T\bigr)$ preserves codescent objects and
coproducts, so we get e.g.
\begin{align*}
  \biggl(\coprod_{j,i} \bK(-, j) \times \bK(j,i) \times \bK(i,-)
  \biggr) \star T
  & \simeq \coprod_{j,i} \Bigl ( \bK(-, j) \times \bK(j,i) \times
    \bK(i,-) \Bigr) \star T \\
  & \simeq \coprod_{j,i} \Bigl(\bigl(\bK^\op \times \bK\bigr)(-,(i,j))
  \star T \Bigr) \times \bK(j,i) \\
  & \simeq T^i_j \bK^j_i
\end{align*}
where in the second line we use the fact that product with $\bK(j,i)$
is the tensor in the functor category and so is preserved by $(-\star
T)$, and in the last we use the Yoneda equivalence \eqref{eq:17}.  It
follows that the rightward arrows in \eqref{eq:4} are given by
composition in $\bK$ and the left and right actions of $\bK$ on $T$.
The results of the previous section now yield an explicit recipe for
computing coends in $\Cat$.

Codescent objects also figure in the 2-categorical version of Beck's
theorem, which we shall need for theorem~\ref{thm:kleisli-2prof}.

\begin{prp}[2-monadicity~theorem, {\cite[theorem~3.6]{le02:_becks}}]
  \label{prp:5}
  Let $U \colon \bK \to \bL$ be a 2-functor with left adjoint $F$.
  The canonical functor $\bK \to \bL^{UF}$ into the 2-category of
  (pseudo) $UF$-algebras is an equivalence if and only if $U$ reflects
  adjoint equivalences and $\bK$ has and $U$ preserves colimits of
  codescent diagrams whose $U$-image has an absolute colimit.
\end{prp}


\chapter{Equipments}
\label{cha:equipments}

This chapter contains our main results.  The first section constructs
the 3-category of 2-profunctors as promised, and shows that it has
well-behaved Kleisli objects for pseudo-monads.  This then gives a
correspondence between pseudo-monads on a 2-category $\bK$ and
identity-on-objects functors $\bK \to \bM$, which is the basis for our
comparison of notions of equipment in section~\ref{sec:defin-equipm}.
We stop short of trying to construct a 2- or 3-category of equipments
directly from the monad definition: this could certainly be done, and
section~\ref{sec:more-equipments} discusses how one might go about it,
but it is more than we need.  The correspondence on objects and
morphisms that we give is enough to get the results that we want, so
it is a natural stopping point.

Section~\ref{sec:equipm-fibr} then defines cartesian equipments as
cartesian objects in the 2-category of equipments constructed in the
previous section, and shows that there is a fully faithful functor
from the 2-category of regular fibrations into that of cartesian
equipments.  Axioms are given that ensure that a given cartesian
equipment is in the image of this functor; these we call
\emph{regular} equipments, and the full sub-2-category of cartesian
equipments on the regular ones is therefore equivalent to the
2-category of regular fibrations.  The last subsection compares
comprehension for predicates in a regular fibration
(def.~\ref{dfn:22}) with tabulation for morphisms and with
Eilenberg--Moore objects for co-monads in a regular equipment.

Finally, section~\ref{sec:effective-topos} examines the two
constructions of the effective topos through the lens of the preceding
material, showing how the equivalence we have given between regular
fibrations and equipments can be used to relate them.

\section{2-profunctors and equipments}
\label{sec:biprofunctors}

A 2-profunctor $H \colon \bK \prof \bL$, as we have said, will be a
functor $\bL^\op \times \bK \to \Cat$.  By the results of
section~\ref{sec:free-cocompl-bicat}, this is essentially the same
thing as a cocontinuous functor $\Pw\bK \to \Pw\bL$.  Because
2-categories of the form $\Pw\bK$ are strict, composition of such
functors is associative and unital on the nose.  So one would hope
that $\Biprof$ would turn out to be a $\cat{Gray}$-category, or even a
strict 3-category, but it is neither: whiskering a transformation by a
functor fails to be strictly functorial.

\subsection{The tricategory $\Biprof$}
\label{sec:tricategory-biprof}

\begin{dfn}
  $\Biprof$ is the tricategory whose objects are 2-categories $\bK,
  \bL, \ldots$ and whose homs are given by $\hom(\bK, \bL) =
  \mathrm{Cocont}(\Pw\bK, \Pw\bL)$.
\end{dfn}
By this definition, all of the structure of $\Biprof$ bar the objects
is imported directly from (a suitably large version of) $\Bicat$.
Because $\Bicat$ is known to be a tricategory \cite[section
6.3]{gurski07}, then, so is $\Biprof$.

Suppose we are given profunctors $H \colon \bK \prof \bL$ and $G
\colon \bL \prof \bM$, corresponding to $\hat{H} \colon \Pw\bK
\to \Pw\bL$ and $\hat{G} \colon \Pw\bL \to \Pw\bM$.  Then we can
compose $H$ with $G$ by composing $\hat{H}$ directly with
$\hat{G}$ and passing back across the equivalence of
prop.~\ref{prp:3} to get a profunctor $GH \colon \bL \prof \bM$:
\begin{equation*}
  GH(m,k) \eqv \hat{G}(\hat{H}(Y k))(m) \eqv
  H(-,k) \star G(m,-)
\end{equation*}
By the coend formula \eqref{eq:2} for weighted limits, this gives:
\begin{equation}
  \label{eq:5}
  G \cmp H \eqv \int^\ell G(-, \ell) \times H(\ell, -)
\end{equation}
just as for ordinary profunctors.  So we may switch freely between
profunctors considered as cocontinuous functors between 2-presheaf
categories, composed as ordinary functors, and profunctors considered
as $\Cat$-valued functors composed as above.

It follows immediately from this and the cartesian closedness of
$\Cat$ that
\begin{prp}
  \label{prp:4}
  $\Biprof$ has stable local colimits; that is, the colimits in
  $\Biprof(\bK,\bL) \eqv [\bL^\op \times \bK, \Cat]$ are preserved by
  composition with a profunctor on either side.
\end{prp}

Sending a functor $F \colon \bK \to \bL$ to the profunctor $\bL(1,F)
\colon \bK \prof \bL$ gives a mapping from $\Bicat$ to $\Biprof$ that
is the identity on objects and locally fully faithful by the Yoneda
lemma (corollary~\ref{cor:5}).  The co-Yoneda lemma shows that it is
functorial, i.e.~that $\bM(1,G) \cmp \bL(1,F) \eqv \bM(1,GF)$.
(Indeed, more is true: by the same lemma, we have, for functors $F
\colon \bK \to \bL$ and $G \colon \bJ \to \bM$, and a profunctor $H
\colon \bL \prof \bM$, that
\begin{equation*}
  \bM(G,1) \cmp H \cmp \bL(1,F) \eqv H(G,F)
\end{equation*}
as for the analogous functor $\Cat \to \bicat{Prof}$
\cite{wood82:proarrows_i,shulman08:_framed_bicat_and_monoid_fibrat}.)

The functor $\widehat{\bL(1,F)} \colon \Pw\bK \to \Pw\bL$ (sometimes
called the \emph{Yoneda extension} of $F$) takes a weight $W$ and an
object $\ell \in \bL$ to $W \star \bL(\ell, F-)$.  As a functor of
$\ell$, this is the pointwise colimit $W \star \bL(1,F)$, where
$\bL(1,F)$ is taken as a functor $\bK \to \Pw\bL$.  Dually,
$\widehat{\bL(F,1)}$ takes $V \in \Pw\bL$ and $k \in \bK$ to $V
\star \bL(Fk,-)$; but by the co-Yoneda lemma this is just $VFk$, so
that $\widehat{\bL(F,1)}$ is the pullback-along-$F$ functor $F^*$.
Using this, we may calculate
\begin{align*}
  \Pw\bL(\widehat{\bL(1,F)} W, V) 
  & \eqv \Pw\bK(W, \Pw\bL(\bL(1,F), V)) \\
  & \eqv \Pw\bK(W, VF) & \text{by Yoneda} \\
  & \eqv \Pw\bK(W, \widehat{\bL(F,1)} V)
\end{align*}
Thus $\widehat{\bL(1,F)}$ is left adjoint to $\widehat{\bL(F,1)}$,
and so we have
\begin{equation}
  \label{eq:7}
  \bL(1,F) \dashv \bL(F,1)  
\end{equation}
in $\Biprof$.  The functor $\Bicat \to \Biprof$ is therefore a
proarrow equipment, in a suitable 3-categorical sense.

\subsection{Kleisli objects in $\Biprof$}
\label{sec:kleisli-objects}

The Kleisli object of a monad $H \colon \C \prof \C$ in $\bicat{Prof}$
is given by the category whose objects are those of $\C$ and whose
homs $\C_H(a,b)$ are given by $H(a,b)$.  Identities and composition
are defined using the unit and multiplication of $H$.  So a monad on
$\C$ in $\bicat{Prof}$ is essentially the same thing as a functor $\C
\to \D$ that is bijective on objects.  Things are much the same in our
2-categorical setting.

\begin{thm}
  \label{thm:kleisli-2prof}
  $\Biprof$ has \emph{tight Kleisli objects}: if $T \colon \bK \prof
  \bK$ is a monad in $\Biprof$, then there is a 2-category $\bK_T$ and
  a functor $F_T \colon \bK \to \bK_T$ such that composition with the
  right $T$-module $\bK_T(1,F_T)$ gives rise to an equivalence
  $\Biprof(\bK_T,-) \sim \mathrm{RMod}(T,-)$.
\end{thm}
\begin{proof}
  A 2-category $\bK$ is the same thing as a pseudo double category
  \cite[section~1.9]{grandis99:_limit} whose category of objects is
  discrete, and this in turn means that $\bK$ is a monad in
  $\Span(\Cat)$ on the discrete category $\bK_0 =
  \operatorname{ob} \bK$.  We will show first that the forgetful
  functor
  \begin{equation*}
    \Biprof(\bK, \bK) = [\bK^\op \times \bK, \Cat] \longrightarrow
    [\bK_0 \times \bK_0, \Cat] = \Span(\Cat)(\bK_0, \bK_0)
  \end{equation*}
  is monoidal (def.~\ref{dfn:26}) and so takes the monad $T$ to a
  monad in $\Span(\Cat)$ (def.~\ref{dfn:29}).

  The monoidal structure on $[\bK^\op\times\bK, \Cat]$ is given by
  profunctor composition \eqref{eq:5}.  Writing the image of a
  profunctor $H$ under the restriction functor above as $\tilde H$, we
  find that the composite $\tilde G \tilde H$ in $[\bK_0\times\bK_0,
  \Cat]$ is given by
  \begin{equation*}
    \tilde G \tilde H(k,\ell)
    = G^k_j H^j_\ell
    = \coprod_j G(k,j) \times H(j,\ell)
  \end{equation*}
  The identity for this composition is the equality predicate on
  $\bK_0$, which sends $k,\ell$ to the terminal category $\mathbf{1}$
  if $k=\ell$, or to the empty category $\nullset$ otherwise.
  Equivalently, it is the identity span on $\bK_0$.

  The required comparison morphisms are then given by the codescent
  morphism \eqref{eq:4}
  \begin{equation*}
    G^k_j H^j_\ell = \coprod_j G(k,j) \times H(j,\ell) \longrightarrow
    \int^j G(k,j) \times H(j,\ell) = (GH)^k_\ell
  \end{equation*}
  and the identity-assigning functor $1_- \colon \mathbf{1} \to
  \bK(k,k)$.  We must show that these satisfy the conditions of
  def.~\ref{dfn:26}.

  For any three composable profunctors, there is a diagram $\Dsc^\op
  \times \Dsc^\op \to [\bK^\op \times \bK, \Cat]$ (where $\Dsc$ is as
  in def.~\ref{dfn:18}) whose colimit is their composite.  By
  universality, this may be calculated directly, or as the colimit of
  the colimit of either of the two adjuncts $\Dsc^\op \to [\Dsc^\op,
  [\bK^\op \times \bK, \Cat]]$ of the diagram.  This gives the
  injections
  \begin{equation*}
    \xymatrix{
      G^k_j H^j_i K^i_\ell \ar[r] \ar[d] &
      G^k_j (HK)^j_\ell \ar[d] \\
      (GH)^k_i K^i_\ell \ar[r] &
      (GHK)^k_\ell
    }
  \end{equation*}
  and the canonical isomorphism filling this square is the required
  associator $a$.  The commutative cube formed from the six
  different such squares associated to a fourfold composite is exactly
  the coherence condition \eqref{eq:21} for $a$.

  To express the coherence condition on the unitors, we first note
  that the required $l$ and $r$ for the composite of $G$ and
  $H$ arise from the left unitor of $G$ and the right unitor of $H$:
  $l$ is the 2-cell in
  \begin{equation*}
    \xymatrix{
      G^k_j H^j_\ell \ar[rr]^1_{}="s" \ar[dr] & & G^k_j H^j_\ell \\
      & G^k_j \bK^j_i H^i_\ell \ar@{}[];"s"|{\lambda} \ar[ur]
    }
  \end{equation*}
  whose components are the induced isomorphisms $(g,h) \iso (g1, h)$,
  where $g1$ is the action by $G$ on $g$ of the appropriate identity
  morphism of $\bK$, and similarly for $r$.  The coherence
  condition itself then requires that the morphism in $(GH)^k_\ell$:
  \begin{equation}
    \label{eq:18}
    (g,h) \leftarrow (g1, h) \leftarrow (g,1,h) \to (g, 1h) \to (g,h)
  \end{equation}
  be the identity.  This is equal to the composite of (the formal
  inverse of)
  \begin{equation*}
    (g,h) \to (g,1,h) \to (g1,h) \to (g,h)
  \end{equation*}
  with
  \begin{equation*}
    (g,h) \to (g,1,h) \to (g,1h) \to (g,h)
  \end{equation*}
  where in both cases the factor $(g,h) \to (g,1,h)$ is given by the
  action of the splitting in the codescent diagram.  In each of these
  the underlying morphism of $\Dsc$ is the identity, and the morphism
  $1^*(g,h) \to (g,h)$ is the composite $(g,h) \to (g1,h) \to (g,h)$,
  respectively $(g,h) \to (g,1h) \to (g,h)$, of inverses.  Both
  morphisms are thus identities, and so their formal composite in the
  coend $(GH)^k_\ell$ is also the identity.  Hence the functor
  $[\bK^\op \times \bK, \Cat] \to [\bK_0 \times \bK_0, \Cat]$ is
  indeed monoidal.

  The monad $T$ in $\Biprof$ is sent by this functor to a 2-category
  $\bK_T$, with objects those of $\bK$ and hom-categories
  $\bK_T(k,\ell)$ the values $T(k,\ell)$ of $T$.  Identities in
  $\bK_T$ are the $\eta$-images of identities in $\bK$, and
  composition is given by the action of $\mu$.  The unit $\eta$ of $T$
  is a morphism of monoids in $\Biprof(\bK,\bK)$ and thus is sent to a
  functor $F_T \colon \bK \to \bK_T$, which of course is the identity
  on objects.

  It remains to show that the profunctor $\bK_T(1,F_T)$ is the
  universal right $T$-module.  The adjunction $\bK_T(1,F_T) \dashv
  \bK_T(F_T,1)$ \eqref{eq:7} gives rise to an adjunction
  \begin{equation*}
    \xymatrix@C+7ex{
      \Biprof(\bK, \bL) \ar@{}[r]|{\bot}
      \ar@<1ex>[r]^{- \cmp \bK_T(F_T,1)}
      &
      \Biprof(\bK_T, \bL) \ar@<1ex>[l]^{- \cmp \bK_T(1,F_T)}
    }
  \end{equation*}
  The unitors of $T$ supply an equivalence $T \eqv \bK_T(F_T, F_T)$
  (whose components are identities), which respects their monad
  structures essentially by definition --- the unit and counit of the
  adjunction above are given by the unit and multiplication of $T$.
  Thus $\mathrm{RMod}(T,\bL)$ is equivalent to the category of
  algebras for the monad induced by the adjunction above, and so there
  is a canonical comparison functor $\Biprof(\bK_T,\bL) \to
  \mathrm{RMod}(T,\bL)$ given by composition with the module
  $\bK_T(1,F_T)$.  To show that this functor is an equivalence, then,
  and hence that this module is the universal one, it suffices to show
  that the right adjoint $- \cmp \bK_T(1,F_T) \colon H \mapsto
  H(1,F_T)$ above is monadic, in the sense of prop.~\ref{prp:5}.  We
  already know (prop.~\ref{prp:4}) that $\Biprof$ has stable local
  colimits, so that $\Biprof(\bK_T, \bL)$ has, and $- \cmp
  \bK_T(1,F_T)$ preserves, the required codescent objects.  It remains
  only to show that this functor reflects adjoint equivalences: if
  $\alpha \colon G \tc H$ is a transformation such that $\alpha \cmp
  \bK_T(1,F_T) \colon G(1,F_T) \tc H(1,F_T)$ is an equivalence, then
  because $F_T$ is the identity on objects, the components of $\alpha$
  are precisely the components of $\alpha \cmp \bK_T(1,F_T)$, and
  hence if the latter are all equivalences then so are the former.
\end{proof}

\begin{rem}
  \label{rem:6}
  Any functor $F \colon \bK \to \bL$ in $\Bicat$ gives rise to a monad
  $\bL(F,F) \eqv \bL(F,1) \cmp \bL(1,F)$ in $\Biprof$.  It is easy to
  see that the Kleisli object of this monad is the \emph{full image}
  of $F$: its objects are those of $\bK$ and the hom from $k$ to
  $\ell$ is $\bL(Fk, F\ell)$.
\end{rem}

Suppose $G \colon \bK \to \bL$ is a functor.  An action of $T$ on $G$
has a mate
\begin{prooftree}
  \AXC{$\bL(1,G) \hcmp T \longrightarrow \bL(1,G)$}
  \UIC{$T \longrightarrow \bL(G,G)$}
\end{prooftree}
and the first underlies a right $T$-module if and only if the second
is a morphism of pseudo-monoids, by remark~\ref{rem:9}.  The morphism
$T \to \bL(G,G)$ corresponding to a right action is then sent by the
construction of theorem~\ref{thm:kleisli-2prof} to a functor $\bK_T
\to \bK_{\bL(G,G)}$ into the full image of $G$ that is the identity on
objects and whose action on hom-categories is given by the components
of the monoid morphism.  This then composes with the fully faithful
$\bK_{\bL(G,G)} \to \bL$ given by the action of $G$ on objects to give
a functor $\bK_T \to \bL$.  The unit axiom for a morphism of monoids
then shows that the composite of this functor with the canonical $\bK
\to \bK_T$ is equivalent to $G$.  This shows how to compute the
functor $\bK_T \to \bL$ corresponding to a representable right
$T$-module, a recipe that it is difficult to extract from the proof of
theorem~\ref{thm:kleisli-2prof}.

\begin{cor}
  The corepresentable profunctor $\bK_T(F_T,1)$ exhibits $\bK_T$ as
  the Eilenberg--Moore object of $T$.
\end{cor}
\begin{proof}
  Apply the argument of theorem~\ref{thm:kleisli-2prof} to the
  adjunction
  \begin{equation*}
    \xymatrix@C+7ex{
      \Biprof(\bL, \bK) \ar@{}[r]|{\bot}
      \ar@<1ex>[r]^{\bK_T(1,F_T) \cmp -}
      &
      \Biprof(\bL, \bK_T) \ar@<1ex>[l]^{\bK_T(F_T,1) \cmp -}
    }
  \end{equation*}
  to show that $\Biprof(\bL,\bK_T) \sim \mathrm{LMod}(T,\bL)$.
\end{proof}

\begin{cor}
  \label{cor:1}
  Precomposition with $\bK_T(1,F_T)$ preserves and detects
  representables.
\end{cor}
\begin{proof}
  If $H \colon \bK_T \prof \bL$ is representable then clearly so is
  $H(1,F_T)$.  Suppose conversely that $H(1,F_T)$ is representable, as
  $\bL(1,G)$, say.  Then, because $F_T$ is the identity on objects,
  each $H(-,k) = H(-,F_T k)$ is a representable presheaf $\bL(-,G k)$,
  and so by remark~\ref{rem:5} $H$ is representable.
\end{proof}

This means in particular that representable right modules correspond
to representable profunctors out of the Kleisli object.  It also means
that if $T = \bK(1,T')$ is a representable monad, then the Kleisli
object $\bK_T$ of $T$ in $\Biprof$ is also the Kleisli object of $T'$
in $\Bicat$; because $T$ is representable, the right adjoint
$\bK_T(F_T,1)$ of the Kleisli morphism, arising as it does from the
right $T$-module structure of $T$ itself, is representable, as
$\bK(1,U_T)$, say.  But the second-to-last corollary shows that this
is not necessarily the case for Eilenberg--Moore objects.

As an aside, we can say something similar about coproducts in
$\Biprof$ (cf.~axioms 4 and 5 of \cite{wood85:_proar_ii}):

\begin{prp}
  $\Biprof$ has representable coproducts: if $\{\bK_i\}_i$ is a small
  family of 2-categories, then its coproduct $\coprod_i \bK_i$ in
  $\Bicat$ is also its coproduct in $\Biprof$.  As before, the
  injections are representable and preserve and jointly detect
  representables, and their adjoints together exhibit the coproduct of
  the $\bK_i$ as their product.
\end{prp}
\begin{proof}
  Let $\{\iota_i \colon \bK_i \to \coprod_i \bK_i\}_i$ be the obvious
  injections, and assume given a family $\{\bK_i \prof \bL\}_i$.  Then
  there are equivalences
  \begin{prooftree}
    \AXC{$\{\bK_i \prof \bL\}_i$}
    \UIC{$\bK_i \to [\bL^\op, \Cat]\}_i$}
    \UIC{$\coprod \bK_i \to [\bL^\op, \Cat]$}
    \UIC{$\coprod \bK_i \prof \bL$}
  \end{prooftree}
  showing that $\coprod \bK_i$ is again a coproduct in $\Biprof$.  The
  equivalence (corollary~\ref{cor:4}) between representability and
  pointwise representability shows that the profunctor in the bottom
  line is representable if and only if the family in the top line is
  so.  Finally, much the same argument (together with the fact that
  $(\coprod \bK_i)^\op = \coprod \bK_i^\op$) shows that the
  corepresentables $\bK_i(\iota_i, 1)$ mediate an equivalence between
  profunctors $\bL \prof \coprod \bK_i$ and families $\{\bL \prof
  \bK_i\}_i$.
\end{proof}

\subsection{Equipments and their morphisms}
\label{sec:defin-equipm}

We now want to argue that the various notions of proarrow equipment in
the literature are either subsumed by or at least clearly related to
the notion of pseudo-monad in $\Biprof$, or, what is essentially the
same thing, the Kleisli object of one such.  We will see, however,
that even though monads and monad morphisms capture the right notion
of equipments and functors between them, the situation is more subtle
when it comes to transformations.  Here we will treat only the case of
equipments over (i.e.~monads on) 2-categories that are locally
discrete, because that is the important one, but we will touch in the
general case again in chapter~\ref{cha:concl-future-work}.

\begin{dfn}
  An \emph{equipment} is, equivalently, a monad in $\Biprof$ on a
  1-category $\cat{K}$ or an identity-on-objects functor
  $\cat{K}\to\bM$.
\end{dfn}

We will follow \cite{lack12:_enhan} in calling $\cat{K}$ the category
of \emph{tight} morphisms of the equipment, a morphism in $\bM$ being
called tight if it is the image of a morphism of $\cat{K}$.

\begin{dfn}[{\cite{wood82:proarrows_i}}]
  \label{dfn:20}
  An \emph{equipment in the sense of Wood} is given by 2-categories
  $\bicat{K}$ and $\bicat{M}$ with the same objects, where $\bK$ is
  strict, and a 2-functor $(-)_\bullet \colon \bicat{K} \to \bicat{M}$
  that is the identity on objects and locally fully faithful, and such
  that the image $f_\bullet$ of every morphism $f$ of $\bicat{K}$ has
  a right adjoint $f^\bullet$ in $\bicat{M}$ (i.e.~the functor
  $(-)_\bullet$ factors through $\Map(\bicat{M})$).
\end{dfn}

Leaving out the condition on the existence of right adjoints, it is
clear that an identity-on-objects functor $\bK \to \bicat{M}$ (out of
a strict 2-category) that is locally fully faithful is the same thing
as an identity-on-objects functor $\cat{K} \to \bicat{M}$ out of a
locally discrete 2-category.  By the results of the previous section,
this is the same thing as the Kleisli object of an essentially unique
monad in $\Biprof$ on $\cat{K}$.

An equipment that satisfies the condition that tight morphisms have
right adjoints we will call a \emph{map-equipment}.  A map-equipment
in which every morphism with a right adjoint is tight will be called
\emph{chordate} \cite{lack12:_enhan}.

\begin{dfn}[{\cite{carboni98:change_base_geom_ii}}]
  An \emph{equipment in the sense of Carboni et.~al.}~is given by a
  category $\cat{K}$ together with a 2-functor $M \colon \cat{K}^\op
  \times \cat{K} \to \Cat$.  A \emph{pointed equipment} in their sense
  is given by such an equipment together with a transformation
  $\hom_{\cat{K}} \nt M$.
\end{dfn}

Any monad in $\Biprof$ has a canonical underlying pointed equipment in
this sense.  Conversely, to give the structure of such a monad on a
pointed equipment $M$ is precisely to specify how heteromorphisms in
the putative Kleisli 2-category $\cat{K}_M$ are to be composed.

\begin{dfn}[{\cite{shulman08:_framed_bicat_and_monoid_fibrat}}]
  An \emph{equipment in the sense of Shulman} (or a \emph{framed
    bicategory}) is given by a pseudo double category whose underlying
  span $\cat{K} \leftarrow \cat{M} \rightarrow \cat{K}$ in $\Cat$ is a
  two-sided bifibration (def.~\ref{dfn:30}).

  One half of this property is equivalent to requiring that every
  vertical morphism have a \emph{horizontal companion} in the sense of
  \cite{grandis04:_adjoin}: the companion of $f \colon x \to y$ is a
  horizontal morphism $f_\bullet \colon x \prof y$ equipped with cells
  \begin{equation}
    \label{eq:26}
    \xymatrix{
      x \ar[r]|+^{f_\bullet} \ar[d]_f \ar@{}[dr]|{\displaystyle\Downarrow}
      & y \ar[d]^1 \\
      y \ar[r]|+_1 & y
    }
    \qquad
    \xymatrix{
      x \ar[r]|+^1 \ar[d]_1 \ar@{}[dr]|{\displaystyle\Downarrow}
      & x \ar[d]^f \\
      x \ar[r]|+_{f_\bullet} & y
    }
  \end{equation}
  that compose vertically and horizontally to the identities on $f$
  and $f_\bullet$.  Similarly, the other half of the bifibration
  property requires every vertical morphism $f$ to have a
  \emph{horizontal adjoint} $f^\bullet$, which is then right adjoint
  to $f_\bullet$ in the horizontal 2-category of $(\cat{K},\cat{M})$,
  i.e.~the 2-category of cells with identity vertical boundaries (we
  will call such cells \emph{globular} and write
  $\cat{M}_{\mathrm{gl}}$ for this 2-category).
\end{dfn}

In \cite[appendix~C]{shulman08:_framed_bicat_and_monoid_fibrat} it is
shown that every (map-)equipment in the sense of Wood gives rise to a
framed bicategory (as long as the former's 2-category $\bK$ of tight
maps is strict), and vice versa, and it is stated that these
constructions are inverses up to isomorphism.  In more detail, from a
framed bicategory as above we get an identity-on-objects functor from
$\cat{K}$ to the horizontal 2-category $\bM$ of $(\cat{K},\cat{M})$,
which sends a vertical map $f \colon x \to y$ to its companion
$f_\bullet \colon x \prof y$.  This is then a map-equipment $\cat{K}
\to \bM$.  In the other direction, given a map-equipment $\cat{K} \to
\bM$ over a locally discrete 2-category, there is a pseudo double
category $(\cat{K},\mathrm{Sq}_{\cat{K}}(\bM))$ with the same objects,
with $\cat{K}$ as vertical category, the morphisms of $\bM$ as
horizontal morphisms and cells
\begin{equation}
  \label{eq:24}
  \xymatrix{
    x \ar[r]|+^m \ar[d]_f \ar@{}[dr]|{\displaystyle\Downarrow} & y \ar[d]^g \\
    z \ar[r]|+_n & w
  }
\end{equation}
the 2-cells $g_\bullet M \tc N f_\bullet$ in $\bM$.  (If we write $F_T
\colon \cat{K} \to \bM$ for the inclusion, then the category of cells
$\mathrm{Sq}_{\cat{K}}(\bM)$ is the category of elements $\int
\bM(F_T,F_T)$.)  By
\cite[prop.~C.3]{shulman08:_framed_bicat_and_monoid_fibrat} this is a
framed bicategory.  Clearly, these constructions are inverses up to an
isomorphism that is the identity on $\cat{K}$:
\begin{equation}
  \label{eq:27}
  \xymatrix{
    & \cat{K} \ar[dr] \ar[dl] \\
    \bM \ar[rr]_-{\sim} & & (\mathrm{Sq}_{\cat{K}} \bM)_{\mathrm{gl}}
  }
  \qquad
  \xymatrix{
    \cat{M} \ar[rr]^-{\sim} \ar[dr] & &
    \mathrm{Sq}_{\cat{K}}(\cat{M}_{\mathrm{gl}}) \ar[dl] \\
    & \cat{K} \times \cat{K}
  }
\end{equation}
that on the right arising from the bijection between globular cells
$g_\bullet M \tc N f_\bullet$ and cells of the form \eqref{eq:24} in
$(\cat{K},\cat{M})$.

\begin{dfn}
  \label{dfn:31}
  An \emph{equipment profunctor} from $(\cat{K},T)$ to $(\cat{L},S)$ is
  a profunctor $H \colon \cat{K} \prof \cat{L}$ that underlies a monad
  (op-)morphism $H T \to S H$, that is, an algebra for the monad
  $\hat{T^*}$ given by precomposition with $T$ on the Kleisli
  2-category $\Biprof(\cat{K},\cat{L})_{S_*}$ of the monad given by
  postcomposition with $S$ (cf.~section~\ref{sec:monads-modules}).

  An \emph{equipment morphism} is a representable equipment
  profunctor.
\end{dfn}

We will now compare this definition to the others.

Wood \cite{wood85:_proar_ii} defines an equipment morphism to be a
functor $F \colon \cat{K} \to \cat{L}$ that fits into a square
\begin{equation}
  \label{eq:25}
  \xymatrix{
    \cat{K} \ar[d]_F \ar@{}[dr]|{\eqv} \ar[r] & \cat{K}_T
    \ar[d]^{\tilde F} \\
    \cat{L} \ar[r] & \cat{L}_S
  }
\end{equation}
To give such a lift $\tilde F$ of $F$ is equivalently to give a right
$T$-module structure on $F_S F$, by the universal property of
$\cat{K}_T$ ($\tilde F$ will be representable if $F$ is, by
corollary~\ref{cor:1}).  The equivalence in the square is also
essentially unique, given $F$ and $\tilde F$, for the same reason.

Recall that the injection $F_S \colon \cat{L} \to \cat{L}_S$ satisfies
$S \eqv \cat{L}_S(F_S,F_S)$, so that
\begin{equation*}
  S_* \eqv \cat{L}_S(F_S,F_S)_* \eqv \cat{L}_S(F_S,1)_* \hcmp
  \cat{L}_S(1,F_S)_*
\end{equation*}
This latter monad is a representable profunctor in any tricategory
$\Biprof'$ of 2-categories and profunctors large enough to contain
$\Biprof(\cat{K},\cat{L})$ as an object, and as such its Kleisli
2-category may be constructed by the above recipe.  Thus by the
adjunction $\cat{L}_S(1,F_S) \dashv \cat{L}_S(F_S,1)$ \eqref{eq:7} we
have
\begin{equation*}
  \Biprof(\cat{K},\cat{L})(1,S_*) \eqv
  \Biprof(\cat{K},\cat{L})(\cat{L}_S(1,F_S)_*, \cat{L}_S(1,F_S)_*)
\end{equation*}
and the Kleisli 2-category of the latter is simply the full image of
the functor $\cat{L}_S(1,F_S)_*$ (remark~\ref{rem:6}) --- its objects
are profunctors $H \colon \cat{K} \prof \cat{L}$ and the hom-object
from $H$ to $H'$ is
\begin{equation*}
  \Biprof(\cat{K},\cat{L})(\cat{L}_S(1,F_S) \hcmp H, \cat{L}_S(1,F_S)
  \hcmp H') 
  \eqv
  \Biprof(\cat{K},\cat{L})(H, S H')
\end{equation*}
Precomposition $T^*$ with $T$ is a monad on this 2-category, and its
algebras are the monad op-morphisms (thus equipment profunctors) $T
\to S$, or equivalently the right $T$-modules whose underlying
morphism is of the form $\cat{L}_S(1,F_S) \hcmp H$.  But as noted above,
the latter are precisely the right $T$-modules $\cat{L}_S(1,F_S) \hcmp
H T \to \cat{L}_S(1,F_S) \hcmp H$ that arise from squares of the form
\eqref{eq:25} above, with representable modules corresponding to
representable profunctors $H = \cat{L}(1,F)$.  So equipment morphisms
in the sense of def.~\ref{dfn:31} are equivalent to morphisms
\eqref{eq:25} in the sense of Wood.

Now suppose given an equipment morphism $(F,\tilde F) \colon T \to S$
of the form \eqref{eq:25}.  The functor $\tilde F$ gives rise to a
right $T$-module by composition with the canonical one, and this has a
transpose $T \to \cat{L}_S(\tilde F F_T, \tilde F F_T)$, which
naturality of transposition shows is the composite
\begin{equation*}
  \xymatrix{
    T \ar[r]^-\sim & \cat{K}_T(F_T, F_T) \ar[r] & \cat{L}_S(\tilde F
    F_T, \tilde F F_T)
  }
\end{equation*}
Here the right-hand morphism arises from the unit $1 \to
\cat{L}_S(\tilde F, \tilde F)$, which is the effect on hom-categories
of $\tilde F$.  The codomain of this is equivalent (as a monad,
because $\tilde F F_T$ and $F F_S$ are equivalent $T$-modules) to
$S(F,F)$, so that the Kleisli objects of the two are equivalent under
$\cat{K}_T$.  We thus get an equivalence of factorizations of $\bar F$
\begin{equation*}
  \xymatrix{
    \cat{K}_T \ar[r] \ar[d] &
    \cat{K}_{S(F,F)} \ar[dl]^{}="s"_{}="t"|{\sim} \ar[d] \\
    \cat{K}_{\cat{L}_S(\tilde F F_T, \tilde F F_T)} \ar[r] &
    \cat{L}_S
    \ar@{}[ul];"t"|{\eqv} \ar@{}[];"s"|{\eqv}
  }
\end{equation*}
where the functors out of $\cat{K}_T$ and the diagonal one are the
identity on objects and the others are fully faithful --- that on the
right acts as $F$ on objects and that on the bottom as $\tilde F$.
The two composites $\cat{K}_T \to \cat{L}_S$ are canonically
equivalent to $\tilde F$, because they give rise to equivalent
modules; the left-and-bottom factorization is essentially $F_T$
itself, but the top-and-right factorization takes the values of $F$ on
objects.  This shows that in an equipment morphism of the form
\eqref{eq:25}, we can always, up to canonical equivalence, take
$\tilde F$ to coincide strictly with $F$ on objects.

A morphism $(\cat{K},\cat{M}) \to (\cat{L},\cat{N})$ of framed
bicategories is defined
\cite[def.~6.5]{shulman08:_framed_bicat_and_monoid_fibrat} to be a
pseudo-functor between their underlying double categories, i.e.~a pair
of functors between their vertical and horizontal categories that
commute with the projections, together with invertible globular cells
witnessing functoriality.  This data immediately gives rise to a
morphism of equipments
\begin{equation*}
  \xymatrix{
    \cat{K} \ar[r] \ar[d] & \cat{M}_{\mathrm{gl}} \ar[d] \\
    \cat{L} \ar[r] & \cat{N}_{\mathrm{gl}}
  }
\end{equation*}
Conversely, given a morphism $(F,\tilde F) \colon (\cat{K},T) \to
(\cat{L},S)$ of equipments, its tight part $F$ is a functor between
the vertical parts of their corresponding double categories.  As noted
above, $\tilde F$ can be taken to coincide with $F$ on objects, and it
acts on a general cell $g_\bullet M \tc N f_\bullet$ in
$\mathrm{Sq}_{\cat{K}}(\cat{K}_T)$ to form
\begin{equation*}
  F(g)_\bullet \tilde F(M) \overset{\sim}{\to}
  \tilde F(g_\bullet) \tilde F(M) \overset{\sim}{\to} \tilde
  F(g_\bullet M) \to \tilde F(N f_\bullet) \overset{\sim}{\to} \tilde
  F(N) \tilde F(f_\bullet) \overset{\sim}{\to} F(N) F(f)_\bullet
\end{equation*}
Vertical functoriality follows from the naturality and associativity
of $\tilde F$'s compositor and the pseudo-naturality of the
equivalence $\tilde F F_T \eqv F_S F$, so that we get a morphism of
spans $(\cat{K},\mathrm{Sq}_{\cat{K}}(\cat{K}_T)) \to
(\cat{L},\mathrm{Sq}_{\cat{L}}(\cat{L}_S))$, and the 2-functoriality
of $\tilde F$ makes this into a morphism of double categories.  This
assignment is in fact strictly functorial, because of how the
comparison cells of a composite functor are defined.

When we come to define equipment 2-cells, however, we run into a
problem.  A transformation between double functors assigns a vertical
morphism of the target double category to each object of the source,
and a cell of the target to each horizontal morphism of the source
(subject to some axioms).  So for functors arising from equipment
morphisms $F,G \colon (\cat{K},T) \to (\cat{L},S)$, a double
transformation would send $k \in \cat{K}$ to $\alpha_k \colon Fk \to
Gk$ in $\cat{L}$, and $m \colon k \prof k'$ to some $\alpha_m \colon
(\alpha_{k'})_\bullet F m \tc G m (\alpha_k)_\bullet$, such that
horizontal identities are sent to identity cells, and the cell
assigned to a composite is the horizontal composite of the cells
assigned to the components.  This amounts precisely to an \emph{oplax}
transformation $\tilde F \tc \tilde G$, with tight components, but we
have no recipe for producing these from our abstract monad machinery.
By the above discussion, a Kleisli 2-cell in the sense of
section~\ref{sec:monads-modules} between monad op-morphisms would
amount to a morphism of right $T$-modules from $\cat{L}_S(1, F_S F)$
to $\cat{L}_S(1, F_S G)$, which corresponds to a \emph{pseudonatural}
transformation $\tilde F \tc \tilde G$, whose components are not
required to be tight.  A `free' Kleisli 2-cell would be one that fits
into a cylinder
\begin{equation*}
  \xymatrix{
    \cat{K}
    \ar@/_2ex/[d]^{}="s1" \ar@/^2ex/[d]_{}="t1"
    \ar@{}"s1";"t1"|{\Leftarrow}
    \ar[r]
    & \cat{K}_T
    \ar@/_2ex/[d]^{}="s2" \ar@/^2ex/[d]_{}="t2"
    \ar@{}"s2";"t2"|{\Leftarrow} \\
    \cat{L} \ar[r] & \cat{L}_S
  }
\end{equation*}
thus amounting to a transformation $\tilde F \to \tilde G$ with tight
components, that is however still required to be pseudonatural.

We will take the easy way out by noting that, just as 2-categories and
pseudofunctors form a strict 2-category, so do equipments and their
morphisms: in a 3-fold composite of squares \eqref{eq:25}, the 1-cells
are uniquely determined, and the equivalence filling the composite
square is determined up to unique isomorphism, as noted above.  This
then forms a category that we will call $\bEqt_1$, and the preceding
discussion supplies a functor $\mathrm{Sq} \colon \bEqt_1 \to
\bicat{FrBicat}_1$ into the category (underlying the strict
2-category) of framed bicategories and their morphisms that we have
already seen to be essentially surjective on objects.  To show that it
is surjective on morphisms, let $G \colon \mathrm{Sq}_{\cat{K}}(\bM)
\to \mathrm{Sq}_{\cat{L}}(\bN)$ be a double functor.  To show that $G$
is equal to the `conjugate' of $\mathrm{Sq}(G_{\mathrm{gl}})$ by the
relevant isomorphisms \eqref{eq:27} it suffices to show that $G$
commutes with the process of passing between squares of the following
form:
\begin{equation*}
  \xymatrix{
    \bullet \ar[r]|+^m \ar[d]_f \ar@{}[dr]|{\displaystyle\Downarrow} & \bullet
    \ar[d]^g \\
    \bullet \ar[r]|+_n & \bullet
  }
  \qquad
  \xymatrix{
    \bullet \ar[r]|+^m \ar[d]_1 \ar@{}[drr]|{\displaystyle\Downarrow} &
    \bullet \ar[r]^{g_\bullet} & \bullet \ar[d]^1 \\
    \bullet \ar[r]_{f_\bullet} & \bullet \ar[r]|+_n & \bullet
  }
\end{equation*}
But this process is given by pasting with the universal squares
\eqref{eq:26}, and these are preserved by double functors by
\cite[prop.~6.4]{shulman08:_framed_bicat_and_monoid_fibrat}.  As for
injectivity, given two functors $F,G \colon (\cat{K},\bM) \to
(\cat{L},\bN)$ in $\bEqt_1$, the definition of $\mathrm{Sq}(F)$ and
$\mathrm{Sq}(G)$ uses all of the structure of the two, namely their
action on objects, morphisms and 2-cells and their functoriality
constraints, and if they differ in any of these then so will their
images under $\mathrm{Sq}(-)$.  So this functor is an equivalence.  We
can now simply define an equipment 2-cell to be a double
transformation between the appropriate double functors.

\section{Equipments and fibrations}
\label{sec:equipm-fibr}

In this section we define cartesian equipments
(section~\ref{sec:cartesian-equipments}), and show that the 2-category
of them receives an `equipment-of-matrices' functor from that of
regular fibrations that moreover is fully faithful
(section~\ref{sec:comp-with-regul}).  In that same section we give
axioms on a cartesian equipment the ensures it is in the image of this
functor.  Section~\ref{sec:kleisli-compl-compr} then shows that a
regular fibration has comprehension in the sense of def.~\ref{dfn:22}
if and only if the corresponding equipment has \emph{tabulations} in a
sense that we will define, and that this holds if and only if every
co-monad in the equipment has an Eilenberg--Moore object.

Henceforth we will use the term `equipment' to mean a map-equipment
over a locally discrete 2-category.  It follows from the discussion of
the previous section that equipments and equipment functors and
transformations between them form a 2-category $\bEqt$ equivalent to
Shulman's strict 2-category of framed bicategories
\cite[prop.~6.8]{shulman08:_framed_bicat_and_monoid_fibrat}.  It
carries a monoidal structure given by the cartesian product of
equipments.

\begin{dfn}
  \label{dfn:23}
  A cell in a map-equipment $\bK \to \bM$
  \begin{equation*}
    \xymatrix{
      X \ar[r]|+ \ar[d] \ar@{}[dr]|{\displaystyle\Downarrow} & X' \ar[d] \\
      Y \ar[r]|+ & Y'
    }
  \end{equation*}
  is \emph{exact} if its mate in $\bM$ is invertible.
\end{dfn}

A commuting square of tight maps gives rise to two distinct vertically
invertible cells, so that there are two senses in which it can be said
to be exact.

\subsection{Cartesian equipments}
\label{sec:cartesian-equipments}

\begin{dfn}
  Let $\bM$ be a cartesian monoidal 2-category.  A \emph{cartesian
    object} in $\bM$ is a pseudomonoid $M$ in $\bM$ whose
  multiplication map is right adjoint to the diagonal at $M$, and
  whose unit map is right adjoint to the map to the terminal object:
  \begin{equation*}
    \xymatrix{M \times M \ar@/_1em/[r]_{\otimes} \ar@{}[r]|{\bot} & M
      \ar@/_1em/[l]_{d_M} }
    \qquad
    \xymatrix{M \ar@/_1em/[r]_{!} \ar@{}[r]|{\top} & \mathbf{1}
      \ar@/_1em/[l]_{I} }
  \end{equation*}
  Clearly, such objects form a full sub-2-category of
  $\operatorname{PsMon}(\bM)$ (def.~\ref{dfn:29}).

  A \emph{cartesian equipment} is a cartesian object in $\bEqt$.  The
  full sub-2-category of the latter on the former will be called
  $\bicat{CartEqt}$.
\end{dfn}

To give a right adjoint $G \colon (\cat{L},\cat{N}) \to
(\cat{K},\cat{M})$ to a morphism $F$ of framed bicategories is, by
(the dual of)
\cite[prop.~8.4]{shulman08:_framed_bicat_and_monoid_fibrat}, to give
the following:
\begin{enumerate}
\item for each object $\ell \in \cat{L}$, a universal morphism $e_\ell
  \colon F G \ell \to \ell$;
\item for each horizontal morphism $n \colon \ell \prof \ell'$ a cell
  $\epsilon_n$
  \begin{equation*}
    \xymatrix{
      F G \ell \ar[r]|+^{F G n} \ar[d]_{e_k} \ar@{}[dr]|{\displaystyle\Downarrow} &
      F G \ell' \ar[d]^{e_{\ell'}} \\
      \ell \ar[r]|+_n & \ell'
    }
  \end{equation*}
  such that any cell as on the left below factors as on the right:
  \begin{equation*}
    \vcenter{
      \xymatrix{
        F k \ar[r]|+^{F m} \ar[d]_{f} \ar@{}[dr]|{\displaystyle\Downarrow} &
        F k' \ar[d]^{f'} \\
        \ell \ar[r]|+_n & \ell'
      }
    }
    \qquad = \qquad
    \vcenter{
      \xymatrix{
        F k \ar[r]|+^{F m} \ar[d]_{F g} \ar@{}[dr]|{\displaystyle\Downarrow} &
        F k' \ar[d]^{F g'} \\
        F G \ell \ar[r]|+^{F G n} \ar[d]_{e_k} \ar@{}[dr]|{\displaystyle\Downarrow}
        &
        F G \ell' \ar[d]^{e_{\ell'}} \\
        \ell \ar[r]|+_n & \ell' 
      }
    }
  \end{equation*}
  where the upper square on the right is the $F$-image of a unique
  square in $(\cat{K},\cat{M})$;
\item \label{item:3} such that horizontal composites of universal
  cells and identity cells on universal vertical morphisms are again
  universal.
\end{enumerate}
The first condition supplies a right adjoint $G_0$ for $F_0$, the
second a right adjoint $G_1$ for $F_1$ that makes $(G_0,G_1)$ a
morphism of spans, and the third ensures that this is a double
functor.  This shows that to give a framed bicategory
$(\cat{K},\cat{M})$ the structure of a cartesian object is to give
finite products in both $\cat{K}$ and $\cat{M}$ that are preserved by
the projections and by the composition and identity functors.

We can transfer these conditions across the equivalence $\bEqt \eqv
\bicat{FrBicat}$.  A family of universal vertical morphisms $e_\ell$
in $\mathrm{Sq}_{\cat{L}}(\bN)$ gives, trivially, a family of
universal tight morphisms in the equipment $(\cat{L},\bN)$.  Suppose
given also a family of universal cells $\epsilon_n$: for each functor
$\tilde F_{kk'} \colon \bM(k,k') \to \bN(Fk,Fk')$ define a map on
objects in the opposite direction by $G'_{kk'} n = h_{k'}^\bullet (Gn)
{h_{k}}_\bullet$, where the $h_k$ are the units of the adjunction $F_0
\dashv G_0$.  Then the mate of the universal $\eta_m$, unit of the
adjunction $F_1 \dashv G_1$, is a globular cell $m \tc G' \tilde F m$,
and it is universal from $m$ to $G'$ by the universal property of
$\eta_m$.  Hence $F$ has \emph{local} right adjoints $\tilde F_{kk'}
\dashv G'_{kk'}$ for each pair $k,k'$ of objects, the $G'$ being
functorial because $G$ is so locally.

Conversely, suppose given a morphism $(F,\tilde F) \colon
(\cat{K},\bM) \to (\cat{L},\bN)$ of equipments, together with objects
$G \ell$ and universal 1-cells $e_\ell \colon F G \ell \to \ell$ in
$\cat{L}$, and universal 2-cells $\epsilon'_n \colon \tilde F_{\ell
  \ell'} G'_{\ell \ell'} n \tc n$ supplying adjunctions $\tilde
F_{kk'} \dashv G'_{kk'}$ as above.  Clearly the universal 1-cells are
also universal vertical morphisms in $\mathrm{Sq}_{\cat{L}}(\bN)$.
Define $\tilde G n = G'_{Gk Gk'}(e_{\ell'}^\bullet n
{e_{\ell}}_\bullet)$.  Then the transpose of
$\epsilon'_{e_{\ell'}^\bullet n {e_{\ell}}_\bullet}$, as on the left
below
\begin{equation*}
  \xymatrix{
    F G \ell \ar[r]|+^{\tilde F \tilde G n} \ar[d]_{e_\ell}
    \ar@{}[dr]|{\displaystyle\Downarrow}
    & F G \ell' \ar[d]^{e_{\ell'}} \\
    \ell \ar[r]|+_n & \ell'
  }
  \qquad \qquad
  \xymatrix{
    F k \ar[r]|+^{\tilde F m} \ar[d]_{f} \ar@{}[dr]|{\displaystyle\Downarrow} &
    F k' \ar[d]^{f'} \\
    \ell \ar[r]|+_n & \ell'
  }
\end{equation*}
is universal in $\mathrm{Sq}_{\cat{L}}(\bN)$: in a cell as on the
right, the vertical maps are $f = e_\ell \hcmp Gg$ for a unique $g
\colon k \to G \ell$, and similarly for $f'$, and so the cell
corresponds to a unique 2-cell in $\bN$
\begin{equation*}
 \tilde F(g'_\bullet \hcmp m \hcmp g_\bullet)
 \overset{\sim}{\longrightarrow}
 Fg'_\bullet \hcmp \tilde Fm \hcmp Fg^\bullet \longrightarrow
 e_{\ell'}^\bullet n {e_{\ell}}_\bullet
\end{equation*}
which we can chase through the following bijections:
\begin{prooftree}
  \AXC{$\xymatrix{\tilde F(g'_\bullet \hcmp m \hcmp g_\bullet) \ar[r] &
      e_{\ell'}^\bullet n {e_{\ell}}_\bullet}$}
  \UIC{$\xymatrix{g'_\bullet \hcmp m \hcmp g_\bullet \ar[r] & \tilde G
      n}$}
  \UIC{$
    \xymatrix{
      k \ar[r]|+^{m} \ar[d]_{g}
      \ar@{}[dr]|{\displaystyle\Downarrow} &
      k' \ar[d]^{g'} \\
      G \ell \ar[r]|+_{\tilde G n} & G \ell'
    }$}
\end{prooftree}
So we get a family of universal cells in $\mathrm{Sq}_{\cat{L}}(\bN)$,
as required.  In short, a right adjoint to an equipment functor $(F,
\tilde F)$ is a right adjoint $G$ of $F$ together with local right
adjoints $\tilde F_{kk'} \dashv G'_{kk'}$ for $\tilde F$, such that
the resulting $\tilde G$ is functorial
(cf.~\cite[theorem~3.19]{carboni98:change_base_geom_ii}).

The above gives rise straightforwardly to a description of cartesian
objects in $\bEqt$ that extends the characterizations of cartesian
bicategories in \cite[thm.~1.6]{carboni87:_cartes_bicat_i} and
\cite[def.~4.1,~prop.~4.2]{carboni08:_cartes_ii}.

\begin{prp}
  \label{prp:18}
  To give cartesian structure on an object $\cat{K} \to \bM$ of
  $\bEqt$ is to give either:
  \begin{enumerate}
  \item \label{item:4} equipment morphisms as follows, that give
    $\cat{K}$ finite products:
    \begin{equation*}
      \xymatrix{
        \mathbf{1} \ar[d] \ar[r]^{\mathbf{1}} & \cat{K} \ar[d] & \cat{K}
        \times \cat{K} \ar[l]_{\times} \ar[d] \\
        \mathbf{1} \ar[r]_{\mathbf{1}} & \bM & \bM \times \bM
        \ar[l]^{\otimes}
      }
    \end{equation*}
    such that $m \wedge m' = d^\bullet(m \otimes m')d_\bullet$ and
    $\top = e^\bullet 1_{\mathbf{1}} e_\bullet$ provide finite
    products in the hom-categories of $\bM$; or
  \item finite products in $\cat{K}$ and the hom-categories of $\bM$,
    such that $(1_{\mathbf{1}})_\bullet \iso
    \top_{\mathbf{1},\mathbf{1}}$ and
    \begin{align*}
      m \otimes m' & = p^\bullet m p_\bullet \wedge q^\bullet m'
      q_\bullet
    \end{align*}
    is functorial, or equivalently such that the universal cells in
    $\mathrm{Sq}_{\cat{K}}(\bM)$ derived as above from the local
    products in $\bM$ satisfy the coherence conditions of
    (\ref{item:3}) above.
  \end{enumerate}
  Of course, the corresponding framed bicategory then has products as
  described above; in particular, $\mathrm{Sq}(\otimes)$ is the
  cartesian product on $\mathrm{Sq}_{\cat{K}}(\bM)$.
\end{prp}

By comparing this with prop.~\ref{prp:15} and remark~\ref{rem:10} we
can immediately conclude the following.
\begin{prp}
  \label{prp:10}
  A chordate cartesian equipment is the same thing as a cartesian
  bicategory.
\end{prp}
This will enable us to make use, in what follows, of results from
e.g.~\cite{walters08:_froben} and \cite{lack10:_bicat} that are proved
there for cartesian bicategories, as long as we are careful to
distinguish between maps and tight maps.

Clearly, a monoid morphism between cartesian framed bicategories is a
double functor whose components are (strong) monoidal, and likewise a
monoid 2-cell is a pair of monoidal transformations that underlie a
double transformation, meaning that their components commute with the
monoidal constraints of the functors involved.

An equipment morphism that preserves the products $\times$ in the base
category of its domain preserves the global tensor $\otimes$ if and
only if it preserves the local products $\wedge$, and this is what a
monoid morphism in $\bEqt$ between cartesian objects amounts to.  On
the other hand, even though a `local' description of monoid 2-cells
could probably be derived, we will have no use for one, and will stick
with the `global' description provided by the framed-bicategory
perspective.

\subsection{Comparison with regular fibrations}
\label{sec:comp-with-regul}

The following proposition constructs from a regular fibration $\E$ a
cartesian equipment of `matrices' in $\E$.  The next result then shows
that this construction is part of a fully-faithful functor $\Matr(-)
\colon \bicat{RegFib} \to \bicat{CartEqt}$.  The image
$\bicat{RegEqt}$ of this functor is then characterized, so that we get
an equivalence between $\bicat{RegFib}$ and $\bicat{RegEqt}$.

\begin{prp}
  \label{prp:20}
  If $\E \to \B$ is a regular fibration, then there is a cartesian
  equipment $\Matr(\E)$, with objects and tight maps the objects and
  morphisms of $\B$, and hom categories $\Matr(\E)(X,Y)$ the fibres
  $\E(X \times Y)$.
\end{prp}
\begin{proof}
  A regular fibration $\E$ over $\B$ is, in particular, a symmetric
  monoidal bifibration with cartesian base, and
  \cite[theorem~14.2]{shulman08:_framed_bicat_and_monoid_fibrat} shows
  that the $\Matr(-)$ construction applied to one such yields a
  symmetric monoidal framed bicategory, which is a symmetric monoid in
  $\bEqt$.\footnote{That result is stated for fibrations satisfying
    Beck--Chevalley for either all pullback squares or for a
    restricted class as long as the fibration satisfies Frobenius, but
    inspection of the proof shows that the conditions are only applied
    for product-absolute pullback squares.}  For reference, here are
  the essential details: composites are given by `relational
  composition':
  \begin{equation*}
    S \hcmp R
    = {p_Y}_! (p_Z^* R \cap p_X^* S)
    = \den{\exists \upsilon . R(x,\upsilon) \wedge S(\upsilon,z)}
  \end{equation*}
  and identities by `identity relations':
  \begin{equation*}
    1_X = d_! \top_X = \den{\exists \xi . (x,x')
      = (\xi, \xi)} = \den{x = x'}
  \end{equation*}
  while a morphism $f \colon X \to X'$ of $\B$ becomes a tight map
  like so:
  \begin{align*}
    f_\bullet & = (f \times 1)^* 1_{X'} \\
    f^\bullet & = (1 \times f)^* 1_{X'}
  \end{align*}
  For a cell in $\Matr(\E)$ of the form
  \begin{equation*}
    \xymatrix{
      X \ar[r]|+^R \ar[d]_f \ar@{}[dr]|{\displaystyle\Downarrow} & Y \ar[d]^g \\
      X' \ar[r]|+_S & Y'
    }
  \end{equation*}
  we have by \cite[(10,
  11)]{shulman08:_framed_bicat_and_monoid_fibrat} that
  \begin{align*}
    g_\bullet R & \iso (1 \times g)_! R \\
    S f_\bullet & \iso (f \times 1)^* S
  \end{align*}
  and hence a morphism $g_\bullet R \to S f_\bullet$ like that above
  has mates
  \begin{prooftree}
    \AXC{$\xymatrix{R \ar[r] & (f \times g)^* S \iso g^\bullet S
        f_\bullet }$}
    \UIC{$\xymatrix{g_\bullet R \iso (1 \times g)_! R \ar[r] & (f
        \times 1)^* S \iso S f_\bullet}$}
    \UIC{$\xymatrix{g_\bullet R f^\bullet \iso (f \times g)_! R \ar[r]
      & S}$}
  \end{prooftree}

  The category $\cat{B}$ we already know to be cartesian, and the
  tensor product on $\Matr(\E)$ is defined as
  \begin{equation*}
    R \otimes R' = p_{X'Y'}^* R \cap p_{XY}^* R'
    \colon X \times X' \rel Y \times Y'
  \end{equation*}
  which implies that 
  \begin{equation*}
    d^\bullet (R \otimes R') d_\bullet \iso (d_X \times
    d_Y)^*(p_{X'Y'}^* R \cap p_{XY}^* R') \iso R \cap R'
  \end{equation*}
  and so we have local binary products.  Because $d_{\mathbf{1}}$ is
  an isomorphism, the associated pull--push adjunction is an
  equivalence, and so
  \begin{equation*}
    e^\bullet (1_{\mathbf{1}}) e_\bullet
    \iso (e_X \times e_Y)^*d_!\top_{\mathbf{1}}
    \iso (e_X \times e_Y)^*\top_{\mathbf{1} \times \mathbf{1}}
    \iso \top_{X \times Y}
  \end{equation*}
  Hence in fact $\Matr(\E)$ has local finite products given by the
  formulas in (\ref{item:4}) of prop.~\ref{prp:18}, and so it is a
  cartesian equipment.
\end{proof}

Proving the following is a simple matter of unwinding definitions.
\begin{cor}
  A commuting square in the base of a regular fibration $\E$ satisfies
  the Beck--Chevalley condition (def.~\ref{dfn:4}) if and only if it
  is exact (def.~\ref{dfn:23}) in both senses in $\Matr(\E)$.
\end{cor}

\begin{thm}
  The $\Matr(-)$ construction of prop.~\ref{prp:20} extends to a fully
  faithful functor $\bicat{RegFib} \to \bicat{CartEqt}$.
\end{thm}
\begin{proof}
  Functoriality follows from
  \cite[theorem~14.9]{shulman08:_framed_bicat_and_monoid_fibrat}.  In
  brief, a morphism $(F, \phi) \colon (\B,\E) \to (\B',\E')$ of
  regular fibrations preserves all of the structure used to define
  $\Matr(\E)$, so that $F$, together with the functors
  \begin{equation*}
    \tilde F_{XY} \colon
    \xymatrix{ \E(X \times Y) \ar[r]^-{\phi_{X \times Y}} & \E'(F(X
      \times Y)) \ar[r]^\sim & \E'(F X \times F Y)}
  \end{equation*}
  which sends $m \colon X \prof Y$ to the pushforward of $\phi m$
  along the relevant coherence map of $F$, gives rise to a functor
  $\tilde F \colon \Matr(\E) \to \Matr(\E')$, with local monoidal
  constraints obtained similarly by pushforward.  The functoriality
  isomorphisms of $\Matr(-)$ come from the pseudonaturality cells of
  the $\phi, \gamma$, etc., and the former are coherent because the
  latter are.

  Conversely, if $(F, \tilde F) \colon \Matr(\E) \to \Matr(\E')$ is a
  map of cartesian equipments, then $F \colon \cat{B} \to \cat{B'}$
  preserves products, and the required transformation $\phi$ from $\E
  \tc F^* \E'$ is given by
  \begin{equation*}
    \phi_X =
    \xymatrix{
      \E X \iso \Matr(\E)(X,\mathbf{1})
      \ar[r]^-{\tilde F_{X,\mathbf{1}}}
      & \Matr(\E')(F X, F \mathbf{1}) \iso \E'(F X)
    }
  \end{equation*}
  which is natural in $X$ because its components are, and it preserves
  products because $\tilde F$ does so locally.  Naturality of $\tilde
  F_{X(Y\times \mathbf{1})}$ with respect to $Y \times \mathbf{1} \iso
  Y$ shows that $\Matr(-)$ applied to this gives an equipment morphism
  isomorphic to $(F, \tilde F)$, so that $\Matr(-)$ is essentially
  surjective on morphisms.

  Thinking of a transformation $(F,\phi) \tc (G,\gamma)$ as a
  `cylinder' $(\alpha, \bar \alpha)$, where $\alpha \colon F \tc G$
  and $\bar \alpha \colon \bar F \tc \bar G$, $\bar F$ and $\bar G$
  being the functors between total categories corresponding to $\phi$
  and $\gamma$, we get for each $m \colon X \prof Y$ a morphism $\bar
  \alpha_m \colon \bar F m \to \bar G m$ over $\alpha_{X \times Y}$,
  and composing this with the (op)cartesian morphisms indicated we get
  \begin{equation*}
    \tilde F m \iso \bar F m \overset{\bar\alpha_m}{\longrightarrow}
    \bar G m \iso \tilde G m
  \end{equation*}
  over
  \begin{equation*}
    FX \times FY \iso F(X \times Y) \overset{\alpha_{X\times
        Y}}{\longrightarrow} G(X \times Y) \iso GX \times GY
  \end{equation*}
  The latter is $\alpha_X \times \alpha_Y$, so the former corresponds
  to a unique cell
  \begin{equation*}
    \xymatrix{
      F X \ar[r]|+^{\tilde F m} \ar[d]_{\alpha_X}
      \ar@{}|{\displaystyle \Downarrow \mathrlap{\tilde \alpha_m}}[dr]
      &
      F Y \ar[d]^{\alpha_Y} \\
      G X \ar[r]|+_{\tilde G m} & G Y
    }
  \end{equation*}
  This assignment is natural in $m \in \E(X \times Y)$ because $\bar
  \alpha$ is.  It also respects horizontal composition and identities,
  because $\bar \alpha$ is a monoidal transformation, and hence
  commutes with the monoidal constraints of $\bar F$ and $\bar G$,
  therefore with those of $\tilde F$ and $\tilde G$.  It respects
  cartesian products in the category of cells for the same reason.
  The map $\bar\alpha \mapsto \tilde\alpha$ is itself clearly
  functorial.

  Conversely, let $(\beta, \tilde \beta) \colon \Matr(F,\bar F) \tc
  \Matr(G,\bar G)$ be a monoidal equipment-transformation.  An object
  $n$ over $X$ in the domain $\E$ of $\bar F$ and $\bar G$ gives a
  morphism $r^{-1}_!n \colon X \prof \mathbf{1}$ in $\Matr(\E)$, where
  $r \colon X \times \mathbf{1} \iso X$, and this in turn gives
  $\tilde \beta_{r^{-1}_! n} \colon \tilde F(r^{-1}_! n) \to \tilde
  G(r^{-1}_!  n)$ over $\beta_X \times \beta_{\mathbf{1}}$.  By the
  above, this is the composite with the evident isomorphisms of some
  morphism $\bar F(r^{-1}_! n) \to \bar G(r^{-1}_!  n)$ over $\beta_{X
    \times \mathbf{1}}$, and composing this with the $\bar F$- and
  $\bar G$-images of the isomorphism $n \iso r^{-1}_! n$ gives a
  morphism $\bar \beta_n \colon \bar F n \to \bar G n$ over $\beta_X$,
  which is natural in $n$ and monoidal because $\tilde \beta$ is.  It
  is easy to see then that $\Matr(\bar \beta) = \tilde \beta$, by
  cancelling inverses and using naturality, and the same in the other
  direction.  So $\Matr(-)$ is locally fully faithful, hence locally
  an equivalence, hence fully faithful as a 2-functor.
\end{proof}

\begin{dfn}[{cf.~def.~\ref{dfn:11}}]
  An object $A$ in a cartesian equipment is \emph{separable}
  \cite[def.~3.2]{lack10:_bicat} if the pullback square that expresses
  the monicity of $d_A$:
  \begin{equation}
    \label{eq:6}
    \xymatrix{
      A \ar[r] \ar[d] & A \ar[d]^{d_A} \\
      A \ar[r]_{d_A} & A \times A
    }
  \end{equation}
  is exact.  The object $A$ is \emph{Frobenius} if the coassociativity
  square
  \begin{equation}
    \label{eq:13}
    \xymatrix{
      A \ar[r]^{d_A} \ar[d]_{d_A} &
      A \times A \ar[d]^{A \times d_A} \\
      A \times A \ar[r]_{d_A \times A} & A^3
    }
  \end{equation}
  is exact, in both senses.  (In fact,
  \cite[lemma~3.2]{walters08:_froben} shows that for this particular
  square either exactness condition implies the other.)
\end{dfn}

The separability and Frobenius conditions hold in an equipment of the
form $\Matr(\E)$: they follow from the Beck--Chevalley conditions
(def.~\ref{dfn:11} type (A), and remark~\ref{rem:7}).

If $\cat{B} \to \bB$ is a cartesian equipment, then
\begin{equation*}
  \Pred(\bB) = \bB((-)_\bullet, \mathbf{1}) \colon \cat{B}^\op \to \Cat
\end{equation*}
is a bifibration, because if $f$ is a tight map then the pullback
functor $f^* = \Pred(\bB)(f_\bullet)$ has a left adjoint $f_! =
\Pred(\bB)(f^\bullet)$.  It clearly also has fibred finite products.
The Beck--Chevalley condition for squares $d f = (f \otimes 1)\langle
1, f \rangle$ of type (A) in def.~\ref{dfn:11} requires invertibility
of
\begin{gather*}
  \xymatrix{
    f_\bullet d^\bullet (1 \times f)^\bullet \ar[r] &
    d^\bullet (f \times 1)_\bullet
  } \\
  \xymatrix{
    (1 \times f)_\bullet d_\bullet f^\bullet \ar[r] &
    (f \times 1)^\bullet d_\bullet
  }
\end{gather*}
these being the mates of the isomorphism exhibiting $f$ as a
$d$-homomorphism (compare the inequalities \ref{eq:15}, \ref{eq:16}).
They are clearly dual.  Frobenius reciprocity requires invertibility
of
\begin{equation*}
  \xymatrix{
    (R \wedge S f_\bullet) f^\bullet \ar[r] &
    R f^\bullet \wedge S f_\bullet f^\bullet \ar[r] &
    R f^\bullet \wedge S
  }
\end{equation*}
but this is the whiskering of the second Beck--Chevalley morphism
above by $d^\bullet (R \otimes S)$, and hence the former is invertible
if the latter is.  The condition for type-(B) squares is precisely
separability \eqref{eq:6}.  The condition for squares of type (C) is a
special case of the functoriality of $\otimes$, as is that for type
(D), and side-by-side pastings always preserve the condition
(cf.~\cite[p.~512]{seely83:_hyper_natur_deduc_and_beck_condit}).

So if every object in $\bB$ is separable, then the only thing keeping
$\Pred(\bB)$ from being a regular fibration is the type-(A)
Beck--Chevalley condition.  Unfortunately, despite a strong suspicion
that the Frobenius condition implies it, I have been unable to find a
proof (note that the converse implication holds by
remark~\ref{rem:7}).  So we must assume it as an axiom, in the most
general cases, in order to get a regular fibration out of the
$\Pred(-)$ construction.  Therefore we define a \emph{regular
  equipment} to be a cartesian equipment satisfying these two
conditions.  There is thus a 2-category $\bicat{RegEqt}$ of regular
equipments, whose 1-cells are monoidal equipment functors and whose
2-cells are equipment transformations.  It is easy to see that if
$\Pred(\bB)$ exists then $\Matr(\Pred(\bB)) \eqv \bB$, so that
$\bicat{RegEqt}$ is equivalent to $\bicat{RegFib}$.

Note, however, that the troublesome condition does follow from the
Frobenius condition in the locally ordered context: by
\cite[prop.~3.6]{walters08:_froben} each object in a regular equipment
is self-dual, giving an identity-on-objects contravariant involution
$(-)^\circ$.  In the locally ordered case the results of
\cite[theorem~2.4]{carboni87:_cartes_bicat_i} follow (modulo the
caveats above regarding the difference between maps and tight maps),
showing that the dual of $(df)_\bullet = \big((f\times
f)d\big)_\bullet$ is an equality of precisely the type required, by
naturality of duality and by the fact that the dual of $d^\bullet$ is
$d_\bullet$.  This isomorphism does still exist in the
non-locally-ordered setting, but there seems to be no good reason why
it should be the inverse of the Beck--Chevalley morphism.

\subsection{Tabulation and comprehension}
\label{sec:kleisli-compl-compr}

In this section we examine notions of tabulation for morphisms in a
regular equipment, and compare them to comprehension in the
corresponding regular fibration.

\begin{dfn}
  Let $R \colon X \rel Y$ be a morphism in an equipment.  A
  \emph{tabulation} of $R$ is an object $\{R\}$ together with a
  universal cell
  \begin{equation*}
    \xymatrix{
      \{ R \} \ar[d]_i \ar[r]|+^1 & \{ R \} \ar[d]^j \\
      X \ar[r]|+_R & Y \ar@{}[ul]|{\displaystyle\Downarrow}
    }
  \end{equation*}
  that is, such that a 2-cell $g_\bullet \to R f_\bullet$ is given by
  composing the above cell with (the identity cell on) a unique tight
  $Z \to \{ R \}$.

  An equipment \emph{has tabulation} if every morphism has such a
  tabulation, and we say that these tabulations are \emph{full} if the
  mate $j_\bullet i^\bullet \to R$ of the universal cell is invertible
  for each $R$.
\end{dfn}

The following result explains why we use the same notation for
tabulations as for comprehension (def.~\ref{dfn:22}).

\begin{prp}
  A regular fibration $\E$ has (full) comprehension if and only if
  $\Matr(\E)$ has (full) tabulation.
\end{prp}
\begin{proof}
  The following sequence of bijections shows that the extension
  $\{R\}$ of an object $R$ of $\E(X \times Y)$ is also the tabulation
  of $R$ considered as a morphism $X \rel Y$ in $\Matr(\E)$:
  \begin{prooftree}
    \AXC{$\xymatrix{ Z \ar[rr] \ar[dr]_{(f,g)} & & \{ R \} \ar[dl]
        \\ & X \times Y}$}
    \UIC{$\xymatrix{ (f,g)_! \top_Z \iso (f \times g)_! 1_Z
        \ar[r] & R}$}
    \UIC{$\xymatrix{Z \ar[r]|+^1 \ar[d]_{f} \ar@{}[dr]|{\displaystyle\Downarrow}
        & Z \ar[d]^{g} \\
        X \ar[r]|+_R & Y}$}
  \end{prooftree}
  Setting $Y = \mathbf{1}$, the same sequence read backwards shows
  that the extension of a predicate $P \in \E X$ is given by the
  tabulation of $P \colon X \rel \mathbf{1}$.  It also shows that the
  morphisms required to be invertible by the two forms of fullness are
  in fact the same.
\end{proof}

Proposition 3.4 of \cite{lack10:_bicat} shows that the separability
axiom for an object $X$ of a cartesian bicategory is equivalent to the
identity $1_X$'s being subterminal, so that an endomorphism $G$ of $X$
can admit at most one `copoint' $G \to 1_X$.  Their lemma~3.15 then
shows that if it does then there exists a unique 2-cell $\gamma \colon
G \to G^2$ making $G$ into a comonad.

We can go further: the proposition referred to also shows that
separability is equivalent to the statement that for any $\epsilon
\colon G \to 1_X$, the span $G \leftarrow G \to 1_X$ is a product.  In
that case there is a unique morphism $G \to G \wedge 1_X$, necessarily
given by $(1_G, \epsilon)$, which is invertible and natural in $G$
(i.e.~with respect to morphisms between copointed endomorphisms of
$X$).  A useful consequence of this is an isomorphism
\begin{equation}
  \label{eq:23}
  \begin{split}
    d^\bullet (G \otimes 1)
    & \iso d^\bullet ((G \wedge 1) \otimes 1) \\
    & \iso d^\bullet (d^\bullet \otimes 1) (G \otimes 1 \otimes 1) (d
    \otimes 1) \\
    & \iso d^\bullet (1 \otimes d^\bullet) (G \otimes 1 \otimes 1) (d
    \otimes 1) \\
    & \iso d^\bullet (G \otimes 1) (1 \otimes d^\bullet) (d \otimes 1)
    \\
    & \iso d^\bullet (G \otimes 1) d d^\bullet \\
    & \iso G d^\bullet
  \end{split}
\end{equation}
using coassociativity of $d$, Frobenius, and separability.  From this
in turn we see that, for example,
\begin{equation}
  \label{eq:20}
  G (M \wedge N) \iso GM \wedge N \iso M \wedge GN
\end{equation}
and in particular that $G \wedge G \iso G(1 \wedge G) \iso G G$.
Indeed, the 2-cell $\gamma \colon G \to G G$ given by
\cite[lemma~3.15]{lack10:_bicat} is equivalently the composite
\begin{equation*}
  \xymatrix{
    G \ar[r]^-{\delta} & G \wedge G \ar[r]_-{\sim} & G G
  }
\end{equation*}
of this isomorphism with the local diagonal at $G$, because the latter
begins with
\begin{equation*}
  \xymatrix{
    G \ar[r]^-\delta & G \wedge G \ar[r]^-{(1,\epsilon) \wedge G}_-{\sim}
    & G \wedge 1 \wedge G \ar[r] & \cdots
  }
\end{equation*}
and the former with
\begin{equation*}
  \xymatrix{
    G \ar[r]^-{\delta_3} & G \wedge G \wedge G \ar[r]^-{G \wedge
      \epsilon \wedge G} & G \wedge 1 \wedge G \ar[r] & \cdots
  }
\end{equation*}
and the composites of the first two morphisms in each are clearly
equal, while both continue identically.  The same lemma then shows
that the product projections $G G \to G$ are given by $\epsilon G$ and
$G \epsilon$.

There is a not-too-dissimilar result for comodules, which allows us to
describe categories of comodules in a neat and useful way.

\begin{prp}
  \label{prp:22}
  \label{prp:23}
  Let $G$ be a comonad on $X$ in the regular equipment $\Matr(\E)$,
  and let $M \colon Z \rel X$ be a morphism.  The isomorphism
  \begin{equation*}
    GM \iso G(M \wedge \top) \iso M \wedge G \top
  \end{equation*}
  whose second factor is \eqref{eq:20} is natural in $M$ and exhibits
  the endofunctor $M \mapsto M \wedge G \top$ as a comonad isomorphic
  as such to $G_* \colon M \mapsto GM$.  There is then an equivalence
  of categories
  \begin{equation*}
    \mathrm{LComod}(G,Z) \eqv \E (Z \times X) / G \top
  \end{equation*}
\end{prp}
\begin{proof}
  The isomorphism is natural in $M$ because its components are, so
  that $G_* \iso (- \wedge G \top)$, and hence the latter acquires the
  structure of a comonad.  For any object $A$ in a monoidal category,
  comonad structures on $(- \otimes A)$ are in bijection with comonoid
  structures on $A$.  But because $\E(Z \times X)$ is cartesian, there
  is one and only one comonoid structure on $G \top$, given by
  projection and diagonal, and hence the resulting comonad structure
  on $(- \wedge G \top)$ must be identical with that transferred from
  $G_*$.

  It follows that the categories of coalgebras of $(- \wedge G \top)$
  and of $G_*$ are equivalent.  But the category of coalgebras of the
  latter is the category of left $G$-comodules, while it is a
  generality that the category of coalgebras for a comonad of the form
  $(- \times A)$ on a cartesian category $\C$ is just $\C/A$.  Hence
  \begin{equation*}
    \mathrm{LComod}(G,Z) \eqv \cat{Coalg}(G_*) \eqv \E(Z \times X) / G
    \top
  \end{equation*}
\end{proof}

Now we want to compare the presence of tabulation for arbitrary
morphisms with the existence of Eilenberg--Moore objects for comonads.

\begin{dfn}
  An Eilenberg--Moore object for a comonad $G$ on an object $X$ in an
  equipment is, as in def.~\ref{dfn:21}, an object $X^G$ that
  represents (left) $G$-comodules, in that $\hom(Z,X^G) \eqv
  \mathrm{LComod}(G,Z)$, naturally in $Z$, but with this equivalence
  also holding for the restriction of each side to tight maps.  That
  is equivalently to say that the universal $X \prof X^G$ is a tight
  map, and that composition with it preserves and detects tight maps
  (cf.~\cite{garner13:_enric}).

  We will say that an object in an equipment is an EM object
  \emph{with respect to tight maps} if only the second part of this
  property holds.  In general, of course, such an object is not
  necessarily a genuine EM object.
\end{dfn}

\begin{lem}
  \label{lem:10}
  If $G$ is a comonad in a regular equipment, $f$ and $g$ are tight
  maps and $g_\bullet \to G f_\bullet$ is a 2-cell, then there is a
  unique isomorphism $f_\bullet \iso g_\bullet$, modulo which the
  given 2-cell makes $f$ a $G$-comodule.
\end{lem}
\begin{proof}
  Composing the given 2-cell with the counit $G \to 1$ gives a 2-cell
  $f_\bullet \to g_\bullet$, which by
  \cite[theorem~3.14]{lack10:_bicat} is unique and invertible.  Then
  2-cells $g_\bullet \to G f_\bullet$ are in natural canonical
  bijection with 2-cells $f_\bullet \to G f_\bullet$.  To say that a
  2-cell of the latter form is a $G$-comodule is the same as to say
  that its mate $f_\bullet f^\bullet \to G$ is a morphism of
  comonoids, but by [\textit{op.~cit.},~theorem~4.2(ii)] every such
  morphism is so.
\end{proof}

\begin{prp}
  \label{prp:19}
  A regular equipment has tabulation if and only if it has EM objects
  with respect to tight maps.  Moreover, the latter are genuine EM
  objects if and only if the corresponding tabulations are full.
\end{prp}
\begin{proof}
  By \cite[theorem~4.3]{lack10:_bicat}, a morphism $R \colon X \rel Y$
  gives rise to a comonad $G_R$
  \begin{equation*}
    \xymatrix{
      X \times Y \ar[r]^-{d \otimes Y} &
      X \times X \times Y \ar[r]^{X \otimes R \otimes Y} &
      X \times Y \times Y \ar[r]^-{X \otimes d^\bullet} &
      X \times Y 
    }
  \end{equation*}
  that comes equipped with a cell (given by projecting out $R$)
  \begin{equation*}
    \xymatrix{
      X \times Y \ar[r]|+^{G_R} \ar[d]_{p_1} \ar@{}[dr]|{\displaystyle\Downarrow
        \mathrlap{\mu}} & X \times Y \ar[d]^{p_2} \\
      X \ar[r]_R & Y
    }
  \end{equation*}
  whose mate $p_2 G_R p_1^\bullet \to R$ is invertible.  To give a
  tight comodule for $G_R$ is to give a morphism $(f,g) \colon Z \to X
  \times Y$ and a cell
  \begin{equation*}
    \xymatrix{
      Z \ar[r]|+^1 \ar[d]_{(f,f,g)} \ar@{}[dr]|{\displaystyle\Downarrow} & Z
      \ar[d]^{(f,g,g)} \\
      X \times X \times Y \ar[r]_{X \otimes R \otimes Y} & X \times Y
      \times Y
    }
  \end{equation*}
  Because the lower morphism is a three-fold product in the category
  of cells (prop.~\ref{prp:18}), to give such a cell is to give three
  cells of the following form:
  \begin{equation*}
    \xymatrix{
      Z \ar[r]|+^1 \ar[d]_f \ar@{}[dr]|{\displaystyle\Downarrow} & Z \ar[d]^f \\
      X \ar[r]|+_1 & X
    }
    \qquad
    \xymatrix{
      Z \ar[r]|+^1 \ar[d]_f \ar@{}[dr]|{\displaystyle\Downarrow} & Z \ar[d]^g \\
      X \ar[r]|+_R & Y
    }
    \qquad
    \xymatrix{
      Z \ar[r]|+^1 \ar[d]_g \ar@{}[dr]|{\displaystyle\Downarrow} & Z \ar[d]^g \\
      Y \ar[r]|+_1 & Y
    }
  \end{equation*}
  But the outer two of these are always unique (as in the proof of the
  lemma above), so that to give a square of the central form is
  precisely to give a tight $G_R$-comodule, and if one is universal
  then so is the other.  Such a universal cell into $R$ is then, as in
  \cite[theorem~4.7]{lack10:_bicat}, the composite of the EM comodule
  with $\mu$ above. But the mate of the structure 2-cell of a genuine
  EM object is always invertible, as is that of $\mu$, so that the
  mate of the cell exhibiting the tabulation of $R$ is invertible too
  (being the composite of cells with invertible mates) and hence in
  this case the tabulation is full.

  In the other direction, if $G$ is a comonad, then its tabulation
  comes together with a universal 2-cell $j_\bullet \to G i_\bullet$.
  By lemma~\ref{lem:10} above, $i$ is then a $G$-comodule, and the
  universal property of the tabulation is (again by the lemma)
  precisely the universal property of an EM object wrt tight maps.  By
  prop.~\ref{prp:23} above, a comodule $j_\bullet \to G j_\bullet$ is
  the same thing as a morphism $j_\bullet \to G\top$, which is the
  same thing as a cell
  \begin{equation*}
    \xymatrix{
      Z \ar[rr]|+^1 \ar[d]_{t} \ar@{}[drr]|{\displaystyle\Downarrow} & & Z
      \ar[d]^j \\
      \mathbf{1} \ar[r]|+_{t^\bullet} & X \ar[r]|+_G & X
    }
  \end{equation*}
  That means that if $i$ is the EM object with respect to tight maps
  of $G$, then $i$ tabulates $G t^\bullet \iso {p_X}_! G$, where $p_X
  \colon X \times X \to X$ is the projection onto the first component,
  and hence $\{ G \} \iso \{{p_X}_! G\}$.  If these tabulations are
  full, then, propositions~\ref{prp:23} and \ref{prp:24} give
  \begin{align*}
    \hom(Z, \{ G \}) & \eqv \E(Z \times \{G\}) \\
    & \eqv \E(Z \times \{ {p_X}_! G\}) \\
    & \eqv \E \{ p_Z^* {p_X}_! G\} \\
    & \eqv \E(Z \times X)/ p_Z^* {p_X}_! G \\
    & \eqv \E(Z \times X)/ G \top \\
    & \eqv \mathrm{LComod}(Z, G)
  \end{align*}
  in which line 3 follows from line 2 because the pullback of any $i_P
  \colon \{ P \} \to X$ along the projection $Z \times X \to X$ is
  $1_Z \times i_P$, and because extension, being a right adjoint,
  preserves pullbacks.  So full tabulations give rise to genuine EM
  objects.
\end{proof}

\section{The effective topos}
\label{sec:effective-topos}

A \emph{realizability topos} \cite{oosten08:_realiz} is, roughly, a
topos built out of some collection of computable objects (a
\emph{partial combinatory algebra}, or pca).  In particular, the
\emph{effective topos} $\Eff$ is constructed relative to the partial
recursive functions $\mathbb{N} \rightharpoonup \mathbb{N}$ on the
natural numbers.  The connection with realizability in the traditional
sense is that the canonical interpretation of higher-order Heyting
arithmetic in $\Eff$ yields precisely Kleene's realizability
interpretation \cite{kleene45,troelstra98:_realiz} of intuitionistic
arithmetic.

There are two ostensibly quite different ways to build a realizability
topos starting from a given pca, and here we will use the results of
the preceding sections to explain (to a certain extent, at least) how
they are related.

\subsection{The two constructions}
\label{sec:hylands-construction}

The first definition of the effective topos was Hyland's
\cite{hyland82:_effec_topos}.  We start by considering sets $S \sub
\mathbb{N}$ as non-standard truth values, so that the set $[X, \PN]$
of functions $X \to \PN$ is thought of as the set of non-standard
predicates, called $\PN$-sets, on the set $X$.  This set carries the
structure of a category: if $\phi,\psi \colon X \to \PN$, then a
morphism $\phi \to \psi$ is given by a partial recursive function
$\Phi$ that satisfies the following condition: for any $x \in X$ and
any $n \in \phi x$, $\Phi n$ is defined and $\Phi n \in \psi x$.
Moreover, as a category, the set $[X,\PN]$ has finite products: the
terminal object is given by $x \mapsto \mathbb{N}$, while the product
$\phi \times \psi$ is $x \mapsto \set{\tup{n,m}}{n \in \phi x \mbox{
    and } m \in \psi x}$.  (Recall that pairing $\tup{-,-} \colon
\mathbb{N} \times \mathbb{N} \to \mathbb{N}$ can be chosen to be a
total bijection).

The usual construction of the effective topos uses the \emph{preorder
  reflection} of this category structure on $[X, \PN]$: it is the
preorder where $\phi \leq \psi$ if there is a morphism $\phi \to
\psi$, that is, if there exists a partial recursive $\Phi$ that
satisfies the condition above.  This preorder $[X,\PN]$ is (equivalent
to) a Heyting algebra, but the finite meets given by the finite
products defined above are enough for our purposes.

\begin{dfn}[{\cite{hyland82:_effec_topos}}]
  The \emph{effective tripos} $\cat{ET}(-) \colon \Set^\op \to
  \cat{Heyt}$ is the functor that sends a set $X$ to the Heyting
  algebra $[X,\PN]$ and a function $f \colon X \to Y$ to the Heyting
  algebra homomorphism $f^* \colon [Y,\PN] \to [X,\PN]$ given by
  precomposition with $f$.
\end{dfn}

The total category $\int \cat{ET}$ is the category of $\PN$-sets.

The following is proved in \cite[p.~53]{oosten08:_realiz}.
\begin{prp}
  $\cat{ET}$ is an ordered regular fibration.
\end{prp}

\begin{dfn}
  \label{dfn:2}
  A \emph{partial equivalence relation} (per) on a type $X$ is a
  symmetric transitive relation on $X$; that is, a binary relation
  $R(x,x')$ of type $(X,X)$ such that
  \begin{align*}
    R(x_1,x_2) & \sq R(x_2,x_1) \\
    R(x_1,x_2), R(x_2,x_3) & \sq R(x_1,x_3)
  \end{align*}
  If $\E \to \B$ is an ordered regular fibration and $X \in \B$, then
  a per on $X$ is thus an object $r=\den{R}$ over $X \times X$
  satisfying $r \leq r^\circ = \sigma^* r$ (where $\sigma$ is the
  symmetry map of $X \times X$) and $r \cap r \leq r$.

  A morphism $(R,X) \to (S,Y)$ of pers is a relation $F$ of
  type $(X,Y)$ satisfying
  \begin{align*}
    F(x,y) & \sq R(x,x) \land S(y,y) & \text{(strict)} \\
    F(x,y) \land R(x,x') \land S(y,y') & \sq F(x',y') &
    \text{(relational)} \\
    F(x,y) \land F(x,y') & \sq S(y,y') & \text{(single-valued)} \\
    R(x,x) & \sq \exists y . F(x,y) & \text{(total)}
  \end{align*}
\end{dfn}

\begin{dfn}
  The \emph{effective topos} is the category of pers in the effective
  tripos $\cat{ET}$.
\end{dfn}

That this category is indeed a topos is proved in e.g.~\cite[theorem
2.2.1]{oosten08:_realiz}.

The second approach to constructing the effective topos is due to
Carboni, Freyd and \v{S}\v{c}edrov
\cite{carboni88:_categ_approac_to_realiz_and_polym_types}.

\begin{dfn}
  \label{dfn:1}
  An \emph{assembly} $A$ over a set $X$ is given by an
  $\mathbb{N}$-indexed sequence $\{A_i \sub X\}_{i\in\mathbb{N}}$.
  The sets $A_i$ are the \emph{caucuses} and the set
  $|A|=\bigcup_iA_i$ the \emph{carrier} of $A$.  A morphism $A \to B$
  of assemblies is given by a function $f \colon |A| \to |B|$ such
  that there exists a partial recursive $\Phi_f$ satisfying the
  following condition: for any $i$ and any $a \in A_i$, $\Phi_f i$ is
  defined and $f a \in B_{\Phi_f i}$.
\end{dfn}

\begin{rem}
  \label{rem:4}
  An assembly $\{A_i \sub X\}_i$ is essentially the same thing as a
  function $X \to \PN$, because $[X, \PN] \cong [\mathbb{N}, \pwrset
  X]$ as sets.  Moreover, the ordering on assemblies over $X$ induced
  by morphisms whose underlying function is (a restriction of) the
  identity on $X$ coincides with the ordering on $\PN$-sets defined
  above.

  However, morphisms `between the fibres' are not the same: an
  assembly morphism takes no account of elements that are not
  contained in any caucus.  In particular, assemblies with exactly the
  same caucuses (even ones over different sets) must be isomorphic in
  $\cat{Asm}$, but need not be so in (the total category of)
  $\cat{ET}$.

  It does, however, follow from this that every assembly $A$ over $X$,
  say, is isomorphic to an assembly over its carrier $|A|$, and this
  is clearly the same thing as a $\PN$-set $\phi$ such that each $\phi
  x$ is non-empty.  Taking that point of view, a morphism of
  assemblies is then precisely a morphism of $\PN$-sets, and so
  $\cat{Asm}$ is equivalent to a full subcategory of $\int \cat{ET}$.
\end{rem}

\begin{prp}[{\cite[Proposition~1]{carboni88:_categ_approac_to_realiz_and_polym_types}}]
  The category $\cat{Asm}$ of assemblies and assembly morphisms is
  regular.
\end{prp}

\begin{dfn}
  \label{dfn:3}
  The \emph{effective topos} $\cat{Eff}$ is the exact completion
  $\cat{Asm}_{\mathrm{ex/reg}}$ of $\cat{Asm}$.
\end{dfn}

By corollary~\ref{cor:2}, the exact completion of $\cat{Asm}$ is the
category of maps in the splitting (def.~\ref{dfn:17}) of the
equivalences in $\Rel(\cat{Asm})$.  An equivalence in an allegory is
the same thing as a monad $s$ that is symmetric (i.e.~$s^\circ = s$),
and a morphism of idempotents between two such is precisely a
(bi)module (def.~\ref{dfn:10}), because in this locally ordered
context a module $m \colon s \prof s'$ is indeed simply a morphism
such that $m s = m = s' m$.

\subsection{Relating the two}
\label{sec:relating-two}

\begin{dfn}
  We will denote by $\Rel(\cat{Asm})_{|\Set}$ the full sub-2-category
  of $\Rel(\cat{Asm})$ on the constant assemblies, which can be
  identified with the locally ordered 2-category whose objects are
  sets $X,Y,\ldots$ and in which a morphism $X \to Y$ is given by an
  assembly $\{A_i \sub A\}_i$ together with a jointly monic span of
  functions $X \leftarrow |A| \to Y$.  The ordering is induced in the
  obvious way by (necessarily unique) assembly morphisms.
\end{dfn}

\begin{rem}
  \label{rem:1}
  Because assemblies with the same caucuses are isomorphic in
  $\cat{Asm}$ and in $\Rel(\cat{Asm})_{|\Set}$ (remark~\ref{rem:4}),
  we may assume without loss of generality that a morphism in the
  latter from $X$ to $Y$ is given by an assembly $\{A_i \sub X \times
  Y\}_i$.
\end{rem}

\begin{lem}
  \label{lem:1}
  $\Rel(\cat{Asm})_{|\Set}$ is equivalent to the underlying 2-category
  of $\Matr(\cat{ET})$.
\end{lem}
\begin{proof}
  By definition, the two have the same objects.  Isomorphism on hom
  posets follows essentially from remark~\ref{rem:4}.  In detail, the
  equivalence sends $r \colon X \times Y \to \PN$ to the assembly
  $\bar r = \{\bar r_i \sub X \times Y\}_i$, where $(x,y) \in \bar
  r_i$ if $i \in r(x,y)$, together with the projections to $X$ and
  $Y$.  It follows from remark~\ref{rem:1} that this assignment is a
  bijection.

  If $\Phi$ tracks $r \leq s$ in $\cat{ET}(X \times Y)$, then it also
  tracks $\bar r \leq \bar s$ and conversely, because $\Phi i \in
  s(x,y)$ if and only if $(x,y) \in \bar s_{\Phi i}$.

  It is a simple exercise in set theory to show that this
  correspondence preserves identities and composites, and so we have a
  2-functor that is the identity on objects and locally an
  isomorphism, hence an equivalence.
\end{proof}

\begin{lem}[{\cite{carboni88:_categ_approac_to_realiz_and_polym_types}}]
  \label{lem:2}
  $\Rel(\cat{Asm})$ is equivalent to
  $(\Rel(\cat{Asm})_{|\Set})[\check{\mathrm{crf}}]$.
\end{lem}

By lemma~\ref{lem:1}, the functor $\Set \to \Rel(\cat{Asm})_{|\Set}$
is a regular equipment in the image of $\Matr(-)$, so that the results
of section~\ref{sec:kleisli-compl-compr} apply: a coreflexive morphism
is precisely a comonad, and the splitting of these is precisely the
category of comodules.  So by the last result $\Rel(\cat{Asm})$ is
$\mathrm{Mod}^\co(\Matr(\cat{ET}))$, and in particular a comonad in
$\Matr(\cat{ET})$ is an assembly.  One might then wonder whether
$\cat{Asm}$ itself could turn out to be the co-Eilenberg--Moore
completion $\mathrm{EM}^\co(\Matr(\cat{ET}))$, but it is not: a
coreflexive $h \colon X \to X$ is a $\PN$-set such that $h x x' \sub
\den{x = x'}$, and so $\hat h x = h x x$ is a $\PN$-set, or indeed an
assembly, over $X$.  A morphism $h \to g$ of coreflexives is a
function $f \colon X \to Y$ such that $g(x,x') \leq h(fx, fx')$, or
equivalently such that there exists a recursive $\Phi$ such that if $n
\in \hat h$ then $\Phi n \in \hat h f x$.  In other words,
$\mathrm{EM}^\co(\Matr(\cat{ET}))$ is precisely the category $\int
\cat{ET}$ of $\PN$-sets.  This shouldn't be too surprising, since we
are dealing with the co-Eilenberg--Moore completion of an equipment,
which we know corresponds to the comprehensive completion of a
fibration, and the base category of the latter is the total category
of the original fibration \cite[theorem~3.1]{maietti12:_unify}.

\begin{prp}
  \label{prp:6}
  Let $\E \to \B$ be an ordered regular fibration.  The category of
  pers (def.~\ref{dfn:2}) in $\E$ is equivalent to $\Map(\Matr(\E)
  [\check{\mathrm{sym}}])$, where $\mathrm{sym}$ is the class of
  symmetric idempotents in $\Matr(\E)$ (def.~\ref{dfn:35}).
\end{prp}
\begin{proof}
  (Cf.~\cite[corollary A3.3.13(ii) \textit{et
    seq}.]{johnstone02:_sketc_of_eleph}) It is obvious that a
  symmetric idempotent in $\Matr(\E)$ is the same thing as a per in
  $\E$.

  Suppose $f \colon r \prof s$ is a morphism of symmetric idempotents
  that has a right adjoint $f^\bullet$ (which is necessarily equal to
  $f^\circ$).

  The axioms (strict) and (relational) are equivalent to $f$'s being a
  morphism of idempotents, i.e.~its satisfying $fr = f$ and $sf = f$
  (and consequently $f = sfr$).  In one direction, we have that
  $f(x,y)$ is equivalent to $f(x',y) \land r(x,x')$, and by symmetry
  and transitivity $r(x,x')$ implies $r(x,x)$.  The same works for $s$
  and so (strict) follows.  The condition $f = sfr$ easily implies
  (relational).  Conversely, (strict) and (relational) together imply
  that the three conditions
  \begin{align*}
    sfr & = \exists \xi,\upsilon. f(\xi,\upsilon) \land r(x,\xi) \land
    s(y,\upsilon)
    \\
    fr  & = \exists \xi. f(\xi,y) \land r(x,\xi) \\
    sf & = \exists \upsilon. f(x,\upsilon) \land s(y,\upsilon)
  \end{align*}
  are equivalent.  If (strict) holds then $f(x,y)$ implies $r(x,x)$
  and so implies $\exists \xi.f(\xi,y) \land r(x,\xi)$, while
  (relational) yields that $sfr$ as above implies $f(x,y)$, so that $f
  = sfr = fr = sf$.

  The axioms (total) and (single-valued) correspond to the adjunction
  $f \dashv f^\circ$; that is, to $r \leq f^\circ \cmp f$ and $f \cmp
  f^\circ \leq s$.  The latter gives
  \begin{equation*}
    f \cmp f^\circ \leq s \quad \text{\Iff} \quad \exists
    \xi. f(\xi,y) \land f(\xi,y') \Rightarrow s(y,y')
  \end{equation*}
  which by adjointness of $\exists$ and weakening is equivalent to
  (single-valued).  Finally, we have
  \begin{equation*}
    r \leq f^\circ \cmp f \quad \text{\Iff} \quad r(x,x') \Rightarrow
    \exists \upsilon. f(x,\upsilon) \land f(x',\upsilon)
  \end{equation*}
  which yields (total) when $x=x'$.  Conversely, $r(x,x')$ yields
  $r(x,x)$ and $r(x',x')$, and from these we get $\exists \upsilon,
  \upsilon'.f(x,\upsilon) \land f(x',\upsilon')$; (relational) gives
  $f(x,y) \land f(x,y')$ from this and $r(x,x')$, (single-valued)
  gives $s(y,y')$ and (relational) again yields $f(x,y) \land
  f(x',y)$.
\end{proof}

So the two different constructions of $\cat{Eff}$ are linked as shown
by the `map' in figure~\ref{fig:efftopmap} (where $\bullet$ denotes a
category that we don't really care about): we may start with the
effective tripos $\cat{ET}$, move to the corresponding framed
bicategory of relations, take the co-Eilenberg--Moore and then the
Kleisli completions, functionally complete the result and pass back to
a regular fibration (which is allowed because everything is locally
ordered), and the effective topos will be the base category of the
result.  The construction that starts with the category of assemblies
merges with this one at the stage indicated, modulo the slight
mismatch noted above between $\cat{Asm}$ and
$\mathrm{EM}^\co(\Matr(\cat{ET}))$.

\begin{figure}[hbtp]
  \centering
  \def\labelstyle{\textstyle}
  \begin{equation*}
    \xymatrix@+.5em@M+1em{
      \cat{ET} \ar@{|->}[r]^-{\Matr(-)} \ar@{|->}[ddd]
      \ar@/_3pc/@{|->}[dddd]_{\txt{cat. of \\ pers}}
      & \Set \to \Matr(\cat{ET})
      \ar@{|->}[d]_{\operatorname{EM}^\co(-)}
      & \cat{Asm} \ar@{|->}[dl]^{\Rel(-)} \\
      & \bullet \to \Rel(\cat{Asm})
      \ar@{|->}[d]_{\operatorname{Kl}_{\mathrm{sym}}(-)} \\
      & \bullet \to \Rel(\cat{Eff})
      \ar@{|->}[d]^{\txt{functional\\completion}} \\
      \operatorname{Sub}(\cat{Eff}) \ar@{|->}[d]^{\txt{base cat.}}
      & \ar@{|->}[l]_-{\Pred(-)} \cat{Eff} \to \Rel(\cat{Eff}) \\
      \cat{Eff}
    }
  \end{equation*}
  \caption{\label{fig:efftopmap}Constructions of the effective topos}
\end{figure}
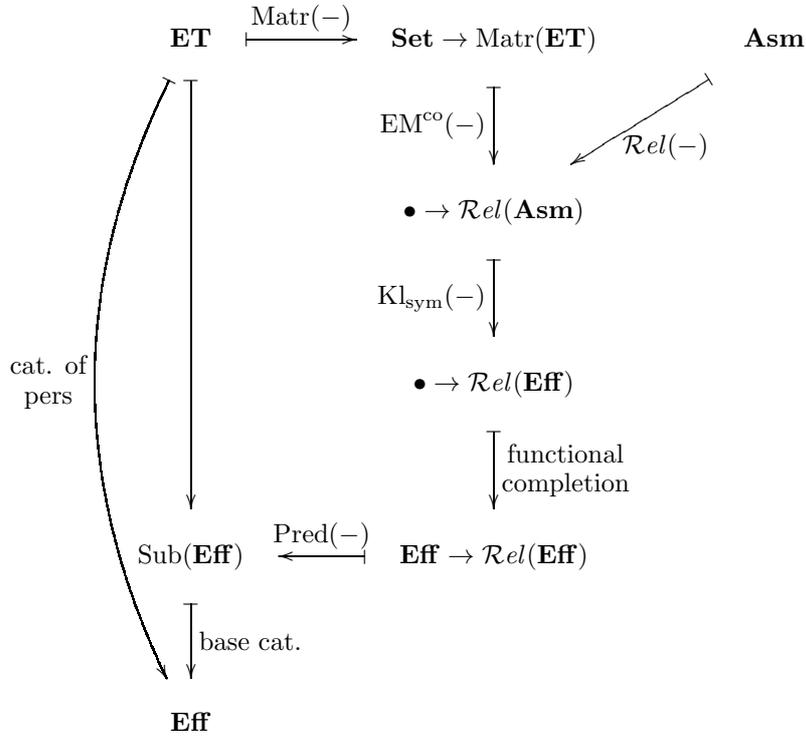


\chapter{Conclusions and future work}
\label{cha:concl-future-work}

\section{Recapitulation and comparison with existing work}
\label{sec:recapitulation}

Our main concrete results have been as follows:
\begin{itemize}
\item For regular theories, the fibrational and bicategorical
  semantics outlined in section \ref{sec:motivation} are equivalent
  once the latter is slightly augmented, and definitions and
  constructions may be translated back and forth across this
  equivalence in interesting and useful ways.  (And this is true not
  just in the (locally/fibrewise) ordered context of logic in the
  traditional sense, but also in the `proof-relevant' realm of type
  theory and category theory.)
\item In particular, one may translate the category-of-pers
  construction into the world of equipments, where it naturally
  decomposes into a sequence of constructions that each has a
  universal property, namely, the co-Eilenberg--Moore completion,
  followed by the Kleisli completion with respect to symmetric monads,
  followed by the functional completion.
\end{itemize}
Re-translating this back into the world of fibrations exhibits the
category of pers as the category of definable functions of the
effective completion of the comprehensive completion.\footnote{Note to
  arXiv version: More recently, Pasquali \cite{pasquali14:_remar} has
  derived this decomposition of the category-of-pers construction
  entirely within the world of fibrations.}  This decomposition
illuminates, to a certain extent, the relationship between the two
ostensibly quite different constructions of the effective topos,
showing that they `converge' sooner than one might expect.

It is also worth commenting on the techniques we have used to obtain
these results, and on the auxiliary results we have got along the way:
\begin{itemize}
\item Proposition~\ref{prp:16}, that a bicategory of relations is the
  same thing as a unitary tabular allegory, does not seem to have been
  published before, although it is hardly a surprising result.  It is
  what connects our work to the construction of the effective topos by
  taking the exact completion of $\cat{Asm}$, i.e.~by splitting
  idempotents in the allegory $\Rel(\cat{Asm})$.
\item Section~\ref{sec:2-extranaturality} defined bicategorical ends
  and coends in $\Cat$, showing that the former behaved exactly as one
  would expect.  Section~\ref{sec:coends-again} then showed how to
  compute coends as weighted and as conical colimits, the previous
  section having shown how to construct the latter in $\Cat$.  All of
  these results are new, as far as I can tell (but see the footnote
  p.~\pageref{fn:1}), although it would be useful to compare our
  construction of colimits with that of `2-filtered' ones given in
  \cite{dubuc06filt}.
\end{itemize}
These last results then meant that we could define $\Biprof$ as a full
sub-3-category of $\Bicat$, thereby avoiding a lot of calculation, but
also that we could treat its morphisms as category-valued functors
composed using coends in the usual way.

In the following sections we discuss in a little more detail how our
work on regular fibrations and regular equipments, and on equipments
in general, is related to some existing work.

\subsection{Comprehension and tabulation}
\label{sec:compr-tabul}

Suppose given a regular fibration $\E$ over $\B$ that has full
comprehension.  If equality in $\E$ is extensional
(def.~\ref{dfn:33}), then $\B$ has all pullbacks, which satisfy the
Beck--Chevalley condition in $\E$
(cf.~\cite[theorem~4.8]{lack10:_bicat}): for a cospan $(f,g) \colon X
\to Z \leftarrow Y$, put $P(f,g) = (f \times g)^*d_!\top_Z = \den{fx =
  gy}$.  Then $\{P(f,g)\}$ is the pullback of $f$ along $g$, by the
following bijections:
\begin{prooftree}
  \AXC{$\xymatrix{ W \ar[dr]_{(h,j)} \ar[rr] & & \{P(f,g)\} \ar[dl] \\ &
      X \times Y}$}
  \UIC{$\xymatrix{ \im (h,j) \ar[r] & P(f,g)}$}
  \UIC{$\xymatrix{ \im (fh,gj) \ar[r] & \im d }$}
  \UIC{$\xymatrix{ W \ar[dr]_{(fh,gj)} \ar[rr] & & \{ \im d \} \ar[dl]
      \\ & Z \times Z}$}
\end{prooftree}
By prop.~\ref{prp:21} extensionality means that each diagonal $d
\colon X \to X \times X$ is an injection, so that morphisms of the
last form are the same as factorizations of $(fh,gj)$ through $d$, of
which there is at most one, which exists precisely when $fh=gj$.  (The
Beck--Chevalley condition then follows from the fullness of
tabulations in $\Matr(\E)$.)  So $\B$ has finite limits, and hence $X
\mapsto \B/X$ is a regular fibration.  Note that this also means that
the type-(A) Beck--Chevalley condition holds in cartesian equipments
that satisfy the separability condition and that admit full
tabulations.

The adjunctions $\im \dashv \{-\}$ exhibit each fibre $\E X$ as a
reflective subcategory of $\B/X$.  If injections are closed under
composition, then this is equivalent \cite[2.12]{carboni97localiz} to
giving a factorization system on $\B$, whose right class consists of
the injections.  Then the image functors preserve pullbacks (they
always preserve pushforwards) if and only if this factorization system
is pullback-stable.  Consider the pullback square defined above: by
the Beck--Chevalley condition, we have $f^*g_! \iso (f^*g)_!
(g^*f)^*$, but $(g^*f)^*$ preserves the terminal object, so that
$f^*(\im g) \iso \im f^*g$.  Hence full comprehension implies that
image preserves pullbacks.  Therefore, from the definition of a
regular category from def.~\ref{dfn:14}, and the fact that in the
presence of full comprehension, orderedness of a fibration is
equivalent to every injection's being a monomorphism, we have the
following
(cf.~\cite[theorem~4.9.4]{jacobs99:_categ_logic_and_type_theor}).

\begin{prp}
  A regular fibration $\E$ over $\B$ is equivalent to
  $\operatorname{Sub}(\B)$ if and only if $\E$ is locally ordered and
  has full comprehension, such that every monomorphism in $\B$ is an
  injection.
\end{prp}

We also have the following result, an evident consequence of the
definition of injections.

\begin{prp}
  A regular fibration over $\B$ is equivalent to
  $\operatorname{Arr}(\B)$ if and only if it has full comprehension
  and every morphism in $\B$ is an injection.
\end{prp}

We clearly have
\begin{align*}
  \Matr(\operatorname{Sub}(\C)) & \eqv \C \to \Rel(\C) \\
  \Matr(\operatorname{Arr}(\C)) & \eqv \C \to \Span(\C)
\end{align*}
and so the previous two results translate to characterizations of
equipments of relations and of spans.  Saying that a span $(f,g)
\colon X \to Y \times Z$ \emph{tabulates itself} if the following
diagram is a tabulation
\begin{equation*}
  \xymatrix{
    X \ar[rr]|+^1 \ar[d]_f \ar@{}[drr]|{\Downarrow} & & X \ar[d]^g \\
    Y \ar[r]|+_{f^\bullet} & X \ar[r]|+_{g_\bullet} & Z
  }
\end{equation*}
then we have
\begin{prp}
  A regular equipment $\cat{B} \to \bB$ with co-Eilenberg--Moore
  objects is
  \begin{itemize}
  \item the equipment of relations in $\cat{B}$ if it is locally
    ordered and if every relation in $\cat{B}$ tabulates itself in
    $\bB$; or
  \item the equipment of spans in $\cat{B}$ if every span in $\cat{B}$
    tabulates itself in $\bB$.
  \end{itemize}
\end{prp}
This is clearly very similar to the characterization in
\cite[theorems~4,~7]{carboni84:_bicat} of 2-categories of relations
and of spans, although they do not require even cartesianness of the
underlying bicategory $\bB$ but instead that $\Map(\bB)$ be locally
discrete.

In \cite{lack10:_bicat} another characterization is given of
2-categories of spans: they are those that are cartesian and admit
Eilenberg--Moore objects for comonads, and in which every map is
comonadic.  Expressed in our language, taking a cartesian bicategory
to be a chordate cartesian equipment as in prop.~\ref{prp:10},
comonadicity means that if $f \colon X \to Y$ is a map then the cell
\begin{equation*}
  \xymatrix{
    X \ar[rr]|+^1 \ar[d]_f \ar@{}[drr]|{\Downarrow} & & X \ar[d]^f \\
    Y \ar[r]|+_{f^\bullet} & X \ar[r]|+_{f_\bullet} & Y
  }
\end{equation*}
exhibits $X$ as the EM object of $f_\bullet f^\bullet$.  By
prop.~\ref{prp:19} this is the same as saying that
\begin{equation*}
  \xymatrix{
    X \ar[rr]|+^1 \ar[d]_t \ar@{}[drr]|{\displaystyle\Downarrow} & & X
    \ar[d]^f \\
    \mathbf{1} \ar[r]|+_{t^\bullet} & Y \ar[r]|+_{f_\bullet f^\bullet}
    & Y
  }
\end{equation*}
is a tabulation.  But $f_\bullet f^\bullet t^\bullet$ is canonically
isomorphic to $f_\bullet t^\bullet$, which is precisely $\im f$, and
the tabulation above is the extension $\{ \im f \}$.  So to say that a
tight map is comonadic in $\bB$ is precisely to say that it is an
injection with respect to the fibration $\mathrm{Pred}(\bB)$.  Again,
this is very similar, though not identical, to the result above.  In
fact, the only real difference here is that \cite{lack10:_bicat}
derives the separability and Frobenius conditions from the
comonadicity axiom rather than postulating them.  It would be
interesting to see whether a similar but restricted condition would
suffice to axiomatize regular equipments, with or without the type-(A)
Beck--Chevalley condition.

\subsection{Equipments}
\label{sec:equipments}

Apart from the concrete results listed above, probably the most
significant thing we have done is to define the 3-category $\Biprof$
and to identify equipments, in a suitably general sense, as
pseudo-monads in it.  The important step here was the construction of
Kleisli objects in $\Biprof$ in theorem~\ref{thm:kleisli-2prof}.  We
have seen that these monads, when taken over locally discrete
2-categories, are essentially the same as both Wood's and Shulman's
notions when these are taken not to require right adjoints for tight
morphisms, and their relationship with the equipments of Carboni
et.~al.~is clear.

One kind of equipment that we have not compared with ours is Verity's
notion of a \emph{double bicategory} \cite{verity92:_enric}.  Such a
thing is given by a pair of 2-categories with the same objects,
thought of as `vertical' and `horizontal', and the 2-cells of these
act in a functorial way on a set of `squares', whose boundaries are
vertical and horizontal 1-cells, as in a double category.  There is
also an operation of horizontal composition on squares, that commutes
suitably with the action of vertical and horizontal 2-cells.  We won't
work out the details here, but one would expect our equipments, in the
most general sense, to be to double bicategories as Wood's equipments
(etc.) are to double categories, that is, to be (equivalent to) double
bicategories whose squares are uniquely determined by certain
horizontal 2-cells.  Indeed, a monad $T \colon \bK \prof \bK$ in
$\Biprof$ has an underlying vertical 2-category, namely $\bK$, a
horizontal one, namely $\bK_T$, and a set of squares given by the
objects of the 2-category of elements $\int T$.  The squares are acted
on by the 2-cells of $\bK$ and $\bK_T$ and inherit a horizontal
composition operation from the multiplication of $T$.  (This structure
is what \cite[def.~1.2.4]{verity92:_enric} would call
$\mathrm{Sq}(\bK_T, \bK, F_T)$.)  After defining double bicategories,
Verity goes on to use them to discuss morphisms of equipments more
general than the ones we have defined.  We will suggest some ways of
doing this in our context in section~\ref{sec:more-equipments} below.

Shulman \cite{shulman12:_exact} establishes an equivalence between a
kind of equipment called a \emph{framed allegory} and a generalized
notion of \emph{site}, a result that must surely be closely related to
ours.  In particular, Pavlovi\'c \cite{pavlovic96:_maps_ii} explains a
way of viewing a site as a regular fibration.  On the other hand,
Shulman's framed allegories are not required to have finite products,
whereas what we have done relies quite heavily on their presence.
Regardless, these ideas should be compared to and incorporated with
ours in future work.

There has been work done before on constructing Kleisli objects in
$\Bicat$.  In \cite{cheng04:_pseud}, it is shown that for a
pseudo-monad $T \colon \bK \to \bK$ in $\Bicat$, the objects of $\bK$
and the hom-categories $\bK(k,Tk')$ form a 2-category, and their
theorem~4.3 then says that this 2-category represents right
$T$-modules (which they call `cocones') in $\Bicat$.  Our
theorem~\ref{thm:kleisli-2prof} strictly generalizes this result,
because (as noted after corollary~\ref{cor:1}) the Kleisli object in
$\Biprof$ of a representable monad is also its Kleisli object as a
monad in $\Bicat$.  The idea behind the construction in
\textit{op.~cit.} is, in our language, that for an equipment $\bK \to
\bM$ that is the Kleisli object of a monad $T$ on $\bK$ in $\Bicat$,
together with another monad $S$ on $\bK$, to lift the latter to a
pseudo-monad on the equipment is precisely to give a distributive law
\cite{marmolejo99:_distr_laws_for_pseud} of $S$ over $T$.  Now the
point of our theorem~\ref{thm:kleisli-2prof} is that in fact every
equipment is the Kleisli object of some monad in $\Biprof$, just not
necessarily a representable one.  So the problem of lifting monads in
the above sense is contained in the problem of constructing
distributive laws in $\Biprof$.

\section{Future directions}
\label{sec:future-directions}

In this last section we give some ideas and prospects for future work
based on what we have already done.  

\subsection{More on equipments}
\label{sec:more-equipments}
\label{sec:lax-extensions}

We have seen, in section~\ref{sec:defin-equipm}, that while equipments
in the most general sense can be viewed as monads in $\Biprof$ and
equipment morphisms are then monad morphisms in a straightforward way,
this doesn't quite work for 2-cells.  The right notion of equipment
2-cell would reduce to a vertical transformation between double
categories in the case of a locally discrete base bicategory, but
neither monad 2-cells nor Kleisli 2-cells fit the bill.  As it turned
out, we were able to define an ordinary category of equipments, which,
together with double transformations between associated double
functors, was enough to get the results of
section~\ref{sec:comp-with-regul}.

Section \ref{sec:defin-equipm} showed that an equipment morphism from
$\bK \to \bK_T$ to $\bL \to \bL_S$ is given by a functor $F \colon \bK
\to \bL$ and a morphism $T \to S(F,F)$ of pseudo-monoids, the latter
giving the effect on hom-categories of a functor $\tilde F \colon
\bK_T \to \bL_S$.  Along the same lines, an equipment 2-cell will be a
transformation $\alpha \colon F \tc G$ together with a coherent
isomorphism between the two evident morphisms $T \to S(F,G)$, which
gives the morphism-components required to make $\alpha$ a
pseudonatural transformation $\tilde F \tc \tilde G$.  Relaxing the
condition that this latter morphism be invertible should then give the
right notion of 2-cell, and, just as an equipment 2-cell between
morphisms $\tilde F$ and $\tilde G$ from $T$ to $S$ is a (pseudo)
morphism of $T^*$-algebras, where $T^*$ is the precomposition-with-$T$
monad on the Kleisli 2-category of $S_*$, a lax equipment 2-cell ought
to be a lax algebra morphism.  This raises the possibility of using
the theory of lax morphism classifiers \cite{lack02:_codes} to reduce
the lax case to the pseudo case.

A similar possibility suggests itself when it comes to defining lax
morphisms of equipments.  One may define lax algebras for
pseudo-monads just as in def.~\ref{dfn:34}, except that the 2-cells
$\alpha$ and $\upsilon$ are not required to be invertible.  Then a lax
$T^*$-algebra structure on $F$ in the Kleisli 2-category of $S_*$
should correspond to a `lax monoid morphism' $\phi \colon T \to
S(F,F)$, i.e.~one that comes equipped with coherent morphisms $\phi
\vcmp \eta^T \to \eta^{S(F,F)}$ and $\mu^{S(F,F)} \vcmp (\phi \hcmp
\phi) \to \phi \vcmp \mu^T$, which will give a lax functor $\tilde F
\colon \bK_T \to \bL_S$ together with a transformation $\tilde F \hcmp
F_T \tc F_S \hcmp F$, not invertible in general.  Again, it may be
possible to use or generalize existing work on 2-dimensional monads to
reduce the lax case to the pseudo: for a strict 2-monad $T$ on a
strict 2-category, it is possible, under certain conditions, to
construct a new monad $T'$ such that lax $T$-algebras are precisely
strict $T'$-algebras.  That $\Biprof$ has well-behaved local colimits
suggests that it might be possible to do something similar in this
context, in order to construct lax morphism classifiers in some
3-category of equipments.  This would be useful for studying the kind
of change-of-base questions that Verity \cite{verity92:_enric}
considers, as well as liftings of monads to equipments, as discussed
in the last section, but where the lift is a lax monad
\cite{bunge74:_coher} rather than pseudo.  A classic example of the
latter situation is the fact that the ultrafilter monad on $\Set$
lifts to a lax monad on $\bicat{Rel}$, whose lax algebras are
topological spaces \cite{barr70:_relat}, but there are many other
contexts in which such constructions arise \cite{clementino04:_one}.

\subsection{`Variation through enrichment'}
\label{sec:vari-thro-enrichm}

For $\C$ any category, there is a 2-category $\bicat{S}(\C)$ given by
the full sub-2-category of $\Span([\C^\op, \Set])$ on the
representables.  Then categories enriched in $\bicat{S}(\C)$, in the
sense of e.g.~\cite[(5.5)]{benabou67:_introd} or
\cite{walters82:_sheav_cauch}, are very nearly the same as fibrations
over $\C$: by \cite{betti83:_variat} there is an equivalence between
the 2-categories of `Cauchy-complete' objects of each sort.  In fact,
it seems (although I do not know of a published proof) that if
$\bicat{S}(\C)$-enriched functors are defined using the
equipment/double-category structure of $\bicat{S}(\C)$, i.e.~if they
are required to give \emph{vertical} morphisms between extents, then
the equivalence includes even the non-complete categories and
fibrations.

There are two reasons for considering this as a framework in which to
interpret our results.  The first is that we would like to be able to
pare away at the structure of a regular fibration or equipment to see
what the axioms on one side of the equivalence correspond to on the
other side.  The problem, of course, is that nearly all of the
structure of a regular fibration is required in order even to define
the functor $\Matr(-)$ as in section~\ref{sec:comp-with-regul}.  So it
would make sense to try to recover this latter construction as a
special case of the more general one: that is, we know there is an
equivalence $\bicat{Fib}(\C) \to \xcat{\bicat{S}(\C)}$, and we might
ask what is required of a fibration over $\C$ in order for this
functor to factor through equipments over $\C$ in such a way as to
reproduce the results of section~\ref{sec:comp-with-regul}, if indeed
that is possible at all.  What is `regular structure' on a fibration
as an object of $\bicat{Fib}(\C)$?  What does that mean for the
corresponding $\bicat{S}(\C)$-category?  Does this structure on an
$\bicat{S}(\C)$-category make it `equivalent' to a regular equipment
in some way, in a way that coheres with the $\Matr(-)$ construction?

The second reason for moving to this level of generality is a
potential connection with more general forms of realizability.
Longley has recently proposed a notion of `computability structure'
\cite{longley13:_comput} that encompasses partial combinatory algebras
and is similar to the `basic combinatorial objects' of
Hofstra~\cite{hofstra06:_all}.  Each of these is clearly trying very
hard to be a category enriched in some sort of 2-category, and so one
might wonder whether they are examples of a still more general notion
of `coefficient object' for realizability that encompasses the two,
and whether the passage from a partial combinatory algebra to its
associated tripos can be seen in the context of the equivalence
between fibrations and categories enriched in certain 2-categories.
That is rather a vague idea, of course, but it holds out the
possibility of a structural account of realizability: an equivalence
of categories between very general collections (of whatever sort) of
computable objects and the fibrations they induce would allow a direct
comparison between structures borne by the one and by the other;
regular structure, tripos structure, and so on.  That, after all, was
the motivation behind the work reported here in the first place.


\bibliographystyle{alpha}
\bibliography{strings,xref,cat,cs,lam,log,me}

\end{document}